\documentclass{amsart}
\usepackage{bbm}

\usepackage{times}
\usepackage{amsfonts}
\usepackage{amsmath}
\usepackage{amscd}
\usepackage{amssymb}
\usepackage{amsxtra}
\usepackage{latexsym}
\usepackage{multirow}
\usepackage{amsthm} 
\input{xy}
\xyoption{all}

\theoremstyle{plain}

\newtheorem*{conjectuur*}{Conjecture}

\swapnumbers
\newtheorem{theorem}[subsection]{Theorem}
\newtheorem{corollary}[subsection]{Corollary}
\newtheorem{lemma}[subsection]{Lemma}
\newtheorem{proposition}[subsection]{Proposition}

\newtheorem{conjecture}[subsection]{Conjecture}

\theoremstyle{definition}
\newtheorem{definition}[subsection]{Definition}

\newtheorem{example}[subsection]{Example}

\theoremstyle{remark}

\newtheorem{remark}[subsection]{Remark}
\newtheorem{question}[subsection]{Question}

\swapnumbers

\newcommand{\emptyprop}{q}

\newcommand \acf{algebraically closed field}
\newcommand \after{\circ}
\newcommand \ann[2]{\operatorname{Ann}_{#1}(#2)}
\newcommand \binomial[2]{{\bigl( \begin{matrix} #1\cr#2\cr\end{matrix} \bigr)}}
\newcommand \ch{characteristic}

\newcommand \complet[1]{\widehat {#1}} 
\newcommand \DVR{discrete valuation ring}

\renewcommand \hom [3]{\operatorname{Hom}_{#1}(#2,#3)} 
\newcommand \homo{homomorphism}
\newcommand \id{\mathfrak a}
\renewcommand\iff{if and only if}
\newcommand\into{\hookrightarrow}
\newcommand \inv[1]{{#1^{-1}}}
\newcommand \inverse[2]{{#1^{-1}(#2)}}
\newcommand \iso{\cong}
\newcommand \loc{{\mathcal {O}}}

\newcommand \map[1]{\overset{#1}{\to}}
\newcommand \maxim{\mathfrak m}
\newcommand \nat{\mathbb N}
\newcommand \norm[1]{\left|#1\right|}
\newcommand \onto{\twoheadrightarrow}
\newcommand \op\operatorname
\newcommand \pol[2]{#1[#2]}
\newcommand \pow[2]{#1[[#2]]}

\newcommand \pr{\mathfrak p}

\newcommand \range [2]{#1,\dots,#2}
\newcommand \restrict [2]{\left.#1\right|_{{#2}}}
\newcommand \rij[2]{(#1_1,\dots,#1_{#2})}
\let\sub\subseteq

\newcommand \tensor{\otimes}

\newcommand \zet{\mathbb Z}

\newcommand{\commdiagram}[9][]{%
\begin{equation}
{\newcommand{\tmpprop}{#1q} 
\if\tmpprop\emptyprop \relax\else \label{#1}\fi}
\begin{aligned}%
\mbox{
\begin{picture}(130,90)%
\put(120,70){\vector( 0,-1){50}}%
\put(10,80){\vector( 1, 0){100}}%
\put(0,70){\vector( 0,-1){50}}%
\put(10,10){\vector( 1, 0){100}}%
\put(115,80){\makebox(0,0)[l]{$#4$}}%
\put(5,80){\makebox(0,0)[r]{$#2$}}%
\put(115,10){\makebox(0,0)[l]{$#9$}}%
\put(5,10){\makebox(0,0)[r]{$#7$}}%
\put(-3,50){\makebox(0,0)[r]{$#5$}}
\put(123,50){\makebox(0,0)[l]{$#6$}}
\put(60,3){\makebox(0,0)[c]{$#8$}}
\put(60,88){\makebox(0,0)[c]{$#3$}}
\end{picture}}
\end{aligned}
\end{equation}}

\newcommand\commtrianglefront[7][]{%
\begin{equation}
{\newcommand{\tmpprop}{#1q} 
\if\tmpprop\emptyprop \relax\else \label{#1}\fi}
\begin{aligned}%
\mbox{
\begin{picture}(120,80)%
\put(55,68){\vector(-1,-2){30}}
\put(65,68){\vector(1,-2){30}}
\put(30,5){\vector(1,0){60}}
\put(60,75){\makebox(0,0)[c]{$#2$}}
\put(25,5){\makebox(0,0)[r]{$#4$}}
\put(95,5){\makebox(0,0)[l]{$#6$}}
\put(60,0){\makebox(0,0)[c]{$#5$}}
\put(37,43){\makebox(0,0)[r]{$#3$}}
\put(83,43){\makebox(0,0)[l]{$#7$}}
\end{picture}}
\end{aligned}
\end{equation}}

\newcommand\commtriangleback[7][]{%
\begin{equation}
{\newcommand{\tmpprop}{#1q}
\if\tmpprop\emptyprop \relax\else \label{#1}\fi}
\begin{aligned}%
\mbox{
\begin{picture}(120,80)%
\put(55,70){\vector(-1,-2){30}}
\put(65,70){\vector(1,-2){30}}
\put(30,5){\vector(1,0){60}}
\put(60,75){\makebox(0,0)[c]{$#2$}}
\put(25,5){\makebox(0,0)[r]{$#6$}}
\put(95,5){\makebox(0,0)[l]{$#4$}}
\put(60,0){\makebox(0,0)[c]{$#5$}}
\put(37,43){\makebox(0,0)[r]{$#7$}}
\put(83,43){\makebox(0,0)[l]{$#3$}}
\end{picture}}
\end{aligned}
\end{equation}}

\newcommand\NYCCT{\email{hschoutens@citytech.cuny.edu}\address{Department of Mathematics\\
City University of New York\\
365 Fifth Avenue\\ 
New York, NY 10016 (USA)}
}

\usepackage[bookmarks,colorlinks,pagebackref]{hyperref} 

\usepackage{mathabx}

\newcommand \frob{\tuple F}
\newcommand\bc[2]{ #1\times\tuple 1_{#2}}
\newcommand\weak{explicit}
\newcommand\fat{\mathfrak z}
\newcommand\res[1]{\kappa(#1)}
\newcommand\bas{\lambda}
\newcommand\basy{\mu}
\newcommand\bassch{V}
\newcommand\basschy{W}
\newcommand\univpoint{\mathfrak u}
\newcommand\cone[1]{\mathfrak{C}(#1)}
\newcommand\motif[1]{\mathfrak{#1}}
\newcommand\categ[1]{\mathbbmss{#1}}

\newcommand\limfat[1]{\complet{\categ{Fat}}_{#1}}
\newcommand\imarc[3]{\arc{#1/#2}{#3}}
\newcommand\func[1]{#1^\circ}
\newcommand\graf[1]{\Gamma(#1)}

\newcommand\gsect[1]{H_0^{\text{geom}}(#1)}
\newcommand\gsheaf[1]{\loc^{\text{geom}}_{#1}}

\newcommand\sieves[1]{\categ{Sieve}_{#1}}
\newcommand\sievesweak[1]{\categ{S}\overline{\categ{ieve}}_{#1}}
\newcommand\funcalg[1]{\categ {Sch}_{#1}}
\newcommand\funccon[1]{\categ{Con}_{#1}}
\newcommand\funcsub[1]{\categ {subSch}_{#1}}

\newcommand\funcinfsplit[1]{\categ{Form}^{\text{split}}_{#1}}
\newcommand\funcalgflat[1]{\categ {Sch}^{\text{flat}}_{#1}}
\newcommand\funcsubsplit[1]{\categ {subSch}^{\text{split}}_{#1}}
\newcommand\funcinf[1]{\categ{Form}_{#1}}

\newcommand\fld{\kappa}

\newcommand\round[2]{\lceil\frac{#1}{#2}\rceil}

\newcommand\genarc[1]{\wideparen{#1}}

\newcommand \affine[2]{{\mathbb A_{#1}^{#2}}}
\newcommand \gr{Gro\-then\-dieck ring}

\newcommand\explicit{algebraic}
\newcommand\zariski{schemic}

\newcommand \var{x}
\newcommand \vary{y}
\newcommand\tuple[1]{\mathbf{#1}}
\newcommand\fim[1]{\texttt{Im}(#1)}

\newcommand \arc[2]{\nabla_{\!#1}#2} 
\newcommand \sym[1]{{\langle #1\rangle}}
\newcommand \class[1]{{[ #1]}}
\newcommand\latt[2]{\Lambda^{#2}(#1)}
\newcommand\grlatt[2]{\Lambda^{#2}_\bullet(#1)}

\newcommand \grot[1]{{\mathbf {Gr}(#1)}}
\newcommand \grotzero[1]{{\mathbf {Gr}_0(#1)}}
\newcommand \grotmot{\mathbf G}

\newcommand \grotclass[1]{{\mathbf {Gr}(\categ{Var}_{#1})}}

\newcommand\fatpoints[1]{\categ {Fat}_{#1}}
\newcommand\flatpoints[1]{\categ {Fat}^{\text{flat}}_{#1}}
\newcommand\splitpoints[1]{\categ {Fat}^{\text{split}}_{#1}}

\newcommand\motintspl[3]{\int #2\ d_{#3}#1}
\newcommand\motintsplrel[4]{\int_{#4} #2\ d_{#3}#1}
\newcommand \lef{{\mathbb L}}

\newcommand \jet[3]{J_{#1}^{#2}#3}

\newcommand \mor[3]{\op{Mor}_{#1}(#2,#3)}

\newcommand \jac[1]{\op{Jac}_{#1}}

\newcommand\igumot[3]{\op{Igu}^{(#2,#3)}_{ #1^{\op{mot}}}}
\newcommand\poinc[3]{\op{Poin}^{(#2,#3)}_{ #1^{\op{mot}}}}
\newcommand\mil[3]{\op{Mil}^{(#2,#3)}_{ #1^{\op{mot}}}}

\title {The yoga of schemic Grothendieck rings, a topos-theoretical approach}
\author{Hans Schoutens}
\date\today
 \NYCCT

\begin{document}


\begin{abstract}  
We propose a suitable substitute for the classical \gr\ of an \acf, in which any quasi-projective scheme  is represented, while maintaining its  non-reduced structure. This yields a more subtle invariant, called the \zariski\ \gr, in which we can formulate a form of integration resembling Kontsevich's motivic integration via arc schemes.  Whereas the original construction was via definability, we have  translated in this paper everything into a topos-theoretic framework.
\end{abstract}

\maketitle
\setcounter{tocdepth}{1}
\tableofcontents

\section{Introduction}

Kontsevich \cite{Kon}  has formulated a general integration technique on (smooth) schemes  over an \acf\ $\fld$, modeled on $p$-adic integration and called \emph{motivic integration}. This was then extended by Denef and Loeser \cite{DLIgu,DLArcs,DLDwork}   to achieve \emph{motivic rationality}, by which they mean the fact that the rationality of a certain generating series from geometry or number-theory, like the Igusa-zeta series, is ``motivated'' by the rationality of its motivic counterpart which specializes to the given classical series via some multiplicative function (counting function, Euler \ch, \dots). The two main ingredients of this construction are the \gr\  of varieties over $\fld$ (see below), in which the integration takes its values, and the \emph{arc space} $L(X)$ of a variety $X$, that is to say, the reduced Hilbert scheme classifying all  arcs $\op{Spec}\pow\fld \xi\to X$. Our aim is to extend this by replacing varieties by schemes, in such a way that  killing the  nilpotent structure  reverts to the old theory. 
The classical  \gr\ $\grotclass \fld$ of an \acf\ $\fld$ is designed to  encode both combinatorial and geometric properties of varieties. It is 
defined as the quotient of the free Abelian group on  varieties   over $\fld$, modulo the relations $\class X-\class {X'}$, if $X\iso X'$, and 
\begin{equation}\label{eq:scissvar}
\class X=\class{X- Y}+\class Y, 
\end{equation}
if $Y$ is a closed subvariety, for $Y,X,X'$ varieties (=reduced, separated schemes of finite type over $\fld$). We will refer to the former relations as \emph{isomorphism relations} and to the latter   as \emph{scissor relations}, in the sense that we    ``cut out $Y$ from $X$.'' 
Multiplication on $\grotclass \fld$ is then induced by the fiber product. In sum, the three main ingredients for building the \gr\ are:  scissor relations,   isomorphism relations , and   products. Only the former causes problems if one wants to extend the construction of the \gr\ from varieties too arbitrary  finitely generated schemes. Put bluntly,  we cannot cut a scheme in two, as there is no notion of a scheme-theoretic complement, and so we ask what new objects we should add to make this work. Let us call these new objects tentatively \emph{motives}, in the sense that their existence is motivated by a combinatorial  necessity.
There are now two approaches to construct these motives.

The first one, discussed at length in \cite{SchSchGr},   is based on definability, and was the original approach. The point of departure is the representation of a scheme by an equational (first-order) formula  modulo the theory of local Artin algebras. The logical operations of conjunction, disjunction, and negation are then used to form the required new objects:   motives arise as Boolean combinations of schemes. Scissor relations are now easily expressed in this formalism, whereas products are given by conjunction with respect to distinct variables, and isomorphism relations are   phrased in terms  of definable isomorphisms,  cumulating in  the construction of the \emph{\zariski\ \gr} over $\fld$.
To obtain geometrically more significant  motives, one is forced  not only to introduce quantifiers, but also   to resort to some infinitary logic via formularies, leading to larger \gr{s}, all of which still admit a natural \homo\ onto the classical \gr. To define the analogue of arc spaces, one obtains   arc formulae by interpreting  the theory of a local Artin algebra in that of  its residue field. The resulting arc operator is compatible with Boolean combinations and hence induces an endomorphism on the \gr{s}. This was the first striking success of the new theory: no such operation holds in the classical \gr. Other  main advantages of this approach are (i) the presence of negation, allowing one to ``cut up'' a scheme into motives, and (ii) the uniformity inherent in model-theory, allowing one to use parameters, and hence to work over an arbitrary base ring rather than just an \acf. The main disadvantage stems from non-functoriality, in particularly  when dealing with morphisms. Nonetheless, the theory has been applied in \cite{SchSchGr} with success to establish motivic rationality in certain cases, even the, thus far illusive,   positive \ch\ case. 

However, soon after writing down this version, I came to realize that it might beneficial to sacrifice definability for functoriality. In this approach, which forms the content of this paper and is essentially topos-theoretical, schemes are viewed as (contravariant) functors. Traditionally, one views them as functors, called \emph{representable functors}, from the category of $\fld$-algebras to the category of sets, but the power of the present approach comes from narrowing down the former category to that of \emph{fat points}, consisting only of one-point schemes over $\fld$. Thus, given a scheme $X$ of finite type over $\fld$ and a fat point $\fat$, we let $X(\fat)$ be the set of all \emph{$\fat$-rational points}, that is to say, morphisms $\fat\to X$. The functor $\fat\mapsto X(\fat)$ now  determines the scheme $X$ uniquely. Motives are then certain subfunctors of these representable functors, with morphisms between them given as natural transformations. Since these functors take values in the category of sets, all set-theoretic operations are available to us, such as union, intersection, and complement. However, complementation does not behave functorially, and so motives now only form a lattice, leading to the notion of a \emph{motivic site}: apart from a Grothendieck topology inherited from the Zariski topology on the schemes, we also require a categorical lattice structure in order to formulate scissor relations. Defining multiplication by means of fiber products, we thus get the \gr\ of a motivic site. Among the many tools from category theory and topos theory we can now resort to, adjunction takes a primary place: it allows us, for instance, to define, without much effort, arc schemes, which act again nicely on the corresponding \gr{s}. 

Let me  now briefly discuss  in more detail the content of the present paper. In \S\ref{s:sieves}, we discuss the functors that will play the role of motives. Borrowing terminology from topos theory, on the category of fat points, a subfunctor $\motif X$ of a representable functor given by a scheme $X$ is called a \emph{sieve} on $X$, and $X$ is called an \emph{ambient space} of $\motif X$. We may do this over an arbitrary Noetherian base scheme $\bassch$, provided it is also Jacobson. For the sake of this introduction, I will only treat the case of greatest interest to us, namely, when $\bassch$ is the spectrum of an \acf\ $\fld$. A morphism of sieves is in principle any natural transformation, but often such a morphism extends to a morphism of the ambient spaces, in which case we call it \emph{\explicit}. We turn this into a true topos in \S\ref{s:globsect}, by defining an admissible open of a sieve $\motif X$ to be its restriction to an open in its ambient space. We define a  \emph{global section} on  a sieve $\motif X$ to be any morphism into the affine line. We establish an acyclicity result for global sections, allowing us to  define the \emph{structure sheaf} $\loc_{\motif X}$ of $\motif X$. 

In the next four sections, \S\S\ref{s:motsite}--\ref{s:formgr}, we introduce the \gr\ of a motivic site, and discuss the three main cases. As already mentioned, a motivic site is for each scheme, a choice of lattice  of sieves on that scheme, called the \emph{motives} of the site. The associated \gr\ is then defined as the free Abelian group on motives in the site modulo the isomorphism relations and the scissor relations, where   the latter take the lattice form
\begin{equation}\label{eq:scisslat}
\class {\motif X}+\class{\motif Y}=\class{\motif X\cup \motif  Y}+\class{\motif X\cap\motif  Y}, 
\end{equation}
for any two motives $\motif X$ and $\motif Y$ on the same ambient space. The first motivic site of interest consists of the \emph{\zariski} motives, given on each scheme as the lattice of its closed subschemes (viewed as representable subfunctors). The resulting \gr\ is too coarse, as it is freely generated as an additive group by the classes of irreducible \zariski\ motives (Theorem~\ref{T:classinvmot}). A larger, more interesting site is given by the  sub-\zariski\ motives, where we call a sieve $\motif X$ on $X$ \emph{sub-\zariski}, if there is a morphism $\varphi\colon Y\to X$ such that at each fat point $\fat$, the set $\motif X(\fat)$ consists of all $\fat$-rational points $\fat\to X$ that factor through $\varphi$, that is to say, $\motif X(\fat)$ is the image of the induced map $\varphi(\fat)\colon Y(\fat)\to X(\fat)$. 
 Any locally closed subscheme is sub-\zariski, so that in the corresponding \gr, we may express the class of any separated scheme  in terms of classes of affine schemes. Moreover, any morphism of sieves with domain a sub-\zariski\ motif is \explicit\ (Theorem~\ref{T:globsectsubzar}), from which it follows that the sub-\zariski\ \gr\ admits a natural \homo\ into the classical \gr.
 
 Whereas in general the complement of a sieve is no longer a sieve (as functoriality fails), this does hold for any open subscheme. However, such a complement is in general no longer sub-\zariski, but only what we will call a \emph{formal motif}, that is to say, a sieve $\motif X$ that can be approximated by sub-\zariski\ submotives in the sense that for each fat point $\fat$, one of its sub-\zariski\ approximations has the same $\fat$-rational points as $\motif X$. In case of an open $U=X-Y$, the complement is represented by the formal completion $\complet X_Y$, whose approximations are    the \emph{jet spaces} $\jet YnX:=\op{Spec}(\loc_X/\mathcal I_Y^n)$. A morphism in the site of formal motives $\funcinf\fld$ is approximated by \explicit\ morphisms, and therefore, the ensuing \gr\ $\grot{\funcinf\fld}$ still admits a canonical \homo\ onto the classical \gr\ $\grotclass\fld$. 

In \S\ref{s:adj}, we discuss \emph{adjunction} of motivic sites over different base schemes. Formally, this consists of a pair of functors $\eta\colon \fatpoints{\basschy}\to\fatpoints \bassch$ and $\arc{\eta}{}\colon\sieves\bassch\to \sieves{\basschy}$ such that
$$
\motif X(\eta(\fat))=(\arc\eta {\motif X})(\fat)
$$
for any $\bassch$-sieve $\motif X$ and any $\basschy$-fat point $\fat$. We call such an adjunction (sub-)\zariski\ or formal, if $\arc\eta X$ is respectively (sub-)\zariski\ or formal, for any scheme $X$. Under this assumption, the adjunction induces a \homo\ from  the corresponding \gr\ over $V$ to the \gr\ over $W$. Whenever we have a morphism $f\colon \bassch\to\basschy$ of finite type, we obtain an adjunction given by the pair $(f_*,f^*=\arc{f_*}{})$, called \emph{augmentation}, where $f_*$ means restriction of scalars  and $f^*$ means extension of scalars via $f$. If $f$ is moreover finite and flat, then we can go the other way, called \emph{diminution}, via the pair $(f^*,\arc{f^*}{})$. Both adjunctions are \zariski, and they are related by the  projection formula 
$$
\arc{\tilde f^*}{}\arc{\tilde h_*}{}=\arc{h_*}{}\arc{f^*}{}
$$
where $h\colon  \tilde\bassch\to\bassch$ is of finite type, and $\tilde f$ and $\tilde h$ are the corresponding base changes of $f$ and $h$ with respect to the other (Theorem~\ref{T:projform}). Applying diminution to a  rational point $a\in X(\fld)$, we may define for each morphism $Y\to X$ and each motif $\motif Y$ on $Y$, the \emph{specialization} $\motif Y_a$. In this way, $\motif Y$ becomes a family of motives. Another interesting application of adjunction is via the action of Frobenius $\frob$ in \ch\ $p>0$. If we let $\frob(\fat)$ be the fat point with coordinate ring the $p$-th powers of elements in the coordinate ring of $\fat$, then this yields an adjunction $(\frob,\arc\frob{})$ which is only sub-\zariski: $\arc\frob Y$ is given as the image of the relative Frobenius on $Y$ (locally via some embedding in affine space; see Theorem~\ref{T:frobadj} for a precise formulation). 

In \S\ref{s:arc}, we apply the adjunction theory to the special case of the structure morphism $j\colon \fat\to \op{Spec}\fld$ of a fat point, and we define the \emph{arc functor} $\arc\fat{}$ as the composition of augmentation and diminution along $j$, that is to say, 
$$
\arc\fat{}:=\arc{j^*}{}\after\arc{j_*}{}.
$$
This corresponds in the special case that $\fat=\mathfrak l_n:=\op{Spec}\pol\fld\xi/(\xi^n)$ to the classical truncated arc space $L_n$ through the formula
$$
L_n(X)= (\arc{\mathfrak l_n}X)^{\text{red}}
$$
for any variety $X$. 
However, it should be noted that $\arc\fat{}$ does not commute with taking reductions, so that even if $X$ and $X'$ have the same underlying variety, they will in general have different arc schemes $\arc\fat X$ and $\arc\fat{X'}$, even possibly of different dimension (see \S\ref{s:motdim}). In fact, the dimension of the arc scheme of a fat point over itself is an intriguing invariant. 

By the general theory of adjunction, $\arc\fat{}$ is an endomorphism on each \gr. Arcs behave well over smooth varieties, as in the classical case (Theorem~\ref{T:fibmot}): the canonical morphism $\arc\fat X\to X$ is a locally trivial fibration over the non-singular locus of $X$, with general fiber some affine space. Hence, in the smooth case, $\class{\arc\fat X}=\class X\lef^{l(d-1)}$, where $d$ and $l$ are respectively   the dimension of $X$ and the length of (the coordinate ring of) $\fat$, and where $\lef:=\class{\affine\fld1}$ is the \emph{Lefschetz class}. 

Using  the formalism of adjunction,  we discuss some variants of arcs: \emph{deformed arcs} in \S\ref{s:defarc}, and \emph{extendable arcs} in \S\ref{s:extarc}. For the definition of the latter, we discuss in  \S\ref{s:limpt}  a compactification of the category of fat points,  the category of \emph{limit points}, given as  direct limits of fat points (e.g., the formal completion $\complet Y_P$).  Although we can extend the notion of arcs to any limit point, the corresponding arc scheme is no longer of finite type. 

In  \S\ref{s:motseries}, we discuss some of the motivic series that can now be defined using this formalism. As already mentioned, since they or their classical variants specialize to generating series that are known to be rational, we expect them to be already rational over the formal \gr, or rather, over its localization $\grot{\funcinf\fld}_\lef$; this is called \emph{motivic rationality}. Let us briefly describe each of these series here. 
\begin{description}
\item[Igusa]  with $\mathfrak j_n:=\jet PnY$   the $n$-th jet of a closed point $P$ on a scheme $Y$,  and $X$ any scheme of dimension $d$, we set 
$$
\igumot XYP(t):=\sum_{n=1}^\infty\lef^{-d \ell(\mathfrak j_n)}    
\class{ \arc{\mathfrak j_n}X} \ t^n
$$
\item[Hilbert] with $(Y,P)$ and $\mathfrak j_n$   as above, we set 
$$
\op{Hilb}^{\text{mot}}_{(Y,P)}:=\sum_{n=1}^\infty \class{\mathfrak j_n}\ t^n
$$
\item[Hilbert-Kunz] with $\fld$ of \ch\ $p>0$, with $Y\sub X$   a closed subscheme,   and with $Y^{[n]}_X$   the $n$-th Frobenius transform of $Y$ (defined by the $p^n$-th powers of the defining equations of $Y$), we set 
$$
\op{HK}^{\text{mot}}_Y(X):=\sum_{n=1}^\infty \class{Y^{[n]}_X}\ t^n
$$
\item[Milnor]  with  $\mathfrak y_n$ the $n$-th deformation of  a scheme $Y$ at a closed point $P$ given as $\op{Spec}\loc_{Y,P}/(\xi_1^n,\dots,\xi_e^n)$, where $\rij\xi e$ is a system of parameters in $\loc_{Y,P}$, with $X\sub\affine\fld{d+1}$   a hypersurface with equation $f=0$,  and with $X_{\mathfrak y_n}\sub\affine{\mathfrak y_n}{d+1} $   the 
deformed hypersurface with equation $f-(\xi_1\cdots\xi_e)^{n-1}=0$, we set
$$
\mil XYP(t):=\sum_{n=1}^\infty \lef^{-d\ell(\mathfrak y_n)}\class{\arc{j_{\mathfrak y_n}^*}{X_{\mathfrak y_n}}}\ t^n;
$$
\item[Poincare] with $\imarc{\complet Y_P}{\mathfrak j_n}X$   the submotif of $\arc{\mathfrak j_n}X$ consisting of all arcs that factor through the formal completion $\complet Y_P$, we set 
$$
\poinc XYP(t):=\sum_{n=1}^\infty \lef^{-d\ell(\mathfrak j_n)}
\class{\imarc{\complet Y}{\mathfrak j_n}X}t^n
$$
\item[Hasse-Weil] with $X^{(n)}$ the $n$-fold symmetric product of $X$, that is to say, the Hilbert scheme of effective zero cycles of degree $n$ on $X$, we set 
 $$
  \op{HW}^{\text{mot}}_X:=\sum_{n=0}^\infty\class{X^{(n)}}\ t^n.
  $$
\end{description}
Note that the specializations of the motivic Hilbert and Hilbert-Kunz series to the classical \gr\ are trivial since the underlying varieties are just points. Put differently, this kind of motivic rationality cannot even be phrased in the classical setup. However, specializing to the length of the motives, which is well-defined by Proposition~\ref{P:zerosch}, yields their classical counterparts.

The final   section, \S\ref{s:motint} is devoted to motivic integration. We only develop the finitistic theory, that is to say, over a fixed fat point, leaving the case of a limit point to a future paper. One of the great disadvantages of the categorical approach is that fibers are in general not functorial (after all, a fiber is the complement of the remaining fibers). We can overcome this by restricting to the category of \emph{split fat points}, in which the morphisms are now assumed to be split. 
Our motivic integration will take values in the localization $\grot{\funcinf\fld}_\lef$. A functor $s$, viewed on the category of split fat points,  from a formal motif $\motif X$ on $X$ to the constant sheaf with values in this localization $\grot{\funcinf\fld}_\lef$ is called a \emph{formal invariant} if all its fibers are formal motives, with only finitely many non-empty. We  then  define
$$
\motintspl Xs{\fat}:=\lef^{-dl}\sum_{g\in \grot{\funcinf\fld}_\lef}g\cdot 
\class{\arc\fat{(\inverse sg)}},
$$
where $d$ is the dimension of $X$ and $l$ the length of $\fat$. This motivic integral can be calculated locally (Theorem~\ref{T:locmotint}).

\subsection*{Notation and terminology}
Varieties are assumed to be reduced, but not necessarily irreducible. Given a scheme $X$, we let $X^{\text{red}}$ denote its \emph{underlying variety}  or \emph{reduction}. We often denote a morphism of affine schemes $\op{Spec}B\to \op{Spec}A$  by the same letter as the corresponding ring \homo\ $A\to B$, whenever this causes no confusion. By a \emph{germ} $(X,Y)$ we mean a scheme $X$ together with a closed subscheme $Y\sub X$. Most of the time $Y$ is an irreducible subvariety, that is to say, the closure of a point $y\in X$, and we simply write $(X,y)$ for this germ. If $Y$ is a closed point, we call the germ \emph{closed}.  The \emph{$n$-th jet} $\jet YnX$ of a   germ $(X,Y)$ is the closed  subscheme  defined by $\mathcal I_Y^n$, where $\mathcal I_Y$ is the ideal of definition of $Y$.\footnote{Note that many authors take instead the $n+1$-th power.} The \emph{formal completion} $\complet X_Y$ of the germ $(X,Y)$ is the locally ringed space obtained as the direct limit of the $\jet YnX$ (see \cite[II.\S9]{Hart}). For instance, if $Y=P$ is a closed point with maximal ideal $\maxim_P$, then the ring of global sections of $\complet X_P$ is the $\maxim_P$-adic completion $\complet\loc_{X,P}$ of $\loc_{X,P}$. 

We often denote the affine line $\affine\fld 1$ by $L$ and the origin by $O$. The formal completion of the germ $(L,O)$ is   denoted $\complet L$, and the basic open $L-O$, that is to say, the open subscheme obtained by removing the origin, is denoted $L_*$. The respective classes, in whichever \gr\ we consider, are denoted $\lef$, $\complet\lef$ and $\lef_*$. Whereas in the classical \gr, $\lef=\lef_*+1$ (and $\complet \lef$ is undefined),  in the formal \gr, we have  $\lef=\lef_*+\complet\lef$ (see   Proposition~\ref{P:formsieve}), which, after taking global sections, takes the suggestive  form
$$
``{\pol\fld \var}"=``{\pol\fld{\var,1/\var}}"+``{\pow\fld\var}" .
$$
 The $n$-th jet of  $(L,O)$ will be denoted $\mathfrak l_n:=\op{Spec}(\pol\fld\var/(\var^n))$.

\section{Sieves}\label{s:sieves}
Fix a Noetherian scheme $\bassch$, to be used as our base space, and which, for reasons that will become apparent soon, will also be assumed to be Jacobson. Most often, $\bassch$ is just the spectrum of  an \acf\ $\fld$. By a  \emph{$\bassch $-scheme} $X$, we mean a separated scheme $X$ together with a morphism of 
finite type $X\to \bassch$. We call $X$ a \emph{fat $\bassch $-point}, if $X\to 
\bassch $ is finite and $X$ has a unique point. In other words, $X=\op{Spec}R$, 
for some Artinian local ring $R$ which is finite as a  $\loc_\bassch$-module.  We call the length of $R$ the \emph{length} of the fat point and denote  it $\ell(\fat)$. We denote the subcategory of fat $\bassch$-points by $
\fatpoints \bassch$, and we will use letters $\fat, \mathfrak v, \dots$ to 
denote fat points. An important example of fat points are the jets of a closed germ $(X,P)$, that is to say, given   a $\bassch$-scheme $X$  and a closed point $P$ on $X$ with corresponding maximal ideal $\maxim_P$, we let $\jet PnX$, called the \emph{$n$-th jet of $X$ along $P$}, be  the fat point with coordinate ring $\loc_{X,P}/\maxim_P^n$. 

Let $X$ be a  $\bassch$-scheme and $\fat$ a fat point.  A morphism of $\bassch$-schemes $a\colon \fat\to X$ will be called a \emph{$\fat$-rational point} on $X$ 
(over $\bassch$). The set of all $\fat$-rational points on $X$ will be denoted by $X(\fat)$. The image of the unique point of $\fat$ under $a$ is a closed point on $X$, called the \emph{center} (or  \emph{origin}) of $a$. Indeed, since the composition $\fat\to X\to \bassch$ is the structure map whence finite, so is its first component $a\colon \fat\to X$. As finite morphisms are proper, closed points are mapped to closed points. Let $x$ be the center of a $\fat$-rational point $a$ on a $\bassch$-scheme $X$. We will denote the residue field of $x$ and $\fat$ respectively by $\res x$ and $\res\fat$. The $\fat$-rational point $a$ induces a \homo\ of residue fields $\res x\to \res\fat$. By the Nullstellensatz this a is a finite extension, and its degree will be called the \emph{degree} of $a$. In particular, if $\bassch$ is the spectrum of  an \acf\ $\fld$, then $\fld=\res x=\res\fat$, so that any $\fat$-rational point has degree zero.

 The category of contravariant functors (with natural transformations  as morphisms) from  $\fatpoints \bassch$ to     the 
category of sets will be called the category of \emph{presheaves} over $\bassch
$, following standard practice in topos theory, notwithstanding the confusion this causes with the usual notion from algebraic geometry. The \emph{product} $\motif X\times \motif Y$  (respectively,  the 
\emph{disjoint union} $\motif X\sqcup \motif Y$)  of two presheaves $
\motif X$ and $\motif Y$ is defined point-wise by the rule that a fat point 
  $\fat$ is mapped to the Cartesian product $\motif X(\fat)\times 
\motif Y(\fat)$ (respectively, to the disjoint union $\motif X(\fat)\sqcup 
\motif Y(\fat)$). Similarly, we say that $\motif X$ is a \emph{sub-presheaf} 
of $\motif Y$, symbolically $\motif X\sub \motif Y$, if for every fat point   $
\fat$, we have an inclusion $\motif X(\fat)\sub \motif Y(\fat)$, and this inclusion is a natural transformation, meaning that for any morphism $j\colon \tilde\fat\to \fat$, we have a commutative diagram
\commdiagram{\motif X(\fat)}{\sub}{\motif Y(\fat)}{\motif X(j)}{\motif Y(j)} {\motif X(\tilde\fat)}{\sub}{\motif Y(\tilde\fat)}
where the downward arrows are the maps induced functorially by $j$.
 We call a  presheaf  $\motif X$ on $\fatpoints \bassch$   
\emph{representable} (respectively, \emph{pro-representable}), if there 
exists a $\bassch$-scheme $X$ (respectively, a  scheme $X$ which is not 
necessarily of finite type over $\bassch$) such that $X(\fat)= \motif X(\fat)$, 
for all fat points $\fat$. To emphasize that we view the ($\bassch$-)scheme $X$ 
as a contravariant functor on  $\fatpoints \bassch$, we will denote it by $\func 
X:=\mor\bassch\cdot X$. Any morphism $\varphi\colon Y\to X$ of ($\bassch
$-)schemes induces, by composition, a natural transformation $\func
\varphi\colon \func Y\to \func X$, that is to say, a morphism in the 
category of  presheaves on $\fatpoints \bassch$. More precisely, given a fat point $\fat$ and a $\fat$-rational point $b\colon \fat \to Y$, let $\func\varphi(\fat)(b):=\varphi\after b$. Instead of $\func\varphi(\fat)$, we will simply write $\varphi(\fat)$ for the induced map $Y(\fat)\to X(\fat)$ if there is no danger for confusion.

\begin{lemma}\label{L:distrep}
Two closed  closed subschemes $X$ and $Y$ of a $\bassch$-scheme $Z$ are distinct \iff\ there is a fat point $\fat$ such that $X(\fat)$ and $Y(\fat)$ are distinct subsets of $Z(\fat)$.
\end{lemma}
\begin{proof}
One direction is immediate, so assume   $X$ and $Y$ are distinct. Then their restriction to some  affine open of $Z$ remains distinct, and hence we may assume $Z=\op{Spec}A$ is affine.   Let $I$ and $J$ be the ideals in $A$ defining $X$ and $Y$ respectively. Since $I\neq J$, there is a maximal ideal $\maxim\sub A$ such that $IA_\maxim\neq JA_\maxim$. Hence, by Krull's Intersection Theorem in the Noetherian local ring $A_\maxim$, there is some $n$ such that $I$ and $J$ remain distinct ideals in $A/\maxim^n$.  In particular, upon replacing $I$ by  $J$ if necessary, $\fat:=\op{Spec}(A/(I+\maxim^n))$ is a closed subscheme of $X$, but not of $Y$, showing that the closed immersion $\fat\sub Z$ lies in $X(\fat)-Y(\fat)$. 
\end{proof}

The resulting map from the 
category of $\bassch$-schemes to the   category of presheaves on $\fatpoints 
\bassch$  is a full embedding by   Yoneda's Lemma, Lemma~\ref{L:distrep}, and Proposition~\ref{P:morphrep} below. Note that this is no longer true for schemes not of finite type, an obvious 
reason for this failure being that there might be no rational points at all:   for 
instance, $\op{Spec}(\mathbb C)$  has no rational points over any fat point 
defined over the algebraic closure $\bar{\mathbb Q}$.  The 
product of two (pro-)representable presheaves is again (pro-)representable. 
More explicitly, the product $\func X\times \func Y$ is the same as $
\func{(X\times_\bassch Y)}$. If $s\colon \motif X\to \motif Y$ is a morphism 
of presheaves, then we define its \emph{image} $\fim s$ and its 
\emph{graph} $\graf s$ as the sub-presheaf of respectively $\motif X$ and $
\motif X\times \motif Y$, given at each fat point $\fat$ as respectively the 
image and the graph of the map $s(\fat)\colon \motif X(\fat)\to \motif 
Y(\fat)$. 
 
 \subsection*{Sieves}
 By a    \emph{sieve},
we mean a sub-presheaf $\motif X$ of some representable 
$\func X$. If we want to emphasize the underlying $\bassch$-scheme $X$, we say 
that $\motif X$ is \emph{a   sieve on} $X$, or that $X$ is an \emph{ambient space} of $\motif X$.  
Some examples of sieves: if $\varphi\colon Y
\to X$ is a morphism of $\bassch$-schemes, then we let  $
\fim\varphi$ be the image pre-sieve of the corresponding 
natural transformation $\func\varphi\colon\func Y\to \func X$, that is to 
say, $\fim \varphi(\fat)$, for a fat point   $\fat$,  consists of all  $\fat$-rational points on $X$ that \emph{lift}  to a $\fat$-rational point on $Y$, 
meaning that $\fat\to X$ factors through $Y$. Any sieve of the form $\fim\varphi$ for some morphism $\varphi\colon Y\to X$ of $\bassch$-schemes is called
 \emph{sub-\zariski}. If $\varphi$ is a 
(locally) closed or open immersion, then $\fim\varphi$ is equal to $\func Y$, 
whence is itself representable, and we call $\fim\varphi \iso\func Y $ respectively a 
\emph{(locally) closed} or \emph{open subsieve} on $X$. 
 
\begin{lemma}\label{L:vanci}
Let $\motif X$ be a sieve on $X$ and $\fat=\op{Spec}R$ a fat point. A $\fat$-rational point $a\colon \fat\to X$ belongs to $\motif X(\fat)$ \iff\ $\fim a\sub\motif X$. If $\motif X=\func Y$ is a closed subsieve, given by a closed subscheme $Y\sub X$, then this is also equivalent with $\fat\iso\fat\times_XY$ and also with $a^*\mathcal I_Y=0$, where $\mathcal I_Y$ is the ideal sheaf of $Y$ and $a^*\mathcal I_Y$ its image in $R$. 
\end{lemma}
\begin{proof}
Suppose $a\in\motif X(\fat)$. Let $\mathfrak w$ be a fat point  and $b\in\fim a(\mathfrak w)$. Hence $b\colon\mathfrak w\to X$ factors through $a$, that is to say, we can find a morphism $i\colon\mathfrak w\to \fat$ such that the diagram
\commtriangleback{\mathfrak w}{b}{X} a{\fat}i
commutes. By functoriality, $i$ induces a map $\motif X(\fat)\to\motif X (\mathfrak w)$, sending $a$ to $b$, proving that $b\in\motif X(\mathfrak w)$. Since this holds for all $\mathfrak w$ and $b$, we showed   $\fim a\sub\motif X$. Conversely, assume the  latter inclusion of sieves holds. In particular, the identity $1_\fat$ is a $\fat$-rational point whose image under $a(\fat)$ is just $a$, proving that $a\in \fim a(\fat)\sub\motif X(\fat)$.

To see the equivalence with the last two conditions if $\motif X=\func Y$, we may work locally and assume $X=\op{Spec} A$. Let   $I$ be the ideal defining $Y$, and let $A\to R$ be the \homo\ corresponding to $a$. Then $a\in Y(\fat)$ \iff\ $IR=0$ \iff\ $A/I\tensor_AR\iso R$, proving the desired equivalences.
\end{proof}

We may generalize the notion of an image sieve as follows. 
Given a sieve $\motif Y$ on an $\bassch$-scheme $Y$ and a morphism $
\varphi\colon Y\to X$ of $\bassch$-schemes, we define the \emph{push-forward} $\varphi_*\motif Y$ as the sieve on $X$ given at a fat point $\fat$ 
as the image of $\motif Y(\fat)\sub Y(\fat)$ under the map $
\varphi(\fat)\colon Y(\fat)\to X(\fat)$. In particular, $\varphi_*\func Y=
\fim\varphi$. Similarly, given a sieve $\motif X$ on $X$, we define its 
\emph{pull-back}   $\varphi^* \motif X$ as the sieve on $Y$ given at a fat 
point $\fat$ as the pre-image of $\motif X(\fat)$ under the map $
\varphi(\fat)\colon Y(\fat)\to X(\fat)$. In other words, $\varphi^* \motif 
X(\fat)$ consists of those rational points $\fat\to Y$ such that the 
composition $\fat\to Y\to X$ lies in $\motif X(\fat)$. The pull-back of a 
closed subsieve is again a closed subsieve: if $\bar X\sub X$ is a closed 
immersion and $\varphi\colon Y\to X$ a morphism, then $\varphi^*
\func{\bar X}$ is the closed subsieve given by $\inverse\varphi{\bar X}=Y
\times_X\bar X$. More generally, if $\bar X\to X$ is an arbitrary morphism, 
then the pull-back of the sub-\zariski\ sieve $\fim{\bar X\to X}$ is the sub-\zariski\ sieve 
$\fim{Y\times_X\bar X\to Y}$ of the base change.
 
  If $\motif X$ and $\motif Y$ are   sieves on an $\bassch$-scheme $X$, then we 
define their \emph{intersection} $\motif X\cap \motif Y$   and \emph{union} $\motif X
\cup \motif Y$ as the sieves given respectively by point-wise intersection and
  union, that is to say, $(\motif X\cap \motif Y)(\fat)= 
\motif X(\fat)\cap \motif Y(\fat)\sub X(\fat)$  and $(\motif X\cup \motif Y)(\fat)= \motif 
X(\fat)\cup \motif Y(\fat)\sub X(\fat)$. One easily checks that they are again sieves, that is to say, (contravariant) functors. We express this by saying that the sieves on $X$ form a \emph{lattice}. Clearly, the 
intersection   of two closed subsieves    is again a closed subsieve: $\func V
\cap \func W=\func{(V\cap W)}$, for $V,W\sub X$ closed subschemes, but this is no longer true for their union.   The disjoint union  $\motif X\sqcup\motif Y$ is equal to the union of the push-forwards of $\motif X$ and $\motif Y$ under the immersions $X\to X\sqcup Y$ and $Y\to X\sqcup Y$ respectively. 
 
By the \emph{Zariski closure} of $\motif X$ in $X$, denoted $\bar{\motif X}$, we mean the intersection of all closed subschemes $Y\sub X$ such that $\motif X\sub\func Y$. By   Noetherianity, $\motif X$ is a sieve on its Zariski closure $\bar{\motif X}$, and the   latter is the smallest closed subscheme on which $\motif X$ is a sieve. We say that $\motif X$ is \emph{Zariski dense} in $X$ if $\bar{\motif X}=X$. For instance, given a morphism $\varphi\colon Y\to X$ of $\bassch$-schemes, the Zariski closure of   $\fim\varphi$  is the so-called \emph{scheme-theoretic image} of $\varphi$, that is to say, the closed subscheme of $X$ given by the kernel of the induced morphism $\loc_X\to \varphi_*\loc_Y$.  When $\fim\varphi$ is Zariski dense, one says that $\varphi$ 
is \emph{dominant}.

\begin{lemma}\label{L:domfat}
If $\mathfrak w\to \mathfrak v$ is a dominant morphism of fat points, then the induced map $\motif X(\mathfrak v)\to \motif X(\mathfrak w)$ is injective, for any sieve $\motif X$. 
\end{lemma}
\begin{proof}
This is a form of duality: dominant morphisms are epimorphisms and so under a contravariant functor they become monomorphisms. More explicitly, since $\motif X\sub \func X$, for some $\bassch$-scheme $X$, it suffices to show injectivity for the latter, that is to say, we may assume $\motif X$ is representable, and then since the problem is local, we may assume $X$ is affine with coordinate ring $A$. Let $a,a'\in X(\mathfrak v)$ have the same image in $X(\mathfrak w)$. If $R\to S$ is the  \homo\ corresponding to $\varphi$, then dominance means that this \homo\ is injective. Since, by assumption,  the two   \homo{s} $A\to R$  induced respectively  by $a$ and $a'$ give rise to the same \homo\ $A\to S$ when composed with the injection $R\sub S$, they must already be equal, as we needed to show. 
\end{proof}

\begin{example}\label{E:zarclosrat}
Suppose $\bassch$ is the spectrum of an \acf\ $\fld$. Zariski closure does not commute with taking $\fld$-rational points, that is to say, the $\fld$-rational points of the Zariski closure of a sieve $\motif X$ in $X$ may be bigger than the Zariski  closure of $\motif X(\fld)$ (in the usual Zariski topology) in $X(\fld)$. For instance, as we will see shortly, the cone  $\cone O$ of the  origin $O$ on the affine line $\affine\fld 1$ has Zariski closure equal to $\affine\fld 1$, whereas its $\fld$-rational points consist just of the origin. 
\end{example}

\subsection*{Germs of sieves}
Strictly speaking, a sieve is a pair $(\motif X,X)$ consisting of a sub-presheaf $\motif X$ of an ambient space $X$, but we will often   treat   sieves as   abstract objects, that is to say, disregarding their ambient space. This allows us, for instance,  to view the same sieve as already defined  on a smaller subscheme. 
To give a more formal treatment, let us call a \emph{germ of a sieve} any equivalence class of pairs $(\motif Y,Y)$, where we call two such pairs $(\motif Y,Y)$ and $(\motif Y',Y')$ equivalent if  there exists a pair $(\motif X,X)$ and locally closed immersions $\varphi\colon Y\to X$ and $\varphi\colon Y'\to X'$ such that $\varphi_*\motif Y=\motif X=\varphi_*'\motif Y'$. In particular, $\motif Y$ is then also a sieve on the  intersection $Y\cap Y'$, viewed as a locally closed subscheme of $X$. Therefore, we may always assume, if necessary, that $\motif Y$ is Zariski dense in $Y$, and then it is also a Zariski dense sieve on any open subsieve of $Y$ containing it. Henceforth, we will often confuse a sieve with the germ it determines.

\subsection*{Complete sieves}
The complement of a sieve $\motif X\sub\func X$ is in general not a sieve, as witnessed by any closed subsieve. Let us call a sieve $\motif X$ on $X$ \emph{complete} if 
$$
\motif X(\fat)=\inverse{X(j)}{\motif X(\tilde\fat)},
$$
 for every morphism $j\colon\tilde\fat\to \fat$ of fat points,  where $X(j)\colon X(\tilde\fat)\to X(\fat)$ is the map induced functorially by $j$.
 More generally, if $\motif Y\sub\motif X$ is an inclusion of sieves on $X$, then we say that $\motif Y$ is \emph{relatively complete in $\motif X$}, if 
 $$
 \motif Y(\fat)=\inverse{X(j)}{\motif Y(\tilde\fat)}\cap \motif X(\fat),
 $$
  for all  morphisms $j$.

 If $\bassch$ is the spectrum of   an \acf\ $\fld$, then $\bassch$ itself is a fat point, and any fat point $\fat$ admits a unique morphism $\bassch\to \fat$ (given by the residue field of $\fat$). We let $\rho_\fat\colon X(\fat)\to X(\fld)$ be the induced map. One easily verifies that $\motif Y$ is relatively complete in $\motif X$   \iff\ $\motif Y(\fat)=\motif X(\fat)\cap \inverse{\rho_\fat}{\motif Y(\fld)}$, for every fat point $\fat$. In fact, for $\bassch$ arbitrary, we have a similar criterion in that we only have to check the condition for every morphisms of fat points $j\colon\tilde\fat\to \fat$ in which $\tilde\fat$ has length one (necessarily therefore the spectrum of a field extension of the residue field of $\fat$).

As we shall see in Proposition~\ref{P:formsieve} below, open subsieves over an \acf\ are  complete, and their complements are again   sieves. The second property in fact follows from the first, as one easily verifies:

\begin{lemma}\label{L:complcompl}
Given an inclusion of sieves $\motif Y\sub\motif X$, then $\motif Y$ is relatively complete in $\motif X$ \iff\ $\motif X-\motif Y$   is again a sieve.\qed
\end{lemma}

Assume $\bassch$ is the spectrum of  an \acf\ $\fld$. Given a $\fld$-scheme $X$ and a subset $V\sub X(\fld)$, we define the \emph{cone $\cone V$ over $V$} to be the sieve given by
$$
\cone V(\fat):=\inverse{\rho_\fat}V
$$
for every fat point $\fld$. By our previous discussion on complete sieves over an \acf, it follows that $\cone V$ is complete, and its complement is the cone $\cone{-V}$. More generally, for any sieve $\motif X$, the intersection $\motif X\cap\cone V$ is relatively complete in $\motif X$, and its complement in $\motif X$ is equal to $\motif X\cap\cone{-V}$. We have the following converse:

\begin{lemma}\label{L:relcompl}
Over an \acf\ $\fld$, a subsieve $\motif Y\sub \motif X$ is relatively complete \iff\ it is the intersection of $\motif X$ and a  cone.
\end{lemma}
\begin{proof}
Let $\motif Y$ be relatively complete in $\motif X$, and let $V:=\motif Y(\fld)$. Then $\motif Y=\cone V\cap \motif X$, since both have the same $\fat$-rational points, for any fat point $\fat$. 
\end{proof}

Given a sieve $\motif X$, we define its \emph{completion} as the cone $\complet{\motif X}:=\cone{\motif X(\fld)}$. By functoriality, $\motif X$ is contained in $\complet{\motif X}$, and $\complet{\motif X}$ is the smallest complete sieve containing $\motif X$.

\subsection*{The category of sieves}
By definition, a \emph{morphism} between sieves is a  natural transformation, giving rise to the   \emph{category  of $\bassch$-sieves}, denoted $\sieves \bassch$.    If $\bassch$ is Jacobson,\footnote{A scheme is called \emph{Jacobson} if it admits an open covering by affine Jacobson schemes; an affine scheme $\op{Spec}A$ is Jacobson, if $A$ is, meaning that any radical ideal is the intersection of the maximal ideals containing it.}
 then the category of $\bassch$-schemes fully embeds in the category  $\sieves\bassch$, by the following result:

\begin{proposition}\label{P:morphrep}
Assume $\bassch$ is Jacobson. 
Let $\motif X$ be a sieve on a $\bassch$-scheme $X$, and let $Y$ be a $\bassch$-scheme. Any morphism $s\colon\func Y\to\motif X$ is induced by a morphism $\varphi\colon Y\to X$, that is to say, $s=\func\varphi$ and $\fim\varphi\sub\motif X$.
\end{proposition}
\begin{proof}
Assume first that $Y=\fat$ is a fat point, and $\motif X=\func X$ is representable. Let $q:=s(\fat)(1_\fat)$, where $1_\fat$ is the identity morphism on $\fat$ viewed as a $\fat$-rational point on $\fat$. Hence $q\in X(\fat)$, that is to say, a $\bassch$-morphism $\fat\to  X$. We have to show that $s=\func q$, so let $\mathfrak w$ be any fat point, and $a\colon\mathfrak w\to \fat$ a $\mathfrak w$-rational point on $\fat$. By definition of $\func q$, we have $ q(\mathfrak w)(a)=qa$. On the other hand, functoriality yields a commutative diagram
\commdiagram {\fat(\fat)}{\func\fat(a)}{\fat(\mathfrak w)}{s(\fat)}{s(\mathfrak w)}{X(\fat)}{\func X(a)}{X(\mathfrak w)}
We trace the image of $1_\fat$ in $X(\mathfrak w)$ through this diagram. The top arrow sends it to $a$ and hence its image in $X(\mathfrak w)$ is $s(\mathfrak w)(a)$. On the other hand, $s(\fat)(1_\fat)=q$ and its image under $\func X(a)$ is $qa$, showing that $s(\mathfrak w)(a)=qa=q(\mathfrak w)(a)$, whence $s=\func q$.

Assume next that $Y$ is arbitrary. Let $Z\sub Y$ be a closed subscheme of dimension zero. There exist finitely many fat points $\fat_1,\dots,\fat_s\sub Y$ such that $Z=\fat_1\sqcup\dots\sqcup\fat_s$. By what we just proved, there exists a morphism $q_Z\colon Z\to X$ such that $\restrict s{\func Z}=\func q_Z$, for each zero-dimensional closed subscheme $Z\sub Y$.   By Lemma~\ref{L:dirlimJac} below, the  direct limit of all such closed subschemes $Z$ is $Y$, and hence, by the universal property of direct limits, we get   a morphism $\varphi\colon Y\to X$, such that $q_Z=\restrict\varphi Z$. Checking at every fat point, we see that  $\func\varphi=s$. Finally, assume $\motif X$ is an arbitrary sieve. By what we just proved, the composition $\func Y\to \motif X\sub\func X$ is induced by a morphism $\varphi\colon Y\to X$. Since the image of $s=\func\varphi$ at any fat point $\fat$ lies inside $\motif X(\fat)$, we must have $\fim\varphi\sub\motif X$.\end{proof}

\begin{lemma}\label{L:dirlimJac}
If $\bassch$ is Jacobson, then any $\bassch$-scheme is the direct limit of its zero-dimensional closed subschemes.
\end{lemma}
\begin{proof}
I will only consider the case that $\bassch$  is  the spectrum of a field $\fld$. Reasoning on a finite open affine covering, we may further assume that   $X=\op{Spec}A$ is an affine $\fld$-scheme. Let $\tilde A$ be the inverse limit of all residue rings $A/\mathfrak n$ of finite length (we say that $\mathfrak n$ has \emph{finite colength}). We have to show that $A=\tilde A$. The inclusion $A\sub \tilde A$ follows from the fact that the intersection of all ideals of finite colength is zero. 

Before we give the proof, we establish a preservation result under finite maps: if $A=\tilde A$ and $A\to B$ is a finite \homo, then also $B=\tilde B$. Indeed, if $\beta_1,\dots,\beta_s$ generate $B$ as an $A$-module, then they also generate every $B/\mathfrak nB$ as a $A/\mathfrak n$-module, where $A/\mathfrak n$ has finite length. Since the $\mathfrak nB$ are cofinal in the set of all ideals of $B$ of finite colength, $\tilde B$ is generated by the $\beta_i$ as an $\tilde A$-module, and hence $B= \tilde B$. This argument also  shows that  $\widetilde{(A/I)}$ is equal to $\tilde A/I\tilde A$, for any ideal $I\sub A$. By Noether Normalization, $A$ is a finite extension of a polynomial ring over $\fld$ in the same number of variables as the dimension $d$ of $A$. Hence, it will suffice to show the  the identity $A=\tilde A$ for polynomial rings. 

We  induct  on  $d$, where the case $d=0$ is trivial. Let $A=\pol\fld \var$ with $\var$ a $d$-tuple of variables, and let $\mathfrak o$ be a complete \DVR\ containing $A$ and having non-zero center $\pr
\sub A$ (recall that $\pr$ is the prime ideal of elements in $A$ of positive value). Assume first that $d=1$, so that $\pr$ is   a maximal ideal. Since $\tilde A$ is then contained in the completion $\complet A_\pr$ (as the latter is the inverse limit of all $A/\pr^n$), and since by the universal property of completion $\complet A_\pr\sub \mathfrak o$, we showed $\tilde A\sub \mathfrak o$. Since the intersection of all these complete \DVR{s} is equal to the normalization of $A$, whence to $A$, we showed $A=\tilde A$ in this case. As mentioned above,   Noether Normalization then proves the result for all one-dimensional algebras, and so we may assume $d>1$. If $\pr$ is a maximal ideal, the previous argument yields again $\tilde A\sub \mathfrak o$. So assume $\pr$ is a non-maximal, non-zero prime ideal. Since $A/\pr^n$ has dimension less than $d$, induction plus preservation under homomorphic images then yields $A/\pr^n=\tilde A/\pr^n\tilde A$. In particular, the $\pr$-adic completion of $\tilde A$ is equal to $\complet A_\pr$, and hence is contained in $\mathfrak o$. Hence we showed that $\tilde A\sub \mathfrak o$, for any such \DVR\ $\mathfrak o$, so that  the same argument as above yields $A=\tilde A$.
\end{proof}

\begin{corollary}\label{C:fatmor}
Suppose $\bassch$ is Jacobson. For any sieve $\motif X$, we have  an isomorphism of sieves
$$
\mor{\sieves\bassch}{\func{(\cdot)}}{\motif X}\iso \motif X(\cdot).
$$
\end{corollary}
\begin{proof}
Let $\fat$ be a fat point and $s\colon\func\fat\to \motif X$ a morphism of sieves. By Proposition~\ref{P:morphrep}, this morphism is induced by a morphism $a\colon \fat\to X$ of schemes, where $X$ is some ambient space of $\motif X$. Since $s=\func a$, we have an inclusion $\fim a\sub \motif X$,  and hence  $a\in \motif X(\fat)$ by   Lemma~\ref{L:vanci}. The converse follows along the same lines.  
\end{proof}

Inspired by the  result in Proposition~\ref{P:morphrep}, we call a natural transformation $s\colon \motif Y\to \motif X$  between sieves \emph{\explicit}, if there exists a morphism of  $\bassch$-schemes $\varphi\colon Y\to X$ such that $\motif X$ and $\motif Y$ are  sieves on respectively $X$ and $Y$, and such that $s$ is the restriction of $\func\varphi\colon\func Y\to \func X$; we might also express this by  saying that $s$ \emph{extends} to a morphism of schemes.  Since the definition allows the ambient spaces to be dependent on $s$, we have actually defined a morphism between germs of sieves. It follows that if $\motif Y$ is Zariski dense in $Y$ and $\motif X$ is a sieve on $X$ (without any further restriction), then $s$ extends to a morphism $\tilde Y\to X$ for some open $\tilde Y\sub Y$  on which $\motif Y$ is also a sieve, that is to say, such that $\motif Y\sub\func {\tilde Y}$.

The composition of two \explicit\ natural transformations is again \explicit. Indeed, let $s\colon \motif Z\to \motif Y$ and $t\colon\motif Y\to \motif X$ be \explicit, extending respectively to   morphisms $\varphi\colon Z\to Y$ and $\psi\colon Y'\to X$. Since $Y$ and $Y'$ are both ambient spaces for $\motif Y$, so is their intersection $Y'':=Y\cap Y'$, which is therefore locally closed in either.
 Hence the restriction of $\varphi$ to $\inverse\varphi {Y''}$ (respectively, the restriction of $\psi$ to $Y''$ )   is a morphism extending $s$ (respectively $t$), and therefore, the composition $\inverse\varphi{Y''}\to X$ extends $t\after s$. We can therefore define the \emph{\weak\ category of sieves}, denoted $\sievesweak\bassch$, as the subcategory of all sieves in which the  morphisms are only the \explicit\ ones.

A note of caution: not every morphism of  sieves is \explicit, and we will discuss some examples later. Moreover, even if it is,  one cannot always extend it to any ambient space of the source sieve. An example is in order:
 
\begin{example}\label{E:hyperbola}
Let $H\sub\affine\fld 2$ be the hyperbola with equation $\var\vary=1$ over a field $\fld$, and let $L_*$ be the \emph{punctured} line, that is to say, the affine line minus the origin. Note that this is an affine scheme with coordinate ring $\pol\fld{\var,1/\var}$. The projection $\affine\fld 2\to \affine\fld 1$ onto the first coordinate induces an isomorphism   $H\to L_*$. Its inverse induces an  isomorphism $\func L_*\to\func H$, which is trivially \explicit, as both sieves are representable. However, although $\func L_*$ is an open subsieve on $\affine \fld 1$, the above isomorphism does not extend  (since $1/\var$ is not a polynomial). 
\end{example}

We may generalize the definitions of pull-back and push-forward along a morphism of   sieves as follows. Let $s\colon \motif Y\to \motif X$ be a morphism of sieves. Given a subsieve $\motif Y'\sub\motif Y$, we define its \emph{push-forward} $s_*\motif Y'$ as the presheaf defined at each fat point $\fat$ as the image of $\motif Y'(\fat)$ under $s(\fat)$. Similarly, given a subsieve $\motif X'\sub\motif X$, we define its \emph{pull-back} $s^*\motif X'$  as the presheaf defined at each fat point $\fat$ as the pre-image of $\motif X'(\fat)$ under $s(\fat)$.

\section{The topos of sieves}\label{s:globsect}
In view of Proposition~\ref{P:morphrep}, we will henceforth assume that the base space $\bassch$  is Jacobson.  Recall that $\affine\bassch n:=\affine\zet n\times\bassch$ is the \emph{affine $n$-space} over $\bassch$. 
 
\subsection*{Section rings}
Given a sieve $\motif X$ on $X$, we define its \emph{global section 
ring} as  
$$
H_0(\motif X):=\mor{}{\motif X}{\affine \bassch1},
$$
 where by the latter, we actually mean the collection of morphisms  $\motif X\to 
\func{(\affine \bassch1)}$, but for notational simplicity, we will identify a scheme with the functor it  represents if there is no danger for confusion. For each fat point $\fat=\op{Spec}R$, we have a 
natural bijection $\Psi_{\fat}\colon\affine \bassch1(\fat)\iso R$ defined as follows. Given a 
rational point $a\colon\fat\to \affine \bassch 1$, it factors through an affine open of the form $\affine\bas1\sub\affine\bassch1$, for some affine open $\op{Spec}\bas\sub\bassch$, and hence induces a \homo\ $\pol\bas\vary\to R$, which, again for notational simplicity, we denote again by $a$. Now, set $\Psi_\fat(a):=a(\vary)\in R$. 
This identification endows $\affine\bassch 1(\fat)$ 
with a ring structure, and by transfer, then makes  $H_0(\motif X )$ into a 
ring. Indeed, given sub-\zariski\ morphisms  $s,t\colon \motif X\to 
 \affine \bassch1$, we define their sum $s+t$ (respectively, their 
product $st$) as the   morphism which at a fat point $\fat$ maps $a\in 
\motif X(\fat)$ to 
$$
s(\fat)(a)+t(\fat)(a)=\inv{\Psi_{\fat}}\big(\Psi_{\fat}(s(\fat)(a))+\Psi_{\fat}
(t(\fat)(a))\big)
$$
and a similar formula for  $s(\fat)(a)\cdot t(\fat)(a)$. The functoriality of $s+t$ 
and $st$ is easy, showing that they are again global sections.

 As we shall see below,   \explicit\ morphisms will play a key role, and so we define the ring  of \emph{\explicit\ sections} $\gsect{\motif X}$ as the subset of $H_0(\motif X)$ consisting of all \explicit\ morphisms $\motif X\to  \affine\bassch 1$. To see that this is closed under sums and products, let $s$ and $s'$   be two \explicit\ sections. By definition, there exist morphisms $X\to \affine\bassch 1$ and $X'\to\affine\bassch 1$ inducing $s$ and $s'$ respectively, where   $X$ and $X'$ are ambient spaces for $\motif X$. In particular,  the locally closed subscheme $X'':=X\cap X'$ is then also  an ambient space for $\motif X$. Hence $s+t$ and $st$ extend to the sum and product of the restrictions to $X''$ of $X\to \affine\bassch 1$ and $X'\to\affine\bassch 1$, proving that they are also \explicit. By Theorem~\ref{T:globsectsubzar}, every global section on a \zariski\ motif is \explicit.

 \begin{lemma}\label{L:funcglobsect}
The assignment  $\motif X\mapsto H_0(\motif X)$ is a contravariant 
functor  from $\sieves\bassch$ to the category of $\loc_\bassch$-algebras. Similarly, $\motif X\mapsto \gsect{\motif X}$ is a contravariant functor on the \weak\   category $\sievesweak\bassch$. In particular, 
if two sieves are  isomorphic, then they have the same 
global section ring, and if they are \explicit{ally} isomorphic, then they have the same \explicit\ section ring.
\end{lemma}
\begin{proof}
If $s\colon \motif X\to \motif Y$ is a morphism of sieves, then pulling-back induces a $\loc_\bassch$-algebra \homo\ 
$H_0(\motif Y)\to H_0(\motif X)$ given by $t\mapsto t\after s$, for $t
\colon \motif Y\to \affine \bassch1$. One easily verifies that this 
constitutes a contravariant functor. Since the pull-back of a  \explicit\ morphism under an \explicit\ morphism is easily seen to be \explicit\ again, we get an induced \homo\ $\gsect{\motif Y}\to \gsect{\motif X}$. 
\end{proof}
  
\begin{proposition}\label{P:globsect}
The global section ring of a representable functor $\func X$ is equal to the 
ring of global sections $H_0(X):=\Gamma(\loc_X,X)$ of the corresponding scheme, and this is also its \explicit\ section ring.
\end{proposition}
\begin{proof}
A morphism $\func X\to \affine\bassch1$ corresponds 
by Proposition~\ref{P:morphrep} to a morphism of $\bassch$-schemes $X\to\affine
\bassch 1$, and it is well-known that the collection of all these is precisely the 
ring of global sections on $X$ (see, for instance, \cite[II. Exercise 2.4]{Hart}). 
 \end{proof}
 
In particular, if $\motif X$ is a sieve on an affine scheme $\op{Spec}A$, then $H_0(\motif X)$ and $\gsect{\motif X}$ are $A$-algebras.

\begin{corollary}\label{C:expsectlim}
The \explicit\ section ring  $\gsect{\motif X}$ of a sieve $\motif X$ is the inverse limit of all $H_0(X)$, where $X$ runs over all ambient spaces of $\motif X$.  
\end{corollary}
\begin{proof}
 If $s\colon\motif X\to \affine\bassch 1$ is an \explicit\ section, then it extends to a morphism $X\to \affine\bassch 1$, where $X$ is some ambient space of $\motif X$, and hence by the argument in the above proof, it is the image of an element in $H_0(X)$ under the  \homo\ $H_0(X)\to \gsect{\motif X}$ induced by the inclusion   $\motif X\sub\func X$. If $X'\sub X$ is a locally closed subscheme which is also an ambient space for $\motif X$, then $s$ extends to a morphism with domain $X'$, and this is therefore necessarily the restriction of the global section in $H_0(X)$ determined by $s$. This shows that the $H_0(X)$ form an inverse system as $X$ varies over the germ of $\motif X$, with limit equal to $\gsect{\motif X}$. 
 \end{proof}

\begin{example}\label{E:explsectformfat}
If $\motif X$ is Zariski dense in $X$, then the only ambient spaces of $\motif X$ inside $X$ are open, so that we may think of  the ring of \explicit\ sections as a sort of stalk:
$$
\gsect{\motif X}=\varprojlim \gsect U
$$
where $U$ runs over all open subschemes on which $\motif X$ is a sieve. 
In particular, if $\bassch$ is the spectrum of an \acf\ $\fld$ and $\motif X(\fld)=X(\fld)$, then there are no proper opens in $X$ on which $\motif X$ is a sieve, and hence $\gsect{\motif X}=H_0(X)$.  This holds automatically if $X$ is   for instance a fat point.
\end{example}

\begin{theorem}\label{T:globsectunion}
If $\motif X$ and $\motif Y$ are sieves on a common $\bassch$-scheme, 
then the natural commutative diagram 
\commdiagram[un]{H_0(\motif X\cup \motif Y)}{i_{\motif Y}}{H_0(\motif 
Y)}{i_{\motif X}}{p_{\motif Y}}{H_0(\motif X)}{p_{\motif X}}{H_0(\motif X
\cap \motif Y)}
is  Cartesian. In particular, if $\motif X$ and $\motif Y$ are disjoint, then 
$H_0(\motif X\cup \motif Y)\iso H_0(\motif X)\oplus H_0(\motif Y)$. Similar properties hold for the \explicit\ section ring.
\end{theorem} 
\begin{proof}
The construction of the commutative square~\eqref{un} and the 
verification that it is commutative, follows easily from Lemma~
\ref{L:funcglobsect}. Recall that \eqref{un}, as a commutative square in the 
category of $\loc_\bassch$-algebras, is called \emph{Cartesian} or a \emph{pull-back}, if   $H_0(\motif X\cup \motif Y)$ is universal in this category for 
making the diagram commute, or, equivalently, if it is the equalizer of the 
two compositions $p_{\motif X}i_{\motif X}$ and $p_{\motif Y}i_{\motif Y}
$. To verify this property, let $s\in H_0(\motif X)$ and $t\in H_0(\motif Y)$ 
be such that $\restrict s{\motif X\cap \motif Y}=p_{\motif X}(s)=p_{\motif 
Y}(t)=\restrict t{\motif X\cap \motif Y}$. We then define $u\in H_0(\motif 
X\cup \motif Y)$ as follows. Given a fat point $\fat$, let $u(\fat)$ be the 
map sending  $a\in \motif X(\fat)\cup \motif Y(\fat)$  to $s(\fat)(a)$ if  $a
\in \motif X(\fat)$, and to $t(\fat)(a)$ if $a\in \motif Y(\fat)$. This is well-defined, since $s(\fat)$ and $t(\fat)$ agree on $\motif X(\fat)\cap \motif 
Y(\fat)$ by assumption. It is now easy to verify that this defines a morphism 
of sieves $u\colon \motif X\cup \motif Y\to  \affine \bassch 1$, that is 
to say, a global section of $\motif X\cup \motif Y$, and that $i_{\motif X}
(u)=s$ and $i_{\motif Y}(u)=t$. Moreover, if $s$ and $t$ are \explicit, then so is $u$. 
\end{proof}

\begin{remark}\label{R:globsectunion}
To formulate a version for more than two sieves, we resort to the language of sheaf theory (and, shortly, we will put everything in this context anyway): given a union  $\motif X=\motif X_1\cup\dots\cup\motif X_s$, we have an exact sequence
\begin{equation}\label{eq:sheaf}
0\to H_0(\motif X)\to \bigoplus_i H_0(\motif X_i)\ \underset{\delta_{ji}}{\overset{\delta_{ij}}\rightrightarrows}\ \bigoplus_{ i< j} H_0(\motif X_i\cap\motif X_j)
\end{equation}
where $\delta_{ij}\colon H_0(\motif X_i)\to H_0(\motif X_i\cap\motif X_j)$ is the restriction \homo\ on the global sections induced by the inclusion $\motif X_i\cap\motif X_j\sub\motif X_i$. The corresponding exact sequence for \explicit\ sections is
\begin{equation}\label{eq:sheafgeom}
0\to \gsect{\motif X}\to \bigoplus_i \gsect{\motif X_i}\ \underset{\delta_{ji}}{\overset{\delta_{ij}}\rightrightarrows}\ \bigoplus_{ i< j} \gsect{\motif X_i\cap\motif X_j}.
\end{equation}
\end{remark}

\begin{theorem}\label{T:globsectsubzar}
Any morphism $s\colon \motif X\to \motif Z$ of sieves with $\motif X$   sub-\zariski, is \explicit. In particular, $H_0(\motif X)= \gsect{\motif X}$. 
\end{theorem}
\begin{proof}
Let us prove the second assertion first. 
Let $\varphi\colon Y\to X$ be a morphism of $\bassch$-schemes, so that $\motif X=\fim\varphi$. Replacing $X$ by the Zariski closure of $\fim \varphi$, we may assume that $\varphi$ is dominant. 
 Our objective is to show that we have an equality
\begin{equation}\label{eq:subzarsect}
\gsect{\fim\varphi}=H_0(\fim\varphi).
\end{equation}
Assume first that $Y=\mathfrak y$ is a fat point, and hence, since $\varphi$ is dominant, so is then $X=\mathfrak x$. Consider the induced \homo\ of Artinian local rings $R\sub S$, which is injective precisely because $\varphi$ is dominant. Let $s\colon\fim\varphi\to \affine\bassch1$ be a global section. By Proposition~\ref{P:morphrep}, we can find $q\in S$ which, when viewed as  a global section $\mathfrak y\to \affine\bassch1$, extends   the composition $s\after \func\varphi$. Let $g_1\colon S\to S\tensor_RS$ and $g_2\colon S\to S\tensor_RS$ be given by respectively $a\mapsto a\tensor 1$ and $a\mapsto 1\tensor a$. Since $g_1$ and $g_2$ agree on $R$, the two corresponding rational points 
$\mathfrak y\times_{\mathfrak x}\mathfrak y\to \mathfrak y$ have the same image under $s(\mathfrak y\times_{\mathfrak x}\mathfrak y)$ and hence, $g_1(q)=g_2(q)$. By Lemma~\ref{L:tensCart} below, we get $q\in R$.

By Corollary~\ref{C:expsectlim}, this  proves  \eqref{eq:subzarsect} 
 whenever the domain of $\varphi$ is  a fat point. If the domain is a zero-dimensional scheme $Z$, then it is a finite disjoint union $\mathfrak y_1\sqcup\dots\sqcup\mathfrak y_s$ of fat points. By an induction argument, we may assume $s=2$, so that  $\fim\varphi=\fim{\restrict\varphi{\mathfrak y_1}}\cup \fim{\restrict\varphi{\mathfrak y_2}}$. In particular, $H_0(\fim\varphi)$ and $\gsect{\fim\varphi}$ satisfy each the same Cartesian square~\eqref{un} by what we just proved for fat points (instead of induction, we may alternatively use \eqref{eq:sheaf} and \eqref{eq:sheafgeom}). By uniqueness, they must therefore be equal, showing that \eqref{eq:subzarsect} holds whenever the domain is zero-dimensional. For $Y$ arbitrary,  we may write it as the direct limit of all its zero-dimensional closed subschemes by  Lemma~\ref{L:dirlimJac}. Let $\{\motif Z:=\fim{\restrict\varphi Z}\}$ be the collection of all sub-\zariski\ motives, where $Z\sub Y$ varies over all zero-dimensional closed subschemes of $Y$. Since the $\motif Z$ form a direct system, $H_0(\fim\varphi)$ is   the inverse limit of all $H_0(\motif Z)=\gsect{\motif Z}$, by Lemma~\ref{L:dirlimform} below, and by the same argument, this direct limit is equal to $\gsect{\fim\varphi}$, proving  \eqref{eq:subzarsect}. 
 
 To prove the first assertion, we may assume, without loss of generality, that $\motif Z=\func Z$ is representable, and then, since the problem is local, that $X=\op{Spec}A$ and $Z$ are affine. Hence $Z\sub\affine\bas n$ is a closed subscheme for some $n$ and some open $\op{Spec}\bas\sub\bassch$. Let $s_i$ be the composition of $s$ with the morphism induced by the projection $\affine\bas n\to\affine\bas 1$ onto the $i$-th coordinate. Hence $s_i\in H_0(\motif X)=\gsect{\motif X}$. Replacing $X$ by an open subscheme if necessary as per Corollary~\ref{C:expsectlim}, we can find  $q_i\in A$  such that, viewed as a global section $X\to \affine\bas 1$, it extends $s_i$. Therefore,  the morphism $X\to \affine\bas n$ given by $(q_1,\dots,q_n)$ is an extension of $s$, as we wanted to show. 
\end{proof}

\begin{remark}\label{R:globsectsubzar}
By the above proof, we actually showed that if the domain of $\varphi$ is zero-dimensional, then $H_0(\fim\varphi)$ is equal to the global section ring of the Zariski closure of $\fim\varphi$. However, this is no longer true in the general case, as can be seen from Corollary~\ref{C:expsectlim}. 
\end{remark}

\begin{lemma}\label{L:dirlimform}
Let $\{\motif Z\}$ be a direct system of sieves on some scheme $X$ and let $\motif X$ be their direct limit. Then $H_0(\motif X)$ is the inverse limit of all $H_0(\motif Z)$. 
\end{lemma}
\begin{proof}
Contravariance turns a direct limit into an inverse limit, and the rest is now   an easy consequence of the universal property of inverse limits:
\begin{align*}
H_0(\motif X) &=\mor{}{\motif X}{\affine\bassch 1}\\
&=\mor{}{\varinjlim \motif Z}{\affine\bassch 1}\\
&=\varprojlim \mor{}{\motif Z}{\affine\bassch 1}= 
\varprojlim H_0(\motif Z).
\end{align*}
\end{proof}

\begin{lemma}\label{L:tensCart}
Let $R\sub S$ be an injective \homo\ of rings. Then the tensor square
\commdiagram R{}S{}{}S{}{S\tensor_RS}
is Cartesian, that is to say, if $q\tensor1=1\tensor q$ in $S\tensor_RS$ for some $q\in S$, then in fact $q\in R$. 
\end{lemma}
\begin{proof}
Let for simplicity assume that $R$ and $S$ are algebras over some field $\fld$ (since we only need the result for $R$ and $S$ Artinian, this already covers any equi\ch\ situation).
Let $T:=S\tensor_\fld S$ be the tensor product over $\fld$, and let $\mathfrak n$ be the ideal in $T$ generated by all expressions of the form $r\tensor1-1\tensor r$ for $r\in R$. Hence $S\tensor_RS\iso T/\mathfrak n$. If $q\in S$ satisfies $q\tensor1=1\tensor q$ in $S\tensor_RS$, then viewed as an element in $T$, the tensor $q\tensor1-1\tensor q$ lies in $\mathfrak n$. The canonical surjection $S\to S/R$ induces a \homo\ of tensor products $T\to (S/R)\tensor_\fld(S/R)$. Under this \homo,   $\mathfrak n$ is sent to the zero ideal, whence so is in particular $q\tensor1-1\tensor q$. If the image of $q$ in $S/R$ were non-zero, then we can find a basis of $S/R$ containing $q$. Hence $q\tensor 1$ and $1\tensor q$ are two independent basis vectors of  $(S/R)\tensor_\fld(S/R)$, contradicting that they are equal in the latter ring. Hence $q\in R$, as we wanted to show. 
\end{proof}

By Lemma~\ref{L:funcglobsect} and properties of tensor products, we have 
for any two sieves $\motif X$ and $\motif Y$, a canonical \homo
\begin{equation}\label{eq:globsectprod}
H_0(\motif X)\tensor_{\loc_\bassch} H_0(\motif Y)\to H_0(\motif X\times \motif Y),
\end{equation}
and a similar formula for \explicit\ sections. 
If $\motif X$ and $\motif Y$ are both representable, then this is an 
isomorphism, but not so in general. 
 
\subsection*{The topos of sieves}
 The Zariski topology on a $\bassch$-scheme $X$ induces   a \emph{topos} on each of 
its sieves $\motif X$. More precisely, the \emph{admissible opens} on $\motif X$ are 
the sieves of the form $\motif X\cap \func U$, where $U\sub X$ 
runs over all opens of $X$; and the admissible coverings are all collections of admissible opens $\motif U_i\sub\motif X$ such that their union (as sieves) is equal to $\motif X$, that is to say, such that the corresponding opens $U_i\sub X$ cover some ambient space of $\motif X$. For simplicity, we will simply write $\motif X\cap U$ for $\motif X\cap\func U$. In particular, since $X$ is quasi-compact, any admissible covering contains a finite admissible subcovering. The collection of admissible opens does not depend on the ambient space $X$, for if $X'\sub X$ is a locally closed subscheme on which $\motif X$ is also a sieve, then since its topology is induced by that of $X$, it induces the same admissible opens on $\motif X$, and the same admissible coverings. Without going into details, we claim that the collection of admissible opens  and admissible coverings  yields  
a Grothendieck topology on $\motif X$, turning $\sieves\bassch$ into a Grothendieck site. 
Nonetheless,  since for each fat point $\fat$, this induces a topological 
space on $\motif X(\fat)$,  we will just pretend that we are working in a genuine topological space, and borrow the usual topological jargon.

 \begin{remark}\label{R:zarsep}
 Suppose $\bassch$ is the spectrum of    a field $\fld$. 
Unless $\fat$ is the geometric point $\bassch$ itself, the topological space $
\motif X(\fat)$ is not separated: two $\fat$-rational points $a,b\in\motif 
X(\fat)$ are inseparable \iff\ they have the same center, that is to say, \iff\ 
$\rho(\fld)(a)=\rho(\fld)(b)$, where $\rho(\fld)\colon\motif X(\fat)\to
\motif X(\fld)$ is the canonical map. 
Given an open $U\sub X$ with corresponding open 
sieve $\motif U:=\motif X\cap  U$, then, as we shall prove shortly in Theorem~\ref{T:openarc} below, we have 
$\motif U(\fat)=\inverse{\rho(\fld)} {U(\fld)}$. Therefore,  if $\rho(\fld)(a)=\rho(\fld)(b)\in U(\fld)$, then $a,b\in\motif 
U(\fat)$. 
\end{remark}

\begin{proposition}\label{P:ctugeomsect}
If $s\colon\motif Y\to \motif X$ is an \explicit\ morphism of sieves, then it is continuous in the sense that the pull-back of any open in $\motif X$ is an open in $\motif Y$.
\end{proposition}
\begin{proof}
By assumption, $s$ extends to  a morphism $\varphi\colon Y\to X$, where $\motif Y$ is a sieve on $Y$ and $\motif X$ on $X$. An open $\motif U\sub\motif X$ is of the form $\motif X\cap  U$, for some open subscheme $U\sub X$. Since $\inverse\varphi U$ is open in $Y$ and $\varphi^*{\motif U}=\motif X\cap \inverse\varphi U$, the claim follows. 
\end{proof}

  To make $\sieves\bassch$ into a topos, we need to define structure sheafs for a given sieve $\motif X$. 
 We define   presheaves $\loc_{\motif X}$ and $\gsheaf{\motif X}$ on $\motif X$, by associating to an 
open $\motif U:= \motif X\cap  U$ its $\loc_\bassch$-algebra of global sections 
 $$
 \loc_{\motif X}(\motif U):=H_0(\motif X\cap  U)\qquad\text{and}\qquad \gsheaf{\motif X}(\motif U):=\gsect{\motif X\cap  U}.
 $$
 
\begin{corollary}\label{C:structsheaf}
For each sieve $\motif X$, the presheaves $\loc_{\motif X}$ and $\gsheaf{\motif X}$ are sheaves 
(in the topos sense).
\end{corollary}
\begin{proof}
This is essentially the content of Theorem~\ref{T:globsectunion} (see also Remark~\ref{R:globsectunion}), applied to 
a (finite) admissible covering of an open $\motif U= \motif V_1\cup\dots\cup \motif 
V_s$.
\end{proof}

So we are justified in calling $\loc_{\motif X}$ the \emph{structure sheaf} 
of the sieve $\motif X$ on a  $\bassch$-scheme $X$, and $\gsheaf{\motif X}$ its \emph{\explicit\ structure sheaf}.

\subsection*{Stalks}
Let  $\motif X$ be a sieve with ambient space $X$. A closed point $P\in X$ is called a \emph{point} on $\motif X$, if the closed immersion $i_P\colon P\sub X$, viewed as a $P$-rational point, belongs to $\motif X(P)$, or equivalently, if $\func P\sub \motif X$. We define the stalk at a point $P\in\motif X$ as usual as 
the respective direct limits 
$$
\loc_{\motif X,P}:=\varinjlim \loc_{\motif X}(\motif U)\qquad\text{and}\qquad \gsheaf{\motif X,P}:=\varinjlim \gsheaf{\motif X}(\motif U)
$$
where $\motif U$ runs over all admissible opens of $\motif X$ such that $P\in \motif 
U $. Clearly, if $\motif X=\func X$ is representable, then $\loc_{\func 
X,P}=\gsheaf{\func X,P}$ is just the local ring $\loc_{X,P}$ at the closed point $P\in X$ by 
Proposition~\ref{P:globsect}.  In fact, we have:

\begin{proposition}\label{P:expstalk}
If $\motif X$ is a   sieve which is Zariski dense in $X$, and if $P$ is a point on $\motif X$, then $\gsheaf{\motif X,P}=\loc_{X,P}$.
\end{proposition}
\begin{proof}
One inclusion is immediate, so let $s\in \gsheaf{\motif X,P}$. Hence there exists an open $U\sub X$ containing $P$ such that $s\colon \motif X\cap  U\to  \affine\bassch1$ is an \explicit\ section. Since $\motif X\cap  U$ is then Zariski dense in $U$, there exists an open ambient space $\tilde U\sub U$ of $\motif X\cap U$ and a morphism $\tilde U\to \affine\bassch 1$ extending $s$.   This morphism corresponds to a global section of $\tilde U$ and hence is an element in $\loc_{\tilde U,P}=\loc_{X,P}$, since $\tilde U$ is open in $X$ containing $P$. 
\end{proof}

More generally, if $\motif X$ is a sieve on $X$ and $P$ a point on $\motif X$, then $\gsheaf{\motif X,P}=\loc_{\bar{\motif X},P}$.

\begin{lemma}\label{L:unitglobsect}
A global section $s\colon\motif X\to \affine\bassch1$ of a sieve $\motif X$ is a unit \iff\ the image of $s(P)$ does not contain zero, for any point $P\in\motif X$. If $\bassch$ is the spectrum of an \acf\ $\fld$, then this is equivalent with the image of $s(\fld)$  not containing zero. 
\end{lemma}
\begin{proof}
One direction is clear, so assume that the image of $s(P)$ does not contain zero, for any closed point $P$. Let $\fat$ be a fat point, and let $P$ be its center.  It follows from the commutative diagram
\commdiagram[sectres]{\motif X(\fat)}{s(\fat)}R{}\pi{\motif X(P)}{s(P)}{k(P)}
where $k(P)$ is the residue field of $P$, that the image of $s(\fat)$ has empty intersection with the maximal ideal of the coordinate ring $R$ of $\fat$,  since $\pi$ is just the residue map. Hence, for each $a\in\motif X(\fat)$, its image $s(\fat)(a)$ is a unit in $R$, and hence we can define $t(\fat)(a)$ to be its inverse. So remains to check that $t$ is a morphism of sieves $\motif X\to \affine\bassch1$, that is to say, a global section, and this is easy. 
\end{proof}


\begin{proposition}\label{P:stalk}
For each point $P$ on a sieve $\motif X$, the  stalk $\loc_{\motif 
X,P}$ is a local ring.
\end{proposition}
\begin{proof}
Let $X$ be an ambient space of $\motif X$. We have to show that given two  non-units   $s,t\in\loc_{\motif 
X,P}$, their sum is a non-unit as well. Shrinking $X$ if necessary, we may assume  that $ s,  t\in H_0(\motif X)$. I claim that $s(P)(i_P)$ and $t(i_P)(P)$ are both equal to zero, where $i_P\colon P\sub X$ is the closed immersion. Indeed, suppose not, say, $s(P)(i_P)\neq0$, so that  there exists an open $  U\sub X$ containing $P$ such that $s(Q)$   does not vanish on $  U(Q)$ for any closed point $Q\in U$. By Lemma~\ref{L:unitglobsect}, this implies that $s$ is a unit in $H_0(\motif X\cap   U)$ whence in $\loc_{\motif 
X,P}$, contradiction. Hence $s(P)+t(P)$ also vanishes at $i_P$ and hence cannot be a unit in $\loc_{\motif X,P}$. Note that we in fact proved that the unique maximal ideal consists of all sections $s\in\loc_{\motif X,P}$ such that $s(P)(i_P)=0$. 
\end{proof}


\section{Motivic sites}\label{s:motsite}
As before, $\bassch$ is a fixed Noetherian, Jacobson scheme. 
  A \emph{motivic\footnote{We will refer to the objects in a motivic site as `motives'. This nomenclature is   to express the fact that a motif represents 
something geometrical which is not a scheme but   ought to be something 
like a scheme, thus `motivating' our geometric treatment of it.} site} $\categ M$ over $\bassch$ is a subcategory of $\sieves 
\bassch$ which is closed under products,
  and such that  for any   $\bassch$-scheme $X$, the  restriction $
\restrict{\categ M}X$   (that is to say, the set of all $\categ M$-sieves on $X
$) forms a   lattice. In other words, if $\motif X, \motif Y\in\categ M$ 
are both sieves on a common scheme $X$, then $\motif X\cap \motif Y$   and $\motif X\cup \motif Y$ belong again to $\categ M
$ (including the minimum given by the empty set and the 
maximum given by $X$). We call $\categ M$ an \emph{\weak} motivic site, if all morphisms are \explicit. If $\categ M$ is an arbitrary motivic site, then we let $\overline{\categ M}$ be the corresponding \weak\ motivic site, obtained by only taking \explicit\ morphisms. 

Given a motivic site $\categ M$, as the sieves form locally a lattice (on each 
$\bassch$-scheme), we can define its associate \gr\ $\grot{\categ M}$ as the 
free Abelian group on symbols $\sym{\motif X}$, where $\motif X$ runs 
over all $\categ M$-sieves, modulo the  \emph{scissor relations} 
    $$
\sym{\motif X}+\sym{\motif Y}-\sym{\motif X\cup \motif Y}-\sym{\motif 
X\cap \motif Y}
  $$
  for any two  $\categ M$-sieves $\motif X$ and $\motif Y$   on a common $
\bassch$-scheme, and the \emph{isomorphism relations} 
  $$
  \sym{\motif X}-\sym{\motif Y}
  $$
  for any two   $\categ M$-sieves $\motif X$ and $\motif Y$ that are $\categ 
M$-isomorphic. We denote the image of an $\categ M$-sieve $\motif X$ in $
\grot{\categ M}$ by $\class{\motif X}$. In particular, since each 
representable functor is in $\categ M$, we may associate to any $\bassch$-scheme $X$ its class $\class X:=\class{\func X}$ in $\grot{\categ M}$. We 
define a multiplication on $\grot{\categ M}$ by the fiber product (one 
easily checks that this is well-defined): $\class{\motif X}\cdot\class{\motif 
Y}:=\class{\motif X\times \motif Y}$. Since a motivic site has the same objects as its \weak\ counterpart, we get a canonical surjective \homo\ $\grot{\overline{\categ M}}\to \grot{\categ M}$, which, however,  need not be injective, since there are more isomorphism relations in the latter \gr.

On occasion, we will encounter variants which are supported only on a subcategory of the category of all $\bassch$-schemes (that is to say, we only require the restriction of the site to one of the schemes in the subcategory to be a lattice), and we can still associate a \gr\ to it. We will refer to this as a \emph{partial motivic site}. Most motivic sites $\categ M$ will also have additional properties, like for instance being \emph{stable under push-forwards along closed immersions}, meaning that  if $i\colon Y\sub X$ is a closed subscheme and $\motif Y$ a motif in $\categ M$, then so is $i_*\motif Y$. If this is the case, then $\categ M$ is also closed under disjoint unions: given motives $\motif X$ and $\motif X'$ on $X$ and $X'$ respectively, then their disjoint union $\motif X\sqcup\motif X'$ is the union of the push-forwards $i_*\motif X$ and $i'_*\motif X'$, where $i\colon X\to X\sqcup X'$ and $i'\colon X'\to X\sqcup X'$ are the canonical closed immersions. 

\subsection*{Lefschetz class}
The class of the affine line $\affine\bassch1$ plays a pivotal role in what follows; we  call it the \emph{Lefschetz class} and denote  it by $\lef$. On occasion, we will need to invert this class, and therefore consider localizations of the form $\grot{\categ M}_\lef$.

\section{The \zariski\ \gr}

 To connect the theory of motivic sites to the classical construction, we must describe motivic sites 
whose \gr\ admits a natural \homo\ into the classical \gr\ $\grotclass\bassch$ 
(obviously, this utterly fails for the motivic site of all sieves). We will first introduce the various motives of interest in the next few sections, before we settle this issue in Theorem~\ref{T:mapclass} below. 
The smallest 
motivic site on $\bassch$ is obtained by taking for sieves on a scheme $X$ only 
the empty sieve and the whole sieve $\func X$. The resulting \gr\   has no 
non-trivial scissor relations and so we just get the free Abelian ring on 
isomorphism classes of $\bassch$-schemes. 

To define larger sites, we want to include at least closed subsieves of a scheme $X$.  Any object in the lattice generated by the closed subsieves of $X$   will be called   a 
\emph{\zariski\ motif  on $X$}. Since closed subsieves are already closed under  intersection, a \zariski\ motif on $X$ is a sieve of the form  
\begin{equation}\label{eq:zarmot}
\motif X=\func{X_1} \cup\dots
\cup\func{X_s},
\end{equation}
 where  the  $X_i$   are closed 
subschemes of $X$. 
Let us call a $\bassch$-scheme $X$ \emph{\zariski{ally} irreducible} if $\func X$ cannot be written as a finite union of proper closed subsieves. In particular, by an easy   Noetherian argument, any \zariski\ motif is the union of finitely many \zariski{ally} irreducible closed subsieves. We call a representation \eqref{eq:zarmot} a \emph{\zariski\ decomposition}, if it is \emph{irredundant}, meaning that there are no closed subscheme relations among any two $X_i$, and the $X_i$  are \zariski{ally} irreducible. Assume \eqref{eq:zarmot} is a \zariski\ decomposition, and let $\motif X=\func Y_1\cup\dots\cup\func Y_t$ be a second \zariski\ decomposition. Hence, for a fixed $i$, we have 
$$
\func X_i=\func X_i\cap \motif X=\func{(X_i\cap Y_1)}\cup\dots\cup\func{(X_i\cap Y_t)}.
$$
Since $X_i$ is  \zariski{ally} irreducible,  there is some $j$ such that $X_i=X_i\cap Y_j$, that is to say, $X_i\sub Y_j$.  Reversing the roles of the two representations, the same argument  yields some $i'$ such that $Y_j\sub X_{i'}$. Since $X_i\sub Y_j\sub X_{i'}$, irredundancy implies that these are equalities. Hence, we proved:

\begin{proposition}\label{P:zarmotrep}
A  \zariski\ decomposition is, up to order, unique.\qed
\end{proposition}

We call the $X_i$ in the (unique) \zariski\ decomposition~\eqref{eq:zarmot} the \emph{\zariski\ irreducible components} of $\motif X$.
 If $X$ has dimension zero, then it is a finite disjoint sum of fat points, its \zariski\ irreducible components. More generally, any \zariski\ motif on $X$ is a disjoint sum of  closed subsieves given by fat points, and hence is itself a closed subsieve.

To give a purely scheme-theoretic characterization of being \zariski{ally} irreducible, recall that a point $x\in X$ is called \emph{associated}, if $\loc_{X,x}$ has depth zero, that is to say, if every element in $\loc_{X,x}$ is either a unit or a zero-divisor. Any minimal point is associated, and the remaining ones, which are also finite in number, are called \emph{embedded}. The closure of an associated point is called a \emph{primary component}. 
We say that $X$ is \emph{strongly connected}, if the intersection of all primary components is non-empty, that is to say, if there exists a (closed) point generalizing to each associated point (in the affine case $X=\op{Spec}A$, this means that the associated primes  generate a proper ideal). 
For instance, if $X$ is the union of two parallel lines and one intersecting line, then it is connected but not strongly. For an example with embedded points, take the affine line with two (embedded) double points given by the ideal $(\var^2,\var\vary(\vary-1))$.  A (reduced) example where any two primary, but not all three, components meet, is given by   the `triangle' $\var\vary(\var-\vary-1)$.

\begin{proposition}\label{P:strconn}
A $\bassch$-scheme $X$ is \zariski{ally} irreducible \iff\ it is strongly connected.
\end{proposition}
\begin{proof}
It is easier to work with the contrapositives of these statements, and we will show that their negations are then also equivalent with the  existence of finitely many non-zero ideal sheafs $\mathcal I_1,\dots,\mathcal I_s\sub\loc_X$ with the property that for each closed point $x$, there is some $n$ such that $\mathcal I_n\loc_{X,x}=0$. To prove the equivalence of this with   being \zariski{ally} reducible,   assume $\func X=\func X_1\cup\dots\cup\func X_s$ for some proper closed subschemes $X_n\varsubsetneq X$. Let $\mathcal I_n$ be the ideal sheaf of $X_n$ and let $x$ be an arbitrary closed point. For each $m$, the closed immersion $\jet xmX\sub X$ is a rational point on $X$ along the $m$-th jet, whence must belong to one of the $X_n(\jet xmX)$, that is to say, $\jet xmX$ is a closed subscheme of $X_n$ by Lemma~\ref{P:morphrep}. Since there are only finitely many possibilities, there is a single $n$ such that each $\jet xmX$ is a closed subscheme of $X_n$. By Lemma~\ref{P:morphrep}, this means that $\mathcal I_n\loc_{X,x}$ is contained in any power of the maximal ideal $\maxim_x$, and hence by Krull's intersection theorem must be zero. Conversely, suppose there are non-zero ideal sheaves $\mathcal I_1,\dots,\mathcal I_s\sub\loc_X$ such that at each closed point, at least one vanishes. Let $X_n$ be the closed subscheme defined by $\mathcal I_n$, let $\fat$ be a fat point, and let $a\colon \fat\to X$ be a $\fat$-rational point. Let $x$ be the center of $\fat$, a closed point of $X$, and let $a_x\colon \loc_{X,x}\to R$ be the induced local \homo\ on the stalks, where $R$ is the coordinate ring of $\fat$. By assumption, there is some $n$ such that $\mathcal I_n\loc_{X,x}=0$, whence so is its image under $a_x$. By Lemma~\ref{L:vanci}, this implies that $a\in X_n(\fat)$. Since this holds for any rational point, $\func X$ is the union of all $\func X_n$. 

We now prove  the equivalence of the above condition with not being strongly connected.   By (global) primary decomposition (see for instance \cite[IV \S3.2] {EGA}), there exist (primary) closed subschemes $Y_n\sub X$ and an embedding $\loc_X\into \loc_{Y_1}\oplus\dots\oplus\loc_{Y_s}$, such that the underlying sets of the $Y_n$ are the primary components of $X$. Let $\mathcal J_n$ be the ideal sheaf of $Y_n$, so that the above embeddability amounts to $\mathcal J_1\cap\dots\cap\mathcal J_s=0$. If $X$ is not strongly connected, then the intersection of all $Y_n$ is empty, which means that 
\begin{equation}\label{eq:sumprim}
\mathcal J_1+\dots+\mathcal J_s=\loc_X.
\end{equation}
 Let $\mathcal I_j$ be the intersection of all $\mathcal J_m$ with $m\neq j$, and let $X_j$ be the closed subscheme given by $\mathcal I_j$. Let $x$ be any closed point. By \eqref{eq:sumprim}, we may assume after renumbering that the maximal ideal of $x$ does not contain  $\mathcal J_1$, that is to say, $\mathcal J_1\loc_{X,x}=\loc_{X,x}$. Therefore, $\mathcal I_1\loc_{X,x}=(\mathcal I_1\cap \mathcal J_1)\loc_{X,x}$, whence is zero, since $\mathcal I_1\cap \mathcal J_1=0$. Conversely, assume there are non-zero ideal sheafs $\mathcal I_1,\dots,\mathcal I_s\sub\loc_X$ such that $\mathcal I_n\loc_{X,x}=0$,  for each closed point $x$ and for  some $n$ depending on $x$. This is equivalent with the sum of all $\ann{}{\mathcal I_n}$ being the unit ideal.  Since any annihilator ideal is contained in some associated prime, the sum of all associated primes must also be the unit ideal, and hence the intersection of all primary components is empty.
\end{proof}

From the proof we learn that $\op{Spec} A$ is \zariski\ reducible \iff\ there exist finitely many proper ideals whose sum is the unit ideal and whose intersection is the zero ideal. We can even describe an algorithm which calculates its \zariski\ irreducible components. Let $0=\mathfrak g_1\cap\dots\cap\mathfrak g_n$ be a primary decomposition of the zero ideal in $A$, and assume $\mathfrak g_1+\dots+\mathfrak g_s=1$, for some $s\leq n$. Then the \zariski\ irreducible components of $X$ are among the \zariski\ irreducible components of the closed subschemes $X_i=\op{Spec}(A/\ann{}{\mathfrak g_i}$ for $i=\range 1s$.  That the $X_i$ can themselves be \zariski\ reducible, whence require further decomposition, is illustrated by the `square' with equation $\var(\var-1)\vary(\vary-1)=0$. A sufficient condition for the $X_i$ to be already \zariski{ally} irreducible is that $s=n$ and no fewer $\mathfrak g_i$  generate the unit ideal. By Noetherian induction, this has to happen eventually. In the same vein, we have:

\begin{proposition}\label{P:zarclosschemic}
Suppose $\bassch$ is Jacobson. Let $X$ be a $\bassch$-scheme, and let $X_1,\dots,X_s\sub X$ be closed subschemes with respective ideals of definition $\mathcal I_1,\dots,\mathcal I_s\sub \loc_X$. The Zariski closure of the \zariski\ motif $\mathcal X:=\func X_1\cup\dots\cup\func X_s$ is the closed subscheme with ideal of definition $\mathcal I_1\cap\dots\cap \mathcal I_s$. In particular, $\mathcal X$ is equal to its own Zariski closure \iff\ $\mathcal I_1+\dots+\mathcal I_s=\loc_X$. 
\end{proposition}
\begin{proof}
Let $Y$ be the closed subscheme with ideal of definition $\mathcal I:=\mathcal I_1\cap\dots\cap \mathcal I_s$. Since each  $X_j\sub Y$, the Zariski closure of $\motif X$ is contained in $Y$. To prove the converse, suppose $Z\sub X$ is a closed subscheme such that $\motif X\sub \func Z$. We have to show that $Y\sub Z$, so suppose not.  This mean  that $\mathcal J\loc_Y\neq 0$, where $\mathcal J$ is the ideal of definition of $Z$. By the Jacobson condition, there is a closed point $x\in Y$ such that $\mathcal J\loc_{Y,x}\neq0$, and then by Krull's Intersection Theorem, some $n$ such that $\mathcal J(\loc_{Y,x}/\maxim^n_x)\neq 0$, where $\maxim_x$ is the maximal ideal corresponding to $x$. We may write $\maxim^n_x=\mathfrak n_1\cap\dots\cap \mathfrak n_t$ as a finite intersection of irreducible ideal 
sheafs,\footnote{An ideal is called \emph{irreducible} if it cannot be written as a finite intersection of strictly larger ideals.} and then for at least one, say for $j=1$, we must have $\mathcal J(\loc_{Y,x}/\mathfrak n_1)\neq 0$. Let $\fat$ be the fat point with coordinate ring  $R:=\loc_{Y,x}/\mathfrak n_1$. By    Lemma~\ref{L:dirlimJac}, this means that the $\fat$-rational point $i$ given by the inclusion $\fat\sub Y$ does not factor through $Z$. On the other hand, by the same Lemma, we have $\mathcal IR=0$. Since  the zero ideal is irreducible, at least one of the $\mathcal I_jR$ must vanish, showing that $i$ lies in $ \motif X(\fat)$, whence by assumption in $Z(\fat)$, contradiction. The last assertion is now immediate from the previous discussion. 
\end{proof}

\begin{corollary}\label{C:globsectzar}
The global section ring of a \zariski\ motif is equal to that of its Zariski closure.
\end{corollary}
\begin{proof}
By an induction argument, we may reduce to the case that $\motif X=\func Y\cup \func Z$ where $Y,Z\sub X$ are closed subschemes (alternatively, use \eqref{eq:sheaf}). In view of the local nature of the problem, we may furthermore reduce to the case that  $X=\op{Spec}A$ is affine, so that  $Y$ and $Z$ are defined by some ideals $I,J\sub A$. In particular, the Cartesian square~\eqref{un} becomes
\commdiagram {H_0(\motif X)}{}{A/I}{}{}{A/J}{}{A/(I+J).}
However, it is easy to check that putting $A/(I\cap J)$ in the left top corner of this square also yields a Cartesian square, and hence, by uniqueness, we must have $H_0(\motif X)=A/(I\cap J)$. By Proposition~\ref{P:zarclosschemic}, the Zariski closure of $\motif X$ is $\op{Spec}(A/(I\cap J))$, proving the assertion.
\end{proof}

\begin{definition}\label{D:zarmot}
We define  the    \emph{\zariski\ motivic  site} over $\bassch$, denoted   $\funcalg \bassch$, as the full subcategory of $\sieves\bassch$ consisting  of all \zariski\ motives.     By Proposition~\ref{P:morphrep},   the category of $\bassch$-schemes fully embeds in $\funcalg 
\bassch$, and   the image of this embedding is precisely the full subcategory of 
representable \zariski\ motives. In fact, by Theorem~\ref{T:globsectsubzar},   all morphisms in $\funcalg\bassch$ are \explicit, so that $\funcalg\bassch$ is an \weak\ motivic site.
\end{definition}

\begin{lemma}\label{L:diffclass}
Any element in $\grot{\funcalg\bassch}$ is of the form $\class X-\class Y$, for some $\bassch$-schemes $X$ and $Y$.
\end{lemma}
\begin{proof}
 For any two  $\bassch$-schemes $X$ and $X'$, we have  $\class X+\class {X'}=\class{X\sqcup X'}$, where $X\sqcup X'$ denotes their disjoint union. So remains to verify that any element in $\grot{\funcalg\bassch}$ is a linear combination of classes of schemes. This reduces the problem to the class of a single \zariski\ motif $\motif X$ on a $\bassch$-scheme $X$. Hence there exist closed subschemes $X_1,\dots,X_n\sub X$ such that $\motif X=\func X_1\cup\dots\cup\func X_n$. For each non-empty $I\sub\{1,\dots,n\}$, let $X_I$ be the closed subscheme obtained by intersecting all $X_i$ with $i\in I$, and let $\norm I$ denote the cardinality of $I$. A well-known argument deduces from the scissor relations the equality
 \begin{equation}\label{eq:nsciss}
 \class{\motif X}=\sum_{\emptyset\neq I\sub\{1,\dots,n\}} (-1)^{\norm I}\class{X_I}
\end{equation}
in $\grot{\funcalg\bassch}$, proving the claim. 
\end{proof}


\begin{theorem}\label{T:classinvmot}
Suppose $\bassch$ is Jacobson. The \zariski\ \gr\ $\grot{\funcalg \bassch}$ is freely generated, as a group, by the classes of strongly  connected $\bassch$-schemes. 
\end{theorem}
\begin{proof}
Let $\Gamma$ be the free Abelian group generated by isomorphism classes $\sym X$ of  strongly connected     $\bassch$-schemes $X$, and let $\Gamma'$ be the free Abelian group generated by  isomorphism classes $\sym {\motif X}$ of     \zariski\ motives $\motif X$. I claim that the composition $ \Gamma\sub \Gamma'\onto \grot{\funcalg\bassch}$   admits an additive inverse $\delta\colon \grot{  \funcalg\bassch}\to \Gamma $.

To construct $\delta$, we will first define an additive morphism $\delta'\colon \Gamma'\to \Gamma $ which is the identity on $\Gamma $, and then argue that it vanishes on each scissor relation, inducing therefore a morphism $\delta\colon\grot{\funcalg\bassch}\to \Gamma $. It suffices, by linearity, to define $\delta'$ on an isomorphism class of a \zariski\ motif    
\begin{equation}\label{eq:Zmot}
\motif X=\func X_1\cup\dots\cup\func X_n,
\end{equation}
 given by a  \zariski\ decomposition with $X_i\sub X$ strongly connected closed subschemes. By Noetherian induction, we may furthermore assume that $\delta'$ has been defined on the isomorphism class of any \zariski\ motif on a proper closed subscheme of $X$. In particular, $\delta'\sym{X_I}$ has already been defined, where we borrow the notation from \eqref{eq:nsciss}. We therefore set
\begin{equation}\label{eq:scissdel}
\delta'\sym{\motif X}:=\sum_{\emptyset\neq I\sub\{1,\dots,n\}} (-1)^{\norm I} \delta'\sym{X_I}
\end{equation}
By Proposition~\ref{P:zarmotrep}, the \zariski\ decomposition is unique, and hence $\delta'$ is well-defined. Moreover,  if $n=1$, so that $\motif X=\func X_1$ is a \zariski{ally} irreducible closed subsieve, whence strongly connected by Proposition~\ref{P:strconn}, then its image under $\delta'$ is just $\sym{X_1}$, showing that $\delta'$   is the identity on $\Gamma $. 

Next, let us show that even if \eqref{eq:Zmot} is not irredundant,   \eqref{eq:scissdel} still holds in $\Gamma$. We can go from   such an arbitrary representation to the \zariski\ decomposition   in finitely many steps, by adding or omitting at each step one strongly connected closed subsieve contained in $\motif X$. So assume \eqref{eq:scissdel} holds for a representation \eqref{eq:Zmot}, and we now have to show that it also holds for the representation adding a $\func X_0\sub\motif X$, with $X_0$ strongly connected. Since  $X_0$ is \zariski{ally} irreducible by Proposition~\ref{P:strconn}, it must be a closed subscheme of one of the others, say, of $X_1$. Let $J$ range over all non-empty subsets of $\{0,\dots,n\}$. We arrange the  subsets containing $0$ in pairs $\{J_+,J_-\}$, so that $J_+$ and $J_-$ differ only as to whether they contain $1$ or not. Since $X_0\sub X_1$, we get $X_{J_+}=X_{J_-}$, and as $J_-$ has one element less than $J_+$, the two terms indexed by this pair in the sum \eqref{eq:scissdel} 
cancel each other out. So, in that sum, only subsets $J$ not containing $0$ contribute, which is just the value for the representation without $\func X_0$. We also have to consider the converse case, where instead we omit one, but the argument is the same. 

We can now show that \eqref{eq:scissdel} is still valid even if the $X_i$ in \eqref{eq:Zmot} are not strongly connected. Again we may reduce the problem to adding or omitting a single closed subsieve $\func X_0$. Let $\func X_0=\func Y_1\cup\dots\cup\func Y_m$ be a \zariski\ decomposition for $X_0$. We have to show that  the value of the sum in \eqref{eq:scissdel} for the representation $\motif X=\func X_0\cup\dots\cup\func X_n$ is the same as that for the representation 
\begin{equation}\label{eq:sumXY}
\motif X=\func X_1\cup\dots\cup\func X_n\cup \func Y_1\cup\dots\cup \func Y_m.
\end{equation}
 The first sum is given by
\begin{equation}\label{eq:sum0n}
\sum_{\emptyset\neq I\sub\{1,\dots,n\}} (-1)^{\norm I} \delta'\sym{X_I}+\sum_{ I\sub\{1,\dots,n\}} (-1)^{\norm I+1} \delta'\sym{X_0\cap X_I}
\end{equation}
By Noetherian induction and the fact that $\func{(X_0\cap X_I)}=\func{(X_I\cap Y_1)}\cup\dots\cup \func{(X_I\cap Y_m)}$, we have an identity 
$$
 \delta'\sym{X_0\cap X_I}=\sum_{\emptyset\neq J\sub\{1,\dots,m\}} (-1)^{\norm J} \delta'\sym{X_I\cap Y_J}
$$
for each subset $I$. Substituting this is in \eqref{eq:sum0n} yields the sum corresponding to representation \eqref{eq:sumXY} (note that $I\cup J$ ranges over all non-empty subsets of $\{1,\dots,n\}\sqcup\{1,\dots,m\}$, as required).

To obtain the induced map $\delta$, we must show next that $\delta'$ vanishes on any scissor relation. To this end, let $\motif Y=\func Y_1\cup\dots\cup\func Y_t$ be a second \zariski\ motif on $X$, again assumed to be given by its \zariski\ decomposition. Put $Z_{ij}:=X_i\cap Y_j$, so that $\motif X\cup \motif Y$ is the union of the $\func X_i$ and $\func Y_j$, whereas $\motif X\cap \motif Y$ is the union of the closed subsieves $\func Z_{ij}$. By our previous argument, we may use these respective representations to calculate $\delta'$ of the scissor relation $\sym{\motif X\cup\motif Y}+\sym{\motif X\cap\motif Y}-\sym{\motif X}-\sym{\motif Y}$. Comparing the various sums given by the respective right hand sides of \eqref{eq:scissdel}, this reduces to the following combinatorial assertion. Given finite subsets $I,J$, let us call a subset $N\sub I\times J$ \emph{dominant}, if its projections onto the first and second coordinates are both surjective; then 
$$
\sum_{N\text{ dominant}} (-1)^{\norm N}=1.
$$
We leave the details to the reader. In conclusion, we have constructed  an (additive) map $\delta\colon \grot{\funcalg\bassch}\to \Gamma$ which is the identity on $\Gamma$. On the other hand, it follows from \eqref{eq:nsciss} that $\class{\delta\class{\motif X}}=\class{\motif X}$, showing that $\delta$ is an isomorphism. 
\end{proof}

Immediately from this we get:

\begin{corollary}\label{C:classinvmot}
Suppose $\bassch$ is Jacobson. If $X$ and $Y$ are strongly connected $\bassch$-schemes, then $\class X=\class Y$ in $\grot{\funcalg\bassch}$ \iff\ $X\iso Y$.\qed 
\end{corollary}
%



\section{The sub-\zariski\ \gr}
 
To allow for additional relations, we want to include also open subsieves, or more generally, locally closed subsieves. For applications, it is more appropriate to put this in a larger context. Our point of departure is: 

\begin{lemma}\label{L:latimag}
For any  $\bassch$-scheme, the set of  its sub-\zariski\ sieves   forms a lattice. Moreover, the product of two sub-\zariski\ sieves is again sub-\zariski.
\end{lemma}
\begin{proof}
Recall that a sub-\zariski\ sieve is just an image sieve. 
If $\varphi\colon Y\to X$ and $\psi\colon Z\to X$ are morphisms of $\bassch
$-schemes, then $\fim\varphi\cap\fim\psi=\fim{\varphi\times_X\psi}$, 
where $\varphi\times_X\psi\colon Y\times_XZ\to X$ is the total morphism in the commutative square 
\commdiagram {Y\times_XZ}{}{Y}{}{\varphi}{Z}{\psi}{X}
given by base change. Likewise  $\fim\varphi\cup\fim\psi=\fim{\varphi\sqcup\psi}$, 
where $\varphi\sqcup\psi\colon Y\sqcup Z\to X$ is the disjoint union of 
the two morphisms. 

As for products, if $\varphi\colon Y\to X$ and $\varphi'\colon Y'\to X'$ are morphisms of $\bassch$-schemes, then the image of $\varphi\times_\bassch\varphi'\colon Y\times_\bassch Y\to X'\times_\bassch X$ is equal to the product $\fim\varphi\times\fim{\varphi'}$, showing that the latter is again sub-\zariski.
\end{proof}

\begin{definition}\label{D:submot}
We define the  \emph{sub-\zariski\ motivic site} $\funcsub \bassch$  as the full subcategory of $\sieves\bassch$ with objects the sub-\zariski\ 
sieves. By Theorem~\ref{T:globsectsubzar}, any morphism in this category is \explicit, so that $\funcsub\bassch$ is again an \weak\ motivic site. 
\end{definition}
Instead, we could have 
opted for a smaller site to   take  care of open coverings: define the 
\emph{motivic constructible site} $\funccon \bassch$ by taking on each 
scheme the lattice generated by locally closed subsieves (note that this is again an \weak\ motivic site). We have a natural \homo\ of \gr{s} $\grot{\funccon \bassch}\to \grot{\funcsub \bassch}$, but I do not know whether 
it is injective and/or surjective.  

\begin{example}\label{E:subfat}
The constructible site is strictly smaller than the   sub-\zariski\ one, as illustrated by the following example: let $\varphi\colon\mathfrak l_4\to\mathfrak l_2$ be the morphism corresponding to the \homo\ $R_2:=\pol\fld\xi/\xi^2\pol\fld\xi\to R_4:=\pol\fld\xi/\xi^4\pol\fld\xi$ given by $\xi\mapsto \xi^2$. Given a fat point $\fat=\op{Spec}R$, the $\fat$-rational points of $\mathfrak l_2$ are in one-one correspondence with  the elements in $R$ whose square is zero, whereas $\fim\varphi(\fat)$ is the   subset of all those that are themselves a square, in general a proper subset. As $\mathfrak l_2$ is a fat point, it has no non-trivial locally closed subsieves, showing that $\fim\varphi$ is sub-\zariski\ but not constructible. Moreover, the Zariski closure of $\fim\varphi$ is $\mathfrak l_2$.  
\end{example}

\begin{lemma}\label{L:opencomp}
If $Y$ is an open in a $\bassch$-scheme $X$, then $\func Y$ is a complete sieve on $X$.
\end{lemma}
\begin{proof}
Let $\mathfrak v\to\mathfrak w$ be a morphism of fat points. We have to show that any $\mathfrak w$-rational point $a\colon \mathfrak w\to X$  whose image under $X(\mathfrak w)\to X(\mathfrak v)$ belongs to $Y(\mathfrak v)$, itself already belongs to $Y(\mathfrak w)$. The condition that needs to be checked is that if the composition $\mathfrak v\to\mathfrak w\map a X$ factors through $Y$, then so does $a$. Let $x\in X$ be the center of $a$. Since $x$ is then also the center of the composition $\mathfrak v\to X$, it is a closed point of $Y$. Therefore, $\loc_{X,x}=\loc_{Y,x}$. Since $a$ induces a \homo\ $\loc_{X,x}\to T$, where $T$ is the coordinate ring of $\mathfrak w$, whence a \homo\ $\loc_{Y,x}\to T$, we get the desired factorization  $\mathfrak w\to Y$.
\end{proof}

This will, among other things, allow us often to reduce the calculation of rational points to the affine case.  
Let $X=X_1\cup\dots\cup X_n$ be an open cover. By Lemma~\ref{L:opencomp}, we get  $\func X=\func X_1\cup\dots\cup\func X_n$. An easy argument on scissor relations, with notation as in \eqref{eq:nsciss}, yields:\footnote{By assumption, all $\bassch$-schemes are separated, and hence the intersection of affines is again affine.}

\begin{lemma}\label{L:opencov}
If $X=X_1\cup\dots\cup X_n$ is an open covering of $\bassch$-schemes, then 
$$
 \class{X}=\sum_{\emptyset\neq I\sub\{1,\dots,n\}} (-1)^{\norm I}\class{X_I}
 $$
 in $\grot{\funcsub\bassch}$. In particular, the class of a \zariski\ motif lies in the subring  generated by classes of affine schemes.\qed
\end{lemma}

\begin{example}\label{E:projline}
As an example, let us calculate the class of the projective line $\mathbb P_\bassch^1$. It admits an open covering $X_1\cup X_2$ where $X_1$ and $X_2$ are obtained by removing respectively the origin and the point at infinity. Since $X_1\iso X_2\iso\affine\bassch 1$, we have 
$$
\class{\mathbb P_\bassch^1} =2\lef-\lef_*
$$
where $\lef_*$ denotes the class of the punctured line $L_*=X_1\cap X_2$, the affine line with the origin removed. One would be tempted to think that $\lef_*$ is just $\lef-1$, but this is false, as we shall see shortly. However, $X_1\cap X_2$ is an affine scheme, by Example~\ref{E:hyperbola},  isomorphic to the hyperbola $H\sub \affine\bassch 2$ with equation $\var\vary-1=0$ under the projection $\affine\bassch 2\to\affine\bassch 1$ onto the first coordinate. In other words, we have
\begin{equation}\label{eq:projline}
\class{\mathbb P_\bassch^1} =2\lef-\class H.
\end{equation}
\end{example}

\section{The formal \gr}\label{s:formgr}
 Let $Y\sub X$ be a closed subscheme. Recall that the   {$n$-th jet} of $X$ 
along $Y$, denoted $\jet YnX$, is the closed subscheme with ideal sheaf $
\mathcal I_Y^n$, where $\mathcal I_Y\sub\loc_X$ is the ideal sheaf of $Y$. The \emph{formal completion}  $\complet X_Y$ of $X$ along $Y$ is then the ringed space whose underlying set is equal to the underlying set of  $Y$ and whose sheaf of rings is the inverse limit of the sheaves $\loc_{\jet YnX}$. In particular, if $X=\op{Spec}A$ is affine and $I$ the ideal of definition of $Y$, then the ring of global sections of $\complet X_Y$ is equal to the $I$-adic completion $\complet A$ of $A$ (see, for instance, \cite[II.\S9]{Hart}). 
We define the \emph{completion 
sieve  along $Y$} to be the sieve $\mor \bassch\cdot{\complet X_Y}$ 
represented by the formal completion $\complet X_Y$ of $X$ at $Y$, that is 
to say, for each fat point   $\fat$, it gives the subset of all $\fat$-rational 
points $\fat\to X$  that factor through $\complet X_Y$. We will simply 
denote it by $\func{\complet X_Y}$ and call   any such presheaf  again  
\emph{pro-representable}.\footnote{\label{f:locringed}Note that $
\complet X_Y$ is no longer a scheme, but only a locally ringed space with 
values in the category of $\loc_\bassch$-algebras, and so for a (formal) scheme $Z$, 
the set $\mor\bassch  Z{\complet X_Y}$ is to be understood as the set of 
morphisms $Z\to \complet X_Y$ of locally ringed spaces with values in the 
category of $\loc_\bassch$-algebras.}  
 
 \begin{proposition}\label{P:formsieve}
For a closed subscheme $Y\sub X$, the completion sieve $\func{\complet 
X_Y}$ of $X$ along $Y$ is equal to the union of all  closed subsieves $
\func{(\jet YnX)}\sub \func X$, for $n\geq 1$.  Moreover, we have an 
identity of sieves $\func{\complet X_Y}=\func X-\func{(X-Y)}$, showing that   $\func{\complet X_Y}$ is a complete sieve.  
\end{proposition} 
\begin{proof}
The inclusion $\jet YnX\sub\func{\complet X_Y}$, for any $n$, is clear since the jets of $\complet X_Y$ (with respect to its closed point) are the same as those of the germ $(X,Y)$. 
Let $\fat$ be a fat point of length $l$. If $a\colon\fat\to \complet X_Y$ is a $\fat$-rational point, then this must already factor through $\jet YlX$, as any $l$-th power of a non-invertible section on $\fat$ is zero. Hence $\func{(\jet YlX)}\sub \func{\complet X_Y}$ have the same $\fat$-rational points,   proving the first assertion. Since $\jet YlX$ has the same underlying variety as $Y$, it lies outside the open $X-Y$, and hence $a\notin (X-Y)(\fat)$. To prove the converse inclusion in the second assertion, suppose now that $a\colon\fat\to X$ does not lie in $(X-Y)(\fat)$. In particular, the center of $a$ lies in $Y$. Let $\op{Spec}A$ be an affine open of $X$ containing the center of $a$, and let $(R,\maxim)$ be the (Artinian  local) coordinate of $\fat$. Hence $a$ induces a $\loc_U$-algebra \homo\ $A\to R$. If $I$ is the ideal locally defining $Y$ in $\op{Spec}A$, then $IR\sub\maxim$. In particular, $I^lR\sub\maxim^l=0$, showing that $a$ lies in $(\jet YlX)(\fat)$ whence in $\complet X_Y(\fat)$ by our first inclusion. The completeness of $\func{\complet X_Y}$ now follows from Lemma~\ref{L:complcompl}. 
\end{proof}

The proof of the first assertion  actually gives a stronger statement, which we formalize as follows. A sieve $\motif X$  on $X$ is called a \emph{formal motif on $X$}, if for each fat point $
\fat$, there exists a sub-\zariski\ subsieve $\motif Y_{\fat}\sub \motif X$ 
such that  $\motif Y_{\fat}(\fat)= \motif X(\fat)$ (we call the $\motif Y_\fat$ the \emph{sub-\zariski\ approximations of $\motif X$}, in spite of the fact that they are not unique).    A sub-\zariski\ motif is a trivial example of a formal motif;   the proof of Proposition~\ref{P:formsieve} shows that completion  
sieves are formal too, whose approximations, in fact, can be taken to be \zariski. More generally, we have

\begin{lemma}\label{L:appform}
If a sieve $\motif X$ on a $\bassch$-scheme $X$ has formal approximations in the sense that for each fat point $\fat$, there exists a formal subsieve $\motif Y_\fat\sub\motif X$ with the same $\fat$-rational points, then $\motif X$ itself is formal.
\end{lemma}
\begin{proof}
By assumption, there exists a sub-\zariski\ approximation $\motif Z_\fat\sub\motif Y_\fat$ with the same $\fat$-rational points, and it is now easy to check that the $\motif Z_\fat$ form a sub-\zariski\ approximation of $\motif X$.
\end{proof}

\begin{lemma}\label{L:empty}
If a formal motif $\motif X$ has no $\fat$-rational points, for some fat point $\fat$, then $\motif X$ itself is empty.
\end{lemma}
\begin{proof}
Suppose first that $\motif X=\fim\varphi$ is sub-\zariski, given by a morphism $\varphi\colon Y\to X$. In order for $\varphi(\fat)$ to be empty, $Y(\fat)$ has to be empty, whence $Y$ has to be the empty scheme by Lemma~\ref{L:dirlimJac}, proving that $\fim\varphi$ is the empty motif. Suppose now that $\motif X$ is merely formal, and let $\mathfrak w$ be an arbitrary fat point. Let  $\motif Y\sub\motif X$ be a sub-\zariski\ approximation with the same $\mathfrak w$-rational points. Since $\motif Y(\fat)\sub\motif X(\fat)=\emptyset$, we get  $\motif Y=\emptyset$, by the sub-\zariski\ case, showing that $\motif X(\mathfrak w)=\emptyset$. Since this holds for all fat points $\mathfrak w$, the assertion follows. 
\end{proof}

 Using Lemma~\ref{L:latimag}, one easily verifies that the formal motives on $X$ form again a lattice, and the product of two formal motives is again a formal motif, leading to   the 
\emph{formal motivic  site} $\funcinf \bassch$, and   its corresponding \gr\    $
\grot{\funcinf \bassch}$.

\begin{theorem}\label{T:mapclass}
Over an \acf\ $\fld $, we have   natural ring \homo{s} 
 $$
 \grot{\funcalg \fld}\to\grot{\funccon \fld}\to\grot{\funcsub \fld}\to
\grot{\funcinf \fld}\to\grotclass \fld.
 $$
\end{theorem}
\begin{proof}
Only the last of these \homo{s} requires an explanation. Given a formal motif $\motif X$, we associate to it the class of $\motif X(\fld)$ in the classical \gr\  $\grotclass\fld$. Note that by definition,  $\motif X(\fld)=\fim\varphi(\fld)$, for some morphism $\varphi\colon Y\to X$ of $\fld$-schemes. In particular, by Chevalley's theorem, $\motif X(\fld)$ is a constructible subset of $X(\fld)$ and hence its class in $ \grotclass\fld$ is well-defined. Clearly, this map is compatible with intersections, unions, and products, so that in order for this map to factor through $\grot{\funcinf\fld}$, we only have to show that it respects isomorphisms.  So assume $s\colon \motif X\to\motif Y$ is an isomorphism of formal motives. Let $\motif Z\sub\motif X$ be a sub-\zariski\ approximation of $\motif X$ with the same $\fld$-rational points.  Its push-forward $s_*\motif Z$ is isomorphic with $\motif Z$. By Theorem~\ref{T:globsectsubzar}, the restriction $\restrict s{\motif Z}$   extends to a morphism $\varphi\colon X\to Y$, where $X$ and $Y$ are some ambient spaces of $\motif Z$ and $\motif Y$ respectively. Since $\varphi(\fld)\colon X(\fld)\to Y(\fld)$ maps $\motif Z(\fld)$ bijectively onto $s_*\motif Z(\fld)$,   these two constructible subsets are isomorphic in the Zariski topology. However,  by definition of push-forward, $s_*\motif Z(\fld)$ is the image of   $\motif Z(\fld)=\motif X(\fld)$ under $s(\fld)$, that is to say, is equal to $\motif Y(\fld)$, as we needed to show. \end{proof}

\begin{theorem}\label{T:morp}
Assume $\bassch$ is Jacobson, and let $s\colon\motif Y\to \motif X$ be a morphism of sieves. If $\motif Y$ and $\motif X$ are  sub-\zariski\ (respectively,  formal) motives, then so is the graph of $s$. Moreover, the pull-back or the push-forward of a sub-\zariski\ (respectively,  formal) submotif is again of that form.
\end{theorem}
\begin{proof}
Let $X$   be an  ambient spaces of $\motif X$. Since the graph of the composition $\motif Y\to \motif X\sub \func X$ is equal to the intersection of the graph $\graf s$ of $s$ with $\motif  Y\times\motif X$, we may assume from the start that $\motif X=\func X$. 
Assume first that $\motif Y$ is sub-\zariski. By Theorem~\ref{T:globsectsubzar}, the morphism $s$ extends to a morphism $\varphi\colon Y\to X$ of $\bassch$-schemes. Let $Z\sub Y\times_\bassch X$ be the graph of this morphism, which therefore is a closed subscheme. Since $\graf s$ is equal to the intersection $\func Z\cap (\motif Y\times\func X)$, it is again sub-\zariski. 
%

Suppose next that  $\motif Y$ is merely a formal motif, and, for each fat point $\fat$, let $\motif Z_\fat\sub\motif Y$ be a sub-\zariski\ approximation.  Since $\graf s$ contains the graph of the restriction  of $s$ to $\motif Z_\fat$, and since they have the same $\fat$-rational points,  the assertion  follows from what we just proved for   sub-\zariski\ motives.  

Let $\motif Y'\sub\motif Y$ be a sub-\zariski\ or formal submotif. Since the push-forward $s_*\motif Y$ is the image of the restriction of $s$ to $\motif Y'$, we may reduce the problem to showing that $\fim s$ is respectively sub-\zariski\ or formal. The formal case follows easily, as in the previous argument, from the sub-\zariski\ one. So assume once more that $\motif Y$ is sub-\zariski, say of the form, $\fim\psi$ with $\psi\colon Z\to Y$ a morphism of $\bassch$-schemes. With $\varphi$ as above, one easily verifies that $\fim s=\fim{\varphi\after\psi}$.
 To prove the same for the pull-back, simply observe that the pull-back $s^*\motif X'$ of a submotif $\motif X'\sub\motif X$ is equal to the image of $\graf s\cap (\motif Y\times\motif X')$ under the morphism induced by the   projection $Y\times X\to Y$. The result then follows from our previous observations. 
\end{proof}


\begin{corollary}\label{C:expsectform}
Given an irreducible closed subscheme $Y\sub X$ with ideal of definition  $\mathcal I_Y$, let $\bar X$ be the closed subscheme given by the intersection of all powers $\mathcal I_Y^n$ (e.g., if $X$ is integral, this is just $X$ itself, by Krull's Intersection Theorem). We have isomorphisms of $\loc_\bassch$-algebras
$$
\gsect{ \func{\complet X}_Y}\iso \loc_{\bar X,y}
\qquad\text{and}\qquad H_0(\func{\complet X}_Y)\iso \complet \loc_{\bar X,y}
$$
where  $y$ is the generic point of $Y$.
\end{corollary}
\begin{proof}
Let $\motif X:=\func{\complet X}_Y$ be the formal motif determined by the formal completion along $Y$. We leave it to the reader to verify that the Zariski closure of $\motif X$ is equal to $\bar X$.  Replacing $X$ by $\bar X$, we may therefore assume that $\motif X$ is Zariski dense. 
An open subscheme $U\sub \bar X$ is an ambient space of $\motif X$ \iff\ $Y\sub U$, since  $\motif X(\fat)\sub Y(\fat)$, for each fat point $\fat$.  
By Corollary~\ref{C:expsectlim}, the ring of \explicit\ sections $\gsect{\motif X}$ is therefore the direct limit of all $H_0(U)$, where $U\sub \bar X$ runs over all opens containing $Y$. The latter condition is equivalent with $y\in U$, and hence this direct limit is just $\loc_{\bar X,y}$. On the other hand, since the jets $\jet YnX$ are approximations of $\motif X$, the inverse limit of the $H_0(\jet YnX)$ is equal to $H_0(\motif X)$ by Lemma~\ref{L:dirlimform}. Since $H_0(\jet YnX)=\Gamma(\loc_X/\mathcal I_Y^n,X)$,  the result follows. 
\end{proof}

\begin{example}\label{E:projlinecomp}
Let $\complet\lef$ be the class of the formal completion $\complet L$ of the affine line along the origin. It follows from Proposition~\ref{P:formsieve} that $\complet\lef=\lef-\lef_*$ (see Example~\ref{E:projline}). In particular, \eqref{eq:projline} becomes
$$
\class{\mathbb P_\bassch^1} =\lef+\complet\lef.
$$
We will shortly generalize  this in Proposition~\ref{P:projclassmot} below, but let us first construct from this an example of a non-\explicit\ morphism. Let $f(\var)$ be a power series in a single indeterminate   which is not a polynomial. The \homo\ $\pol\fld t\to \pow\fld\var$ given by $t\mapsto f$ induces a natural transformation of sieves $s_f\colon \func{\complet L}\to \func{(\affine\fld 1)}$. Since $f$ is not a polynomial, it cannot extend to a morphism of schemes. Its graph, in accordance with 
Theorem~\ref{T:morp}, is the formal motif with approximations the graphs of the  \explicit\ morphisms given by the various truncations of $f$. 

We can also use this to give a counterexample to Proposition~\ref{P:ctugeomsect} for global sections. In general, given a global section $s\colon\motif X\to \affine\bassch1$ of a formal motif $\motif X$, define a sieve $\motif X_s$ by letting $\motif X_s(\fat)$ consist of all $\fat$-rational points $a\in\motif X(\fat)$ such that $s(\fat)(a)$ is a unit, for each fat point $\fat$. Since $\motif X_s=s^*\func L_*$, where $L_*$ is the affine line without the origin, it is a formal motif by Theorem~\ref{T:morp}. 
Applied to the global section $s_f$ above,  $\func{\complet L}_{s_f}$ is the intersection of the open subsieves given by the truncations of $f$, whence not an admissible open in $\motif X$ for the Zariski topos.  In particular, $s_f$ is not continuous. Put differently, the submotives of the form $\motif X_s$ form in general a basis for a Grothendieck topology which is stronger than the Zariski one.
\end{example}

\begin{proposition}\label{P:projclassmot}
For each $n$, the class of projective $n$-space in $\grot{\funcinf\bassch} $ is given by the formula
$$
\class{\mathbb P_\bassch^n}=\sum_{m=0}^n\lef^m\cdot\complet\lef^{n-m}.
$$
\end{proposition}
\begin{proof}
Let $(\var_0:\dots:\var_n)$ be the homogeneous coordinates of $\mathbb P_\bassch^n$, and let $X_i$ be the basic open given as the complement of the $\var_i$-hyperplane. Hence every $X_i$ is isomorphic with $\affine \bassch n$ and their union is equal to $\mathbb P_\bassch^n$. Therefore, by Lemma~\ref{L:opencov}, we have
\begin{equation}\label{eq:projcovmot}
\class{\mathbb P_\bassch^n}= \sum_{\emptyset\neq I\sub\{0,\dots,n\}} (-1)^{\norm I}\class{X_I}
\end{equation}
in $\grot{\funcinf\bassch} $. 
So we need to calculate the class of each $X_I$. One easily verifies that, for $m\geq 0$, any intersection of $m$ different opens $X_i$ is isomorphic to the open $\affine \bassch {n-m}\times(L_*)^m$, where $L_*$ is the affine line $\affine\bassch 1$ minus a point. Since $\class{L_*}=\lef_*=\lef-\complet\lef$ by Proposition~\ref{P:formsieve}, the class of such an intersection is equal to the product $\lef^{n-m}(\lef-\complet\lef)^m$. Since   there are $\binomial{n+1}m$ terms with $\norm I=m$  in \eqref{eq:projcovmot}, the class of ${\mathbb P_\bassch^n}$ is equal to $g(\lef,\complet\lef)$, where
$$
g(t,u):=\sum_{m=0}^n(-1)^{m}\binomial {n+1}mt^{n-m}(t-u)^m.
$$
By the binomial theorem, $t^{n+1}-(t-u)g(t,u)=(t-(t-u))^{n+1}=u^{n+1}$, and hence 
$$
g(t,u)=\frac{t^{n+1}-u^{n+1}}{t-u}=\sum_{m=0}^n t^mu^{n-m},
$$
as we wanted to show.
\end{proof}

We conclude this section with a characterization of the complete sieves among the formal motives over an \acf\ $\fld$.

\begin{theorem}\label{T:formalcone}
Let $\fld$ be an \acf. A sieve $\motif X$ on a $\fld$-scheme $X$ is a complete formal motif \iff\ it is the cone $\cone F$ over a constructible subset $F\sub X(\fld)$.
\end{theorem}
\begin{proof}
Suppose $\motif X$ is complete and formal. By Lemma~\ref{L:relcompl}, it is of the form $\cone F$ for some $F\sub X(\fld)$. In fact, $F=\motif X(\fld)$, and hence by definition of the form $\fim\varphi(\fld)$ for some morphism $\varphi\colon Y\to X$. By Chevalley's theorem, $F$ is constructible. To prove the converse, since cones are easily seen to commute with union and intersection, and since any constructible subset of $\motif X(\fld)$ is an intersection and union of closed and open subsets, it suffices to prove that $\cone F$ is formal, whenever $F$ is a closed or an open subset. The open case follows immediately from Lemma~\ref{L:opencomp}, and in the closed case, we have $\cone F=\func{\complet X_F}$ by Proposition~\ref{P:formsieve} (note that the completion of a scheme along a subscheme only depends on the underlying variety of the subscheme, so that $\complet X_F$ is well-defined). 
\end{proof}

\section{Adjunction}\label{s:adj}
Let $\bassch$ and $\basschy$ be two Noetherian Jacobson schemes. By a  \emph{(\zariski) adjunction}  over the pair  $(\bassch,\basschy)$, we mean a pair of  functors $\eta\colon \fatpoints{\basschy}\to\fatpoints \bassch$ and $\arc{}{}\colon\categ{Sch}_\bassch\to \categ{Sch}_{\basschy}$, called respectively the \emph{left} and \emph{right adjoint},   such that we have, for each $\basschy$-fat point $\fat$ and each $\bassch$-scheme $X$, an adjunction isomorphism
\begin{equation}\label{eq:adjoin}
\Theta_{\fat,X}\colon X(\eta(\fat))=\mor\bassch{\eta(\fat)}X\iso \mor{\basschy}{\fat}{\arc{} X}=\arc{} X(\fat),
\end{equation}
which is functorial in both arguments. Whenever $\fat$ and $X$ are clear from the context, we may just denote this isomorphism by $\Theta$, or even omit it altogether, thus identifying $X(\eta(\fat))$ with $\arc{} X(\fat)$. More generally, by an (arbitrary) \emph{adjunction} we mean the same as above, except that the right adjoint now only takes values in the category of sieves, that is to say, is a functor $\categ{Sch}_\bassch\to \sieves{\basschy}$, where we identify the category of $\bassch$-schemes with its image as a full subcategory of sieves. Of course, the morphisms on the right hand side of \eqref{eq:adjoin} are now to be taken in $\sieves{\basschy}$, where the last equality is then given by Corollary~\ref{C:fatmor} (note, however, that they are all \explicit\ by Proposition~\ref{P:morphrep}). If each $\arc{} X$ is sub-\zariski, or formal, then we call the adjunction respectively \emph{sub-\zariski} or \emph{formal}.

We can formulate the adjunction property as a representability question: given a functor $\eta\colon \fatpoints{\basschy}\to\fatpoints \bassch$ and a $\bassch$-sieve $\motif X$, let $\arc\eta{\motif X}$ be the presheaf over $\basschy$ associating to a fat $\basschy$-point, the set $\motif X(\eta(\fat))$. We have adjunction when all presheaves $\arc\eta{\func X}$ are   sieves, where $X$ varies over all $\bassch$-schemes; the adjunction is (sub-)\zariski\ or formal, if each $\arc\eta{\func X}$ is respectively a  (sub-)\zariski\ or formal motif. From this perspective,  $\arc\eta{}$ is the right adjoint of $\eta$, and we simply call $\arc\eta{}$ the adjunction. We extend this to get a functor  $\arc\eta{}\colon\sieves\bassch\to \sieves{\basschy}$ as follows.  
Given a sieve $\motif X$ on a $\bassch$-scheme $X$, we   define its \emph{adjoint} $\arc\eta{\motif X}$ as the presheaf over $\basschy$ given by
$$
\arc\eta{\motif X}(\fat):=  \Theta_{\fat,X}\big(\motif X(\eta(\fat))\big)
$$
for any   $\basschy$-point $\fat$. It follows immediately from \eqref{eq:adjoin} that 
$\arc\eta{ \func X}={\arc{} X}$, and hence 
$\arc\eta{\motif X}$ is a  subsieve of $\arc{} X$.   The adjunction isomorphism~\eqref{eq:adjoin} then becomes
\begin{equation} 
  \mor{\sieves\bassch}{\eta(\fat)}{\motif X}\iso \mor{\sieves\basschy}{\fat}{\arc\eta \motif X}.
\end{equation}
 
\begin{lemma}\label{L:adjim}
If $\varphi\colon Y\to X$ is a morphism of $\bassch$-schemes, then, with $\arc{} \varphi\colon \arc{}  Y\to \arc{} X$ the induced morphism of $\basschy$-sieves, we have an equality of sieves
\begin{equation}\label{eq:adjim}
\arc\eta{\fim\varphi}=\fim{\arc{} \varphi}.
\end{equation}
In particular, sub-\zariski\ adjunctions preserve    sub-\zariski\ as well as formal motives, whereas formal adjunctions preserve formal motives.
\end{lemma}
\begin{proof}
We verify \eqref{eq:adjim} on a $\basschy$-fat point $\fat$. Functoriality of adjunction implies that we have a one-one correspondence of diagrams
\begin{equation}\label{eq:ratadj}
\xymatrix@R=5pt{&   Y\ar[dd]_{ \varphi}&&& \arc{} Y\ar[dd]^{\arc{}\varphi}\\
\eta(\fat)\ar[ur]^{  b}\ar[dr]_{  a}&\ar@{-->}@/^/[rr]^{\Theta}&&\ar@{-->}@/^/[ll]^{\inv\Theta} \fat\ar[ur]^{\tilde b}\ar[dr]_{\tilde a}&\\
&   X&&& \arc{} X}
\end{equation}
where the right triangle is in $\sieves{\basschy}$. 
So, if $\tilde a\in\fim{\arc{}\varphi}(\fat)$, then by Corollary~\ref{C:fatmor}, we can find $\tilde b$ making the right triangle in \eqref{eq:ratadj} commute. Taking the image under $\inv{\Theta_{\fat,X}}$ yields the commutative triangle on the left, showing that $\inv\Theta(\tilde a)\in\fim\varphi(\eta(\fat))$, and hence that  $\tilde a\in(\arc\eta{\fim\varphi})(\fat)$. The converse holds for the same reason, by going this time from left to right.

It then follows from Theorem~\ref{T:morp}   that the adjoint of a  sub-\zariski\ motif is again sub-\zariski, in case $\eta$ is sub-\zariski\ itself. Suppose next that $\motif X$ is formal, and, for each $\bassch$-fat point $\mathfrak w$, let $\motif Y_{\mathfrak w}\sub \motif X$ be a sub-\zariski\ approximation with the same $\mathfrak w$-rational points. For each $\basschy$-fat point $\fat$, let $\tilde{\motif Y}_\fat$ be defined as $\arc\eta{(\motif Y_{\eta(\fat)})}$. By what we just proved, $\tilde{\motif Y}_\fat\sub \arc\eta{\motif X}$ is a sub-\zariski\ submotif, and one easily verifies that both sieves have the same $\fat$-rational points, proving the last assertion for sub-\zariski\ adjunctions. The case of a formal adjunction then follows from Theorem~\ref{T:morp} and Lemma~\ref{L:appform}.
\end{proof}

\begin{proposition}\label{P:adjgr}
A formal adjunction $\arc\eta{}$ induces a \homo\ of \gr{s} $\arc\eta{}\colon\grot{\funcinf{\bassch}}\to \grot{\funcinf {\basschy}}$. If $\arc\eta{}$ is (sub-)\zariski, then this also induces \homo{s} of the corresponding (sub-)\zariski\ \gr{s}. 
\end{proposition}
\begin{proof}
By Lemma~\ref{L:adjim}, adjunction preserves    motivic sites insofar   it is (sub-)\zariski\ or formal. As it is compatible with unions and intersections, it preserves 
scissor relations, and as it is functorial, it preserves isomorphisms as well as 
products.
\end{proof}

Before we describe some important instances in which we have adjunction, with applications   discussed in  \S\S\ref{s:arc} and \ref{s:defarc}, we give an example of a formal adjunction.

\begin{example}\label{E:formadj}
Given a fat point $\fat$ over an \acf\ $\fld$, and $r\geq 2$, let $\Upsilon(\fat):=\Upsilon_r(\fat)$ be the fat point with coordinate ring $\fld+\maxim^r\sub R$, where $(R,\maxim)$ is the Artinian local ring corresponding to $\fat$. Note that we have a dominant morphism  $\fat\to \Upsilon(\fat)$. For simplicity, let us take $r=2$. For fixed  $n$,  let $\mathfrak l:=\mathfrak l_n$  be the $n$-th jet of a point on a line, with coordinate ring $S:=\pol\fld\xi/\xi^n$. For each   $l$, let $\mathfrak w_l$ be the fat point in $\affine\fld{2l}$  with ideal of definition generated by all $\xi_i^l$ and $Q^n$, where 
$$
Q:=\xi_1\xi_2+\xi_3\xi_4+\dots+\xi_{2l-1}\xi_{2l}.
$$
 Let $\motif Y_l$ be the image sieve of the morphism $\varphi_l\colon \mathfrak w_l\to \mathfrak l$ induced by $\xi\mapsto Q$. I claim that $\arc \Upsilon{\mathfrak l}$ is approximated by the $\motif Y_l$. To this end, fix a fat point $\fat$ with coordinate ring $(R,\maxim)$ and let $l$ be its length. A $\Upsilon(\fat)$-rational point $a\in \mathfrak l(\Upsilon(\fat))$ is completely determined by the image, denoted again $a$, of $\xi$ in $\fld+\maxim^2$. Since $a^n=0$, we must in particular have $a\in\maxim^2$ (note that $\maxim^2$ is the maximal ideal of $\Upsilon(\fat)$), and hence can be written as $a=b_1b_2+\dots+b_{2l-1}b_{2l}$, for some $b_i\in \maxim$. Since $b_i^l=0$ and $a=Q\rij b{2l}$, the assignment $\xi_i\mapsto b_i$  induces a morphism $\fat\to \mathfrak w_l$ which factors through $\varphi_l$. In other words, $a\in \motif Y_l(\fat)$. Conversely, since $Q$ is quadratic, any $\fat$-rational point factoring through $\varphi_l$ must extend to $\Upsilon(\fat)$.
 
 Presumably, this argument should extend to any fat point other than $\mathfrak l$ and any power $r\geq 2$. To extend this to higher dimensional schemes, we face the problem that a rational point can be given by non-units. This forces us to be able to single out the  field elements inside an Artinian local ring $R$. In \ch\ $p$, this can be done: the elements of $\kappa\sub R$ are precisely the $p^l$-th powers. Using this, a slight modification of the above argument, then yields $\arc\Upsilon{\affine\fld 1}$ as a formal motif: in the above, replace $\mathfrak w_l$  by $\affine{\mathfrak w_l}1$ and $\motif Y_l$ by the  image of the morphism $\affine{\mathfrak w_l}1\to\affine\fld 1$ given by $\xi\mapsto \xi_0^{p^l}+Q$.
 It seems likely that we can again   extend this argument to arbitrary schemes and arbitrary $r\geq 2$. 
\end{example}

\subsection{Augmentation}
Let $f\colon\basschy\to\bassch$ be a morphism of Noetherian Jacobson schemes. Via $f$, any $\basschy$-scheme $Y$ becomes a $\bassch$-scheme, and to make a notational distinction between these two scheme structures, we denote the latter by $f_*Y$.  We will show that $f_*$ constitutes a left adjoint, where the corresponding right adjoint is   given by base change: given a $\bassch$-scheme $X$, we set 
$$
f^*X:= \basschy\times_\bassch X.
$$

\begin{theorem}\label{T:bcgr}
If $f\colon\basschy\to\bassch$ is a morphism of finite type of Noetherian Jacobson schemes, then $f_*$ is a functor from $\fatpoints{\basschy}$ to  $\fatpoints\bassch$, and as such, it is the left adjoint of $f^*$. The corresponding adjunction  associates to a $\bassch$-sieve $\motif X$  on a $\bassch$-scheme $X$,     the $\basschy$-sieve  $\arc {f_*}{\motif X} $ on $f^*X$, inducing  natural ring \homo{s} on the respective \gr{s}, namely $\arc {f_*}\colon\grot{\funcalg{\bassch}}\to 
\grot{\funcalg {\basschy}}$, $\arc {f_*}\colon\grot{\funcsub{\bassch}}\to \grot{\funcsub 
{\basschy}}$, and $\arc {f_*}\colon\grot{\funcinf{\bassch}}\to \grot{\funcinf {\basschy}}$.
\end{theorem}
\begin{proof}
Let  $\fat$ be a $\basschy$-fat point   and let $y$ be its center, that is to say,  the closed point on $\basschy$ given as the image  under the structure morphism $\fat\to   \basschy$. By the 
generalized 
Nullstellensatz (\cite[Theorem 4.19]{Eis}), the image $x:=f(y)$ is a closed point on $\bassch$, and the residue field extension $\res x\sub \res y$ is finite. As $\res y\sub R/\maxim$ is also finite, $f_*\fat$ is   a $\bassch$-fat point, 
proving the first assertion.
The adjunction of $f_*$ and $f^*$ are well-known (and, in any case, easily checked; see, for instance \cite[Chapter II.5]{Hart}, but note that left and right are switched there since they are formulated in the dual category of sheaves), proving that $\arc {f_*} X = f^*X$.  The last statement follows from Proposition~\ref{P:adjgr}.
\end{proof}

\begin{remark}\label{R:fatemb}
Although $f_*\colon \fatpoints {\basschy}\to \fatpoints \bassch$  is an embedding 
of categories, it is, however,   not full: so are the closed subschemes in $\affine\fld2$ defined by the ideals
$(\var^2,\vary^3)$ and $(\var^3,\vary^2)$ isomorphic as 
fat $\fld$-points, but not as fat $\pol\fld\var$-points. Nonetheless,  $
\fatpoints {\basschy}$ is cofinal in $\fatpoints\bassch$, or, in the 
terminology of \S\ref{s:limpt} below, both have the same universal point. 
\end{remark}

\subsection{Diminution}
Let $f\colon \basschy\to \bassch$ be a finite and faithfully flat morphism of Noetherian Jacobson schemes. As opposed to the previous section, we will now consider  $f^*$ as a left adjoint. For technical reasons (see Remark~\ref{R:nonloc} below for how to circumvent these), we   make the following additional assumptions:
\begin{equation}\label{eq:loccond}
\parbox{3.7in}{$\bassch$ is of finite type over an \acf\ $\fld$  and  $f$  induces an isomorphism    on the underlying varieties.}
\tag{$\dagger$}
\end{equation}
The second condition implies that for any closed point   $x\in \bassch$ there is a unique closed point $y\in\basschy$ lying above it, and hence the closed fiber $\inverse fx$ is a local scheme. Under these assumptions, the base change $f^*\fat$ of a $\bassch$-fat point $\fat$ is a $\basschy$-fat point. Indeed, since the problem is local, we may assume $\bassch=\op{Spec}\bas$ and $\basschy=\op{Spec}\basy$ are affine. 
Let $(R,\maxim)$ be the coordinate ring of $\fat$, and let $\pr:=\bas\cap \maxim$ be the induced maximal ideal of $\bas$. The coordinate ring of $f^*\fat$ is then $S:=R\tensor_\bas\basy$. By base change, $S$ is finite (and flat) over $R$, whence in particular Artinian. By base change, $S$ is also finite as a $\basy$-module.  Since $\maxim$ is nilpotent, any maximal ideal of $S$ must contain $\maxim S$. Since $S/\maxim S=R/\maxim\tensor_{\bas/\pr}\basy/\pr\basy$, since $\bas/\pr=\fld$ by the Nullstellensatz, and since $R/\maxim$ is a finite extension of the latter, whence trivial, $S/\maxim S=\basy/\pr\basy$ is local by   assumption~\eqref{eq:loccond}, showing that $S$ itself is an Artinian local ring, thus proving the claim.

\begin{theorem}\label{T:flatrestr}
If $f\colon \basschy\to \bassch$ is a finite and faithfully flat morphism  satisfying \eqref{eq:loccond}, then $f^*$ is the left adjoint of a \zariski\ adjunction, inducing on the \gr{s} natural ring \homo{s} $\arc {f^*}\colon\grot{\funcalg{\basschy}}\to 
\grot{\funcalg {\bassch}}$, $\arc {f^*}\colon\grot{\funcsub{\basschy}}\to \grot{\funcsub 
{\bassch}}$, and $\arc {f^*}\colon\grot{\funcinf{\basschy}}\to \grot{\funcinf {\bassch}}$.

More precisely,  for any $\basschy$-scheme $Y$, there   exists a $\bassch$-scheme $\arc{f^*}Y$ and a canonical morphism $\rho_Y\colon f^*(\arc{f^*}Y)\to Y$ of $\basschy$-schemes, such that, for any $\bassch$-fat point $\fat$, the map sending a $\fat$-rational point $a\colon\fat\to \arc{f^*}Y$ to the $f^*\fat$-rational point $\rho_Y\after f^*a\colon f^*\fat\to Y$, induces an isomorphism $(\arc{f^*}Y)(\fat)= Y(f^*\fat)$. 

If $Z\sub Y$ is   a closed  immersion, then so is the canonical morphism $\arc{f^*}Z\to \arc{f^*}Y$.  
\end{theorem}
\begin{proof}
Since $f$ is finite and flat, $\basschy$ is locally free over $\bassch$. Since we may  construct each $\arc{f^*}Y$ locally and then, by the uniqueness of the universal property of adjoints, glue the pieces together, we may assume that $Y=\op{Spec}B$, $\bassch=\op{Spec}\bas$, and $\basschy=\op{Spec}\basy$ are affine, and  that $\basy$ is free over $\bas$ (in all applications, we will already have global freeness anyway). 
Let     $\alpha_1,\dots,
\alpha_l$ be a basis of $\basy$ over $\bas$.
Write $B:=\pol \basy\var/\rij hs\pol 
\basy\var$, for some polynomials $h_i$ over $\basy$. Let $\tilde \var$ 
be a new tuple of variables and define a generic tuple of arcs 
$$
\genarc\var:=\tilde\var_1\alpha_1+\dots+\tilde\var_l\alpha_l
$$
in $\pol \mu{\tilde\var}$. Given any $g\in \pol \basy\var$, let $\tilde g_j\in
\pol \bas{\tilde\var}$ be defined by the expansion
\begin{equation}\label{eq:genarcg}
g(\genarc x) =\sum_{j=1}^l \tilde g_j(\tilde\var)\alpha_i.
\end{equation}
Applying \eqref{eq:genarcg} to the $h_i$, we get   polynomials $\tilde h_{ij}$ over $
\bas$ and we let $A$ be the residue ring of $\pol\bas{\tilde\var}$ 
modulo the ideal generated by all these $\tilde h_{ij}$. I claim that $
X:=\op{Spec}  A$ represents $\arc{f^*}Y$.  It follows from \eqref{eq:genarcg} that the map $\var\mapsto \genarc \var$ yields a $\basy$-algebra \homo\ $B\to f^*A$, where $f^*A:=A\tensor_\bas\basy$ is the base change, and hence a $\basy$-morphism $\rho_Y\colon f^*X\to Y$. 
Fix a $\bas$-fat point $\fat$, and a $\fat$-rational point $a\colon \fat\to X$. By base change, we get a $\basy$-algebra \homo\ $f^*\fat\to f^*X$ which composed with $\rho_Y$ induces a $f^*\fat$-point $\Theta (a)\colon f^*\fat\to Y$. 
To prove that the map $a\mapsto \Theta( a)$  establishes an adjunction isomorphism, we construct its converse. 
Given an $f^*\fat$-rational point $b\colon f^*\fat\to Y$, let $B\to R\tensor_\bas\basy$ be the corresponding 
$\basy$-algebra \homo, where $R$ is the coordinate ring of $\fat$. The latter \homo\  is  uniquely determined by a tuple $\tuple u$ in $R\tensor_\bas\basy$ such that  all $h_i(\tuple u)=0$. 
Expanding this tuple as
\begin{equation}\label{eq:expbas}
\tuple u=\tilde{\tuple u}_1 \alpha_1+\dots+\tilde{\tuple u}_l\alpha_l 
\end{equation}
   yields a (unique) tuple $\tilde{\tuple u}:=(\tilde{\tuple u}_1, \dots,\tilde{\tuple u}_l)$ over $R$ such that all $\tilde h_{ij}(\tilde {\tuple u})=0$,  determining, therefore,  a $\bas$-algebra \homo\ $A\to R$, whence a $\bas$-morphism $\Lambda (b)\colon \fat\to X$. So remains to verify that $  \Lambda $ and $  \Theta$ are mutual inverses. Starting with the $f^*\fat$-rational point $b$, we get the $\fat$-rational point $\Lambda( b)$, which in turn induces the   $f^*\fat$-rational point $\Theta (\Lambda(b))$, given as the composition $\rho_Y\after f^*\Lambda (b)$. The latter corresponds  by \eqref{eq:expbas} to the $\basy$-algebra \homo\ $B\to f^*A\to f^*R$    given by $\var\mapsto \genarc\var\mapsto \tuple u$, showing that $\Theta (\Lambda(b))=b$.  If, on the other hand, we start with the $\fat$-rational point $a$, given by $\tilde \var\mapsto \tilde {\tuple u}$, we   get the $f^*\fat$-rational point $\Theta (a)$,  given by $\var\mapsto \tuple u$, where $\tuple u$ is as in \eqref{eq:expbas}. Hence $\Lambda(\Theta(a))$ is given by $\tilde\var\mapsto \tilde{\tuple u}$, that is to say, is equal to $a$, as we needed to show. 
   
   To prove the last assertion, assume   that $Z$ is a closed subscheme of $Y$, so that its coordinate ring is of the form $B/(h_{s+1},\dots,h_t)B$ for some additional polynomials $h_i\in \pol\basy\var$. Hence $\arc{f^*}Z$ is the closed subscheme of $\arc{f^*}Y$ given by the $\tilde h_{ij}$ for $s<i\leq t$.
\end{proof}

Immediately from the above proof, by taking $Y=f^*X$, we have the following result,  which we will use in the next section:

\begin{corollary}\label{C:flatarc}
If  $f\colon \basschy\to \bassch$ is a finite and faithfully flat morphism  satisfying \eqref{eq:loccond}, then we have for each $\bassch$-scheme $X$, a canonical $\bassch$-morphism $\rho_X\colon \arc{f^*}{f^*X}\to X$. If  $Z\sub X$ 
is   a closed  immersion, then so is the canonical morphism $\arc{f^*}{f^*Z}\to \arc{f^*}{f^*X}$.  \qed
\end{corollary}

\begin{remark}\label{R:nonloc}
Without assumption~\eqref{eq:loccond}, the pull-back  of a $\bassch$-fat point $\fat$ is only a zero-dimensional $\basschy$-scheme, and hence a disjoint sum of $\basschy$-fat points $f^*\fat=\mathfrak w_1\sqcup\dots\sqcup\mathfrak w_s$. We can then still make sense of $Y(f^*\fat)$, as the disjoint union $Y(\mathfrak w_1)\sqcup\dots\sqcup Y(\mathfrak w_s)$, and the adjunction condition  then becomes that this must be equal to $(\arc{f^*}Y)(\fat)$. Since nowhere in  the above proof we used that $f^*R$ is local, we therefore can omit condition~\eqref{eq:loccond} from the statement of Theorem~\ref{T:flatrestr}. 
\end{remark}

We have the following commutation rule for adjunctions in a Cartesian square:

\begin{theorem}[Projection Formula]\label{T:projform}
Let $f\colon \basschy\to \bassch$ be a finite and faithfully flat morphism of Noetherian Jacobson schemes  satisfying \eqref{eq:loccond},   let $u\colon \tilde\bassch\to\bassch$ be a morphism of finite type, and let 
\commdiagram{\tilde\basschy}{\tilde f}{\tilde\bassch}{\tilde u}u\basschy f\bassch
be the base change diagram, where  
$\tilde\basschy:=\basschy\times_\bassch\tilde\bassch$. We have an identity of adjunctions
$$
\arc{\tilde f^*}{}\arc{\tilde u_*}{}=\arc{u_*}{}\arc{f^*}{}
$$
from $\basschy$-sieves to $\tilde\bassch$-sieves.
\end{theorem}
\begin{proof}
Note that $\tilde f$ is again finite and faithfully flat, satisfying \eqref{eq:loccond}, so that the diminution $\arc{\tilde f^*}$ makes sense. To prove the identity we   have to check it on  each $\basschy$-sieve $\motif Y$ and each $\tilde\bassch$-fat point $\fat$, becoming
$$
(\arc{\tilde f^*}{\arc{\tilde u_*}{\motif Y}})(\fat)=\motif Y(\tilde u_*\tilde f^*\fat)\overset?=\motif Y(f^*u_*\fat)=(\arc{u_*}{\arc{f^*}{\motif Y}})(\fat).
$$
But one easily verifies that we have an equality of $\basschy$-fat points
$$
\tilde u_*\tilde f^*\fat=f^*u_*\fat
$$
concluding the proof of the theorem.
\end{proof}

\subsection{Frobenius transform}
Assume for the remainder of this section that the base ring is a field $\fld$ of \ch\ $p>0$. 
Let us denote the Frobenius \homo\ $a\mapsto a^p$ on a $\fld$-algebra $A$   by $\frob$, or in case we need to specify the ring by $\frob_A$, so that we have in particular a commutative diagram
\commdiagram [frobalg]\fld {\frob_\fld}\fld{}{}A {\frob_A}{A.}
Due to the functorial nature, we can glue these together and hence obtain on any $\fld$-scheme $X$ a corresponding endomorphism $\frob_X$. 

Diagram~\eqref{frobalg} implies that $\frob_A$ is not a $\fld$-algebra \homo. To overcome this difficulty, we assume $\fld$ is perfect, so that $\frob $ is an isomorphism on $\fld$. To make \eqref{frobalg} into a $\fld$-algebra \homo, we must view the second copy of $A$ with a different $\fld$-algebra structure, namely, the one inherited from the composite \homo\ $\fld\map \frob \fld \to A$. Several   notational devices have been proposed (see for instance \cite[Chapter IV, Remark 2.4.1]{Hart} or \cite[Chapter 8.1.c]{Ueno3}), but we will use the one already introduced in the previous section: the push-forward of $A$ along $\frob $ will be denoted $\frob_*A$. In other words, $\frob_*A$ is $A$ with its $\fld$-action  given by $u\cdot a=u^pa$. Since $\fld$ is perfect, $A\iso \frob_*A$ as rings, and in many instances, even as $\fld$-algebras. 
In particular,  \eqref{frobalg} yields a $\fld$-algebra \homo\ $A\map \frob \frob_*A$, called the \emph{$\fld$-linear Frobenius}. 
The image of the $\fld$-linear Frobenius \homo\ $A\map \frob \frob_*A$ is the subring of $A$ consisting of all $p$-th powers, and we will simply denote it by $\frob A$ (rather than the more common $A^p$, which might lead to confusions with Cartesian powers). Hence, pushing forward the inclusion \homo\ $\frob A\sub A$ gives a factorization of  the $\fld$-algebra \homo\ $\frob $ as $A\onto \frob_* \frob A \sub \frob_*A$, where the first \homo\ is an isomorphism \iff\ $A$ is reduced. For instance, if $A=\pol\fld \var$, then $\frob A=\pol\fld{\var^p}$, so that this factorization is given by the sequence of $\fld$-algebra \homo{s}
\begin{equation}\label{eq:polfrob}
\xymatrix@R=5pt{\pol\fld \var \ar[r]^-\iso&  \frob_*\pol\fld{\var^p}\ar@{}|-\sub[r]&\frob_*\pol\fld\var\ \ar[r]^-\iso&\pol\fld\var\\
&& h\ar@{|->}[r]^\sigma&\tilde h \\
g\ar@{|->}[r]&g^p\ar@{|->}[rr]&&g(\var^p),
}
\end{equation}
where $\tilde h$ is obtained from $h$ by replacing each coefficient with its (unique) $p$-th root.
So, from this we can calculate $\frob_*A$ for $A$ of the form $\pol\fld\var/\rij fs\pol\fld\var$ as 
$$
\frob_*A\iso \pol\fld\var/\rij {\tilde f}s\pol\fld\var,
$$
with $\tilde f_i=\sigma(f_i)$ as in \eqref{eq:polfrob}. 
The $\fld$-linear Frobenius $A\to \frob_*A$ is then the induced \homo\ by the composite map  $g\mapsto g(\var^p)$ from \eqref{eq:polfrob}. 

Similarly,   viewing $X$ as a $\fld$-scheme via the composition $X\to \op{Spec}\fld\map {\frob}\op{Spec}\fld$,  it will be denoted by $\frob_*X$, yielding a morphism of $\fld$-schemes $\frob_X\colon \frob_*X\to X$, called the \emph{$\fld$-linear Frobenius}. Its scheme-theoretic closure will be denoted by $\frob X$, so that we have a dominant morphism $X\to \frob X$, yielding a factorization
\begin{equation}\label{eq:frobsch}
\frob_X\colon \frob_*X\to   \frob_* \frob X\sub X
\end{equation}
 of $\frob_X$, where the closed immersion $\frob_* \frob X\sub X$ is the identity \iff\ $X$ is a variety. In particular, $\frob X$ is the Zariski closure of $\fim{\frob_X}$ in $X$.
 
 We could view $\frob_\fld$ as an automorphism of the base to get by Theorem~\ref{T:bcgr} an adjunction pair $(\frob^*_\fld,\frob_{\fld*})$. However, since $X$ and $\frob_{\fld*}\frob_\fld^*X$ are isomorphic 
  as $\fld$-schemes, 
this merely induces an action of the Frobenius. For the same reason, diminution does not induce any interesting endomorphism on the \gr. 
Instead we take a relative point of view. To a morphism $\varphi\colon Y\to X$ of $\fld$-schemes, we can associate two commutative squares; the base change and the Frobenius square. Combined into a single commutative diagram of $\fld$-morphisms, we have
$$
\xymatrix@C=50pt{\frob_* Y\ar@/^/[rrd]^{\frob_Y}\ar[rd]\ar@/_/[rdd]_{\frob_*\varphi}\\
&\frob_X^*Y\ar[r]_{\bc{\frob_X}Y}\ar[d]&Y\ar[d]^\varphi\\
&\frob_*X\ar[r]_{\frob_X}&X
}
$$
where $\frob_X^*Y:=\frob_*X\times_XY$ is the pull-back of $Y$ along $\frob_X$, called the \emph{Frobenius transform of $Y$ in $X$}, and where the canonical projection $\bc{\frob_X}Y\colon \frob_X^*Y=\frob_*X\times_XY\to Y$ is called the \emph{relative Frobenius on $Y$ over $X$}. 
In case $\varphi$ is a closed immersion, the natural morphism $\frob_*Y\to \frob_X^*Y$ is then also a closed immersion. We can calculate it explicitly in case $X=\op{Spec}A$ is affine and $Y$ is defined by the ideal $I\sub A$. Traditionally, one denotes the ideal generated by the image of $I$ under the Frobenius $\frob_A$  by $I^{[p]}$; it is the ideal generated by all $f^p$ with $f\in I$. With this notation, we have 
$$
\frob_X^*Y=\frob_*(\op{Spec}A/I^{[p]}).
$$ 
 In particular,    applying  $\sigma$ from \eqref{eq:polfrob} to the previous isomorphism in case $X$ is affine space, we get:
 
 \begin{corollary}\label{C:relfrobpol}
If $Y$ is the closed subscheme of $\affine\fld n$ with ideal of definition $\rij fs$, then $\frob_{\affine\fld n}^*Y$ is the closed subscheme of $\affine\fld n$ with ideal of definition $(f_1(\var^p),\dots, f_s(\var^p))$, and the relative Frobenius $\bc{\frob_{\affine\fld n}}Y$ is the map induced by $\var\mapsto \var^p$. \qed
\end{corollary}

 The assignment $\fat\mapsto \frob\fat$ constitutes a functor on $\fatpoints\fld$, which will play the role of left adjoint. However, in this case, the adjunction will only be sub-\zariski, via the following   right adjoint. For each $\fld$-scheme $Y$, we   define a sub-\zariski\ motif $\motif F_Y$, called its \emph{Frobenius motif}. In order to do this, we will work locally:  show that it is a right adjoint locally, and then deduce its uniqueness and existence, as well as right adjointness, globally. So let $Y$ be affine, say, a closed subscheme of $\affine\fld n$, and let $\bc{\frob_{\affine\fld n}}Y\colon \frob_{\affine\fld n}^*Y \to Y$ be the corresponding relative Frobenius. Set $\motif F_Y:=\fim{\bc{\frob_{\affine\fld n}}Y}$, so that it is a sub-\zariski\ motif on $Y$. To see that this is independent from the choice of closed immersion, we prove the adjunction formula
 \begin{equation}\label{eq:adjfrob}
 Y(\frob\fat)=\motif F_Y(\fat)
\end{equation}
  for any fat point $\fat$. More precisely, the canonical (dominant) morphism $\fat\to\frob\fat$ induces a map $Y(\frob\fat)\to Y(\fat)$. By Lemma~\ref{L:domfat} it is injective, and we want to show that its image is $\motif F_Y(\fat)$.  
  Let $\rij fs\pol\fld\var$ be the ideal defining $Y$. By Corollary~\ref{C:relfrobpol}, the Frobenius transform $ \frob_{\affine\fld n}^*Y$ is given by the ideal $(f_1(\var^p),\dots, f_s(\var^p))\pol\fld\var$. An $\frob\fat$-rational point $a$ in $Y$ corresponds to a $\fld$-algebra \homo\ $A\to \frob R$, where $R$ is the coordinate ring of $\fat$, and hence to a solution of $f_1=\dots=f_s=0$ in $R$ of the form $\tuple r^p$. The image $a'$ of $a$ in $Y(\fat)$ corresponds to the composition $A\to \frob R\sub R$. Since $\tuple r$ is a solution in $R$ of the equations defining $ \frob_{\affine\fld n}^*Y$, it induces a $\fat$-rational point $b\colon \fat\to  \frob_{\affine\fld n}^*Y$ such that $a'=(\bc{\frob_{\affine\fld n}}Y)(\fat)(b)$, proving that $a'\in \motif F_Y(\fat)$. Conversely, by reversing these arguments, we see that any such $\fat$-rational point is induced by a $p$-th power in $R$, and hence comes from a $\frob\fat$-rational point. This concludes the proof of \eqref{eq:adjfrob} when $Y$ is affine, and proves in particular that $\motif F_Y$ does not depend  on the choice of closed immersion. For arbitrary $Y$, let $Y_1,\dots,Y_m$ be an open affine covering. For each $Y_i$ and each intersection $Y_i\cap Y_j$, we have an equality \eqref{eq:adjfrob}. Hence we may glue all pieces together to obtain a sub-\zariski\ motif $\motif F_Y$ satisfying \eqref{eq:adjfrob}. In particular, in view of Proposition~\ref{P:adjgr}, we proved:
 
 \begin{theorem}\label{T:frobadj}
 The functors $\fat\mapsto \frob\fat$ and $Y\mapsto \motif F_Y$ constitute a sub-\zariski\   adjunction. In particular,  we get induced endomorphisms  $\arc\frob{}$ on $\grot{\funcsub{\fld}}$  and $ \grot{\funcinf \fld}$.\qed
\end{theorem}

Unraveling the definitions, the action of this adjunction on a motif $\motif Y$ on a scheme $Y$ is given by
$$
\arc\frob{\motif Y}=\motif Y\cap \motif F_Y.
$$
Moreover, if $Y=\affine\fld n$, then $\motif F_Y$ is just $\fim{\frob_{\affine\fld n}}$, the image of the $\fld$-linear Frobenius. Therefore, if a motif  $\motif Y$ has an ambient space which is  affine, we may take it to be an affine space $\affine\fld n$,  so that
$$
\arc\frob{\motif Y}=\motif Y\cap \fim{\frob_{\affine\fld n}}.
$$
 
\subsection{Families of motives}\label{s:fammot}
Let $s\colon\motif Y\to \motif X$ be a  \explicit\ morphism of $\bassch$-sieves (see Remark~\ref{R:nongeom} below for the non-\explicit\ case). Hence, we can find ambient spaces $Y$ and $X$ of $\motif Y$ and $\motif X$ respectively, and a morphism $\varphi\colon Y\to X$ of $\bassch$-schemes extending $s$. We explain now how we may view $s$ as a family of $\bassch$-sieves, by associating to each $\bassch$-morphism $a\colon\bassch\to X$, a $\bassch$-sieve $\motif Y_a$ as follows. We may view $\motif Y$ as an $X$-sieve via $\varphi$, and, in accordance with previous notation, we denote this $X$-sieve by $\varphi_*\motif Y$. Using $a$ to as augmentation map, we define
$$
\motif Y_a:=\arc{a_*}{\varphi_*\motif Y},
$$
called the \emph{specialization of $\motif Y$ at $a$.}
By Theorem~\ref{T:bcgr}, this is a sieve on the base change $\arc{a_*}Y=a^*Y=\bassch\times_XY$. To see that the specialization $\motif Y_a$  is independent from the choice of ambient space $Y$, we simply observe that 
\begin{equation}\label{eq:spec}
\motif Y_a(\fat)=\bassch(\fat)\times_{\motif X(\fat)}\motif Y(\fat)=\{(r,b)\in 
\bassch(\fat)\times\motif Y(\fat)\mid a(\fat)(r)=s(\fat)(b)\}
\end{equation}
as a subset of $\bassch(\fat)\times_{X(\fat)}Y(\fat)$, for any $\bassch$-fat point $\fat$. Immediately from Theorem~\ref{T:bcgr}, we have:

\begin{proposition}\label{P:specmot}
The specialization of a \zariski, sub-\zariski, or formal motif is again of the 
same type. \qed
\end{proposition}

\begin{remark}\label{R:nongeom}
We can even apply this theory to a non-\explicit\ morphism $s\colon\motif Y\to \motif X$ of formal motives. Indeed, let $\motif Z_\fat\sub\motif Y$ be sub-\zariski\ approximations of $\motif Y$. By Theorem~\ref{T:globsectsubzar}, the restriction $\restrict s{\motif Z_\fat}$ is \explicit. Hence, given $a\colon \bassch\to X$, the specialization $(\motif Z_\fat)_a$ is sub-\zariski\ by Proposition~\ref{P:specmot}. Using \eqref{eq:spec}, it is not hard to show that these specializations $(\motif Z_\fat)_a$ are approximations of $\motif Y_a$ (as defined by the right hand side of \eqref{eq:spec}), showing that the latter is formal too. 
\end{remark}

\section{Arc schemes}\label{s:arc}
From now on, our base scheme will (almost always) be an \acf\ $\fld$. 
Fix a fat point $\fat$   and let $j\colon\fat\to \op{Spec}\fld$ be its   structure morphism. Clearly, it is flat and finite and satisfies condition \eqref{eq:loccond}, and so both augmentation and diminution with respect to $j$ are well-defined. We define the \emph{arc functor} of $\fat$, as a double adjunction\footnote{See \S\ref{s:defarc} below for the corresponding single adjunction.}
$$
\arc\fat{}:=\arc{j^*}{}\after\arc{j_*}{}.
$$
In other words, given a motif $\motif X$ on a $\fld$-scheme $X$, and a fat point $\mathfrak w$, we have
$$
(\arc\fat{\motif})(\mathfrak w)=\motif X(j_*j^*\mathfrak w)
$$
where $j_*j^*\mathfrak w$ is the product $\fat\times_\fld\mathfrak w$ viewed as a fat point over $\fld$, denoted henceforth simply by $\fat\mathfrak w$. Applied to a $\fld$-scheme $X$, we get the so-called \emph{arc scheme} $\arc\fat X$, whose $\mathfrak w$-rational points are in one-one correspondence with the $\fat\mathfrak w$-rational points of $X$, and so we will identify henceforth
$$
X(\fat\mathfrak w)=(\arc\fat X)(\mathfrak w).
$$
Moreover, we have by Corollary~\ref{C:flatarc} a canonical morphism
\begin{equation}\label{eq:canarcmorph}
\rho_X\colon \arc\fat X\to X.
\end{equation}

\begin{remark}\label{R:arcHilb}
In other words, $\arc{\fat}X$ is the Hilbert scheme classifying all maps from 
$\fat$ to $X$. 
When $\fat=\op{Spec}(\pol\fld\xi/\xi^n\pol\fld\xi)$, the resulting arc 
scheme  is also known in the literature as a \emph{jet scheme} (though we prefer to reserve 
the latter nomenclature for   schemes of the form $\jet YnX$). 
\end{remark}

\begin{remark}\label{R:goodbasismot}
By the argument in the proof of Theorem~\ref{T:flatrestr}, for any fat point $\fat$, we may choose a basis $\Delta=\{\alpha_0,\dots,\alpha_{l-1}\}$  of its coordinate ring $(R,\maxim)$ with some additional properties. In particular, unless noted explicitly, we will always assume that the first base element is $1$ and that the remaining ones belong to $\maxim$. Moreover, once the basis is fixed, we let $\tilde\var$   be  the $l$-tuple of arc variables $(\tilde\var_0,\dots,\tilde\var_{l-1})$, so that  $\genarc\var=\tilde\var_0+\tilde\var_1\alpha_1+\dots+\tilde\var_{l-1}\alpha_{l-1}$  is   the corresponding generic arc. It follows from \eqref{eq:genarcg} that $\tilde f_0=f(\tilde\var_0)$, for any $f\in \pol \fld\var$.  By \cite[\S2.1]{SchEC}, we may choose $\Delta$  so that, with  $\id_i:=(\alpha_i,\dots,\alpha_{l-1})R$, we have a  Jordan-Holder composition 
series\footnote{Writing $R$ as a homomorphic image of $\pol \fld\vary$ so that      $\vary:=\rij\vary e$ generates $\maxim$, let $\id(\alpha)$, for $\alpha\in\zet^n_{\geq 0}$,  be the ideal in $R$ generated by all $\vary^\beta$ with $\beta$ lexicographically larger than $\alpha$. Then we may take $\Delta$ to be all monomials $\vary^\alpha$ such that $\vary^\alpha\notin\id(\alpha)$, ordered lexicographically.}  
$$
\id_l=0\varsubsetneq\id_{l-1}\varsubsetneq \id_{l-2}\varsubsetneq\dots \varsubsetneq\id_1=\maxim\varsubsetneq\id_0=R.
$$ 

 Given $r\in R$, we expand it as in \eqref{eq:expbas} in the basis as $r=r_0+r_1\alpha_1+\dots+r_{l-1}\alpha_{l-1}$, with $r_i\in\fld$.  
I claim that $r_j=0$  for $j<i$ whenever $r\in \id_i$. Indeed, if not, let $j<i$ be minimal so that there exists a counterexample with $r_j\neq0$. By minimality, $r=r_j\alpha_j+r_{j+1}\alpha_{j+1}+\dots\in \id_i$ showing that $\alpha_j\in\id_{j+1}$, since $r_j$ is invertible. However, this implies that $\id_j=\id_{j+1}$, contradiction. 
 From this, it is now easy to see that the first $m$ basis elements of $\Delta$ form a basis of $R_m:=R/\id_{m+1}$. 
 Therefore, calculating $\tilde f_m$ in \eqref{eq:genarcg} for $f\in\pol \fld\var$ does not depend on whether we work over $R$ or with $R_m$, and hence, in  particular, $\tilde f_m\in\pol \fld{\tilde\var_0,\dots,\tilde\var_m}$ for every $m< l$.
 \end{remark}
 
 With these observations, we can now prove the following important openness property of arcs:
 
 \begin{theorem}\label{T:openarc}
Given a $\fld$-scheme $X$, a fat point $\fat$, and an open $U\sub X$, we have   isomorphisms
\begin{equation}\label{eq:openarc}
\arc\fat U\iso \inverse{\rho_X}U\iso \arc\fat X\times_XU. 
\end{equation}
\end{theorem}
\begin{proof}
By the universal property of adjunction, whence of arcs, it suffices to verify \eqref{eq:openarc} in case $X=\op{Spec}B\sub\affine\fld n$ is affine and $U=\op{Spec}(B_f)$ is a basic open subset. Let $A$ be the coordinate ring of $\arc\fat X$. Since $U$ is the closed subscheme of $\affine B1$ given by $g:=fy-1=0$, the corresponding arc  scheme $\arc\fat U$ is the closed subscheme of $\affine An$ with coordinate ring $A':=\pol A{\tilde\vary}/(\tilde g_0,\dots,\tilde g_{l-1})\pol A{\tilde\vary}$, where $l$ is the length of $\fat=\op{Spec}R$ and the $\tilde g_i$ are given by \eqref{eq:genarcg}, with $\tilde\vary$ a tuple of $l$ variables. By Remark~\ref{R:goodbasismot}, we may calculate the $\tilde g_i$ using any basis $\alpha_0=1,\dots,\alpha_{l-1}$ of $R$, and so we may assume it has the properties discussed in that remark.  In particular, by the last observation in that remark,  each $\tilde g_i$ only depends on $\tilde \vary_0,\dots,\tilde\vary_i$. Clearly, $\tilde g_0=\tilde f_0\tilde\vary_0-1$. In particular, the $A$-subalgebra of $A'$ generated by $\tilde\vary_0$ is just $A_{\tilde f_0}$. We will prove by induction, that each $\tilde\vary_i$ belongs to this subalgebra, and hence $A'=A_{\tilde f_0}$, as we needed to prove.

To verify the claim, we may assume by induction that $\tilde\vary_0,\dots\tilde\vary_{i-1}$ belong to $A_{\tilde f_0}$. The coefficient of $\alpha_i$ in the expansion
\begin{equation}\label{eq:expanfy}
(\tilde f_0+\tilde f_1\alpha_1+\dots+\tilde f_{l-1}\alpha_{l-1})(\tilde\vary_0+\tilde\vary_1\alpha_1+\dots+\tilde \vary_{l-1}\alpha_{l-1}).
\end{equation}
is equal to $\tilde g_i$, whence zero in $A'$. As observed in Remark~\ref{R:goodbasismot},  the choice of basis allows us to ignore all terms with $\alpha_j$ for $j>i$. Put differently, upon replacing $R$ by $R/(\alpha_{i+1},\dots,\alpha_{l-1})R$, which does not effect the calculation of $\tilde g_i$, we may assume that they are zero in \eqref{eq:expanfy}. Hence, 
$$
\tilde g_i=\tilde f_0\tilde \vary_i+\text{terms involving only $\tilde\vary_0,\dots,\tilde\vary_{i-1}$}
$$
proving the claim, since $\tilde f_0$ is clearly invertible in $A_{\tilde f_0}$. 
\end{proof}

Before we proceed, some simple examples are in order. It is clear from the definitions that $\arc \fat{\affine\fld 1}=\affine\fld l$, where $l$ is the length of $\fat$. We will generalize this in Theorem~\ref{T:fibmot} below. Arc spaces are   sensitive    to singularities, as the next examples show:

 \begin{example}\label{E:arcmot}
Let us calculate the arc scheme of the cusp $C$  given by the equation $ \var^2-\vary^3=0$ along the fat point $\fat$ with coordinate ring the four dimensional algebra  $R:=\pol \fld{\xi,\zeta}/(\xi^2,\zeta^2)\pol \fld{\xi,\zeta}$, using the basis $\Delta:=\{1,\xi,\zeta,\xi\zeta\}$ (in the order listed), and corresponding arc variables    $\tilde\var=(\tilde\var_{00},\tilde\var_{10},\tilde\var_{01},\tilde\var_{11})$ and $\tilde\vary=(\tilde\vary_{00},\tilde\vary_{10},\tilde\vary_{01},\tilde\vary_{11})$. One easily calculates that $\arc \fat C$ is given by the equations
$$
\begin{aligned}
    \tilde\var_{00}^2&=\tilde\vary_{00}^3   \\
    2\tilde\var_{00}\tilde\var_{10}&=3\tilde\vary_{00}^2\tilde\vary_{10} \\  
     2\tilde\var_{00} \tilde\var_{01}&=3\tilde\vary_{00}^2 \tilde\vary_{01} \\  
     2\tilde\var_{00} \tilde\var_{11}+2\tilde\var_{10} \tilde\var_{01}&=3\tilde\vary_{00}^2 \tilde\vary_{11}+6 \tilde\vary_{00}\tilde\vary_{10} \tilde\vary_{01}.
\end{aligned}
$$
Note that above the singular point $\tilde\var_{00}=0= \tilde\vary_{00}$, the fiber consist of two $4$-dimensional hyperplanes, whereas above any regular point, it is a $3$-dimensional affine space, the expected value by Theorem~\ref{T:fibmot} below.
\end{example}

\begin{example}\label{E:dual}
Another example  is classical: let $R_2=\pol \fld \xi/\xi^2\pol \fld \xi$ be the ring of dual numbers and $\mathfrak l_2:=\op{Spec}(R_2)$ the corresponding fat point. Then one verifies that a $\fld$-rational point on $\arc {\mathfrak l_2}X$ is given by a $\fld$-rational point $P$ on $X$, and a tangent vector  $\tuple v$ to $X$ at $P$, that is to say, an element in the kernel of the Jacobian matrix $\jac X (P)$. 
\end{example}

\begin{example}\label{E:lnlm}
As a last example, we calculate $\arc {\mathfrak l_n}{\mathfrak l_m}$, where $\mathfrak l_n$ is the $n$-th jet of the origin on the affine line, that is to say, $\op{Spec}(\pol \fld \xi/\xi^n\pol \fld \xi)$. With $\genarc \var=\tilde\var_0+\tilde\var_1\xi+\dots+\tilde\var_{n-1}\xi^{n-1}$, we will expand $\genarc \var^m$ in the basis $\{1,\xi,\dots,\xi^{n-1}\}$ of $\pol \fld \xi /\xi^n\pol \fld \xi$;  the coefficients of this expansion then generate the ideal of definition   of $\arc{\mathfrak l_n}{\mathfrak l_m}$. A quick calculation shows that these generators are the polynomials
$$
g_s(\tilde\var_0,\dots,\tilde\var_{n-1}):=\sum_{i_1+\dots+i_m=s}\tilde\var_{i_1}\tilde\var_{i_2}\cdots \tilde\var_{i_m}
$$
for $s=0,\dots,n-1$, where the $i_j$ run over $\{0,\dots,n-1\}$. Note that $g_0=\tilde\var_0^m$. One shows by induction that $(\tilde\var_0,\dots,\tilde\var_s)\pol \fld {\tilde\var}$ is the unique minimal prime  ideal of $\arc{\mathfrak l_n}{\mathfrak l_m}$, where $s=\round nm$ is the \emph{round-up}  of $n/m$, that is to say, the  least integer greater than or equal to $n/m$. In particular, $\arc{\mathfrak l_n}{\mathfrak l_m}$ is irreducible (but not reduced)  of dimension $n-\round nm$.
\end{example}

Immediately from Theorems~\ref{T:bcgr} and \ref{T:flatrestr}, we get:

\begin{theorem}\label{T:arcmot}
For each fat point $\fat$, the arc functor $\arc\fat{}$ induces a ring 
endomorphism on each of the motivic \gr{s} $\grot{\funcalg \fld}$, $
\grot{\funcsub\fld}$ and $\grot{\funcinf\fld}$. \qed
\end{theorem}

In case of complete formal motives, we even have:

\begin{lemma}\label{L:arcform}
For  any closed immersion $Y\sub X$ of $\fld$-schemes, and any fat point $
\fat$, we have isomorphisms
$$
\arc\fat{(\complet X_Y)}\iso \arc\fat X\times_X\complet X_Y\iso 
(\complet{\arc {\fat}X})_{\inverse\rho Y},
$$
where $\rho\colon\arc{\fat}X\to X$ is the canonical map from \eqref{eq:canarcmorph}. 
\end{lemma}
\begin{proof}
Let $U:=X-Y$. By Proposition~\ref{P:formsieve}, we have an equality 
\begin{equation}\label{eq:compcomp}
-\func{\complet X_Y}=\func U
\end{equation}
 of sieves on $X$.
 By 
Theorem~\ref{T:morp}, we may pull  back \eqref{eq:compcomp} under 
the map $\rho\colon\arc\fat X\to X$, to get a relation 
$$
-\func{(\arc\fat X\times_X\complet X_Y)}=\rho^*{(-\func{\complet 
X_Y})}=\rho^*\func U=\func{(\arc\fat X\times_XU)}=\func{(\arc\fat U)}
$$
where we used 
the openness of arcs (Theorem~\ref{T:openarc}) for the last 
equality. On the other hand, taking arcs  in identity~\eqref{eq:compcomp}, 
yields
$$
-\arc\fat{\func{\complet X_Y}}=\arc\fat{(-\func{\complet X_Y})}=\arc
\fat{\func U}=\func{(\arc\fat U)}
$$  
where one easily checks that arc functors commute with complements of complete sieves. Combining both identities and 
taking complements then proves the first isomorphism.  

To see the second isomorphism, we may assume, in view of the local nature 
of arcs, that $X=\op{Spec}A$ is affine. Let $I\sub A$ be the ideal of 
definition of $Y$, so that the global sections of $\complet X_Y$ is the 
completion $\complet A_I$ of $A$ with respect to $I$. Let  $
\pol A\vary/J$ be the coordinate ring of  the arc scheme $\arc\fat X$, for some $J\sub\pol A\vary$ and some tuple of variables $
\vary$. By the first isomorphism, the global section ring of $\arc
\fat{\complet X_Y}$ is  equal to the base change $\pol{\complet A_I}\vary/
J\pol{\complet A_I}\vary$. The ideal defining $\inverse\rho Y$ in $\arc\fat 
X$ is $I(\pol A\vary/J)$, and the completion of $\pol A\vary/J$ with respect 
to this ideal is   $\pol{\complet A_I}\vary/J\pol{\complet A_I}\vary$, 
proving the second isomorphism.
\end{proof}

It is easy to check that $\arc{\mathfrak v}{}\arc{\mathfrak w}{}=\arc{\mathfrak v\mathfrak w}{}=\arc{\mathfrak w}{}\arc{\mathfrak w}{}$, so that all arc functors commute with one another. If $\fld$ has positive \ch, we also have a Frobenius adjoint acting on the sub-\zariski\ and formal \gr{s}, and we have the following commutation relation
\begin{equation}\label{eq:arcfrobcomm}
\arc\fat{}\arc\frob{}=\arc\frob{}\arc{\frob\fat}{}
\end{equation}
for any fat point $\fat$. Indeed, we verify this on an arbitrary motif $\motif X$ and a fat point $\mathfrak w$. The left hand side of \eqref{eq:arcfrobcomm} becomes
\begin{align*}
(\arc\fat{}\arc\frob{})(\motif X)(\mathfrak w)&=\arc\fat{(\arc\frob{\motif X})}(\mathfrak w)\\
&=(\arc\frob{\motif X})(\fat\mathfrak w)\\
&=\motif X(\frob(\fat\mathfrak w)),
\end{align*}
whereas the right hand side becomes
\begin{align*}
(\arc\frob{}\arc{\frob\fat}{})(\motif X)(\mathfrak w)&=\arc\frob{(\arc{\frob\fat}{\motif X})} (\mathfrak w)\\
&=(\arc{\frob\fat}{\motif X})(\frob\mathfrak w)\\
&=\motif X((\frob\fat)(\frob\mathfrak w)),
\end{align*}
and these are both equal since an easy calculation shows that  $\frob(\fat\mathfrak w)=(\frob\fat)(\frob\mathfrak w)$.

\subsection*{Arcs and locally trivial fibrations}
By adjunction, any morphism $\bar{\fat}\to {\fat}$   of fat points  induces a  natural transformation  of arc functors $ 
\arc\fat{}\to \arc {\bar {\fat}}{}$. In particular, taking $
\bar {\fat}$ to be the geometric point   given by $\fld $ itself, we get a 
canonical morphism $\arc\fat{\motif X}\to \motif X$, for any motif $\motif X$, since   $
\arc\fld\cdot$ is the identity functor. In   case   $\motif X=\func X$ is 
representable, this is none other than  the canonical morphism   $\rho_X\colon \arc {\fat}X\to X$ from \eqref{eq:canarcmorph}. To formulate the key property of this morphism, we need a definition.

We call a morphism $Y\to X$ of $\fld $-schemes a \emph{locally trivial fibration with fiber $Z$} if for each (closed) point $P\in X$, we can find an open $U\sub X$ containing $P$ such that the restriction of $Y\to X$ to $U$ is isomorphic with the projection $U\times_\fld  Z\to U$.

\begin{lemma}\label{L:fibmot}
If $f\colon Y\to X$ is a locally trivial fibration of $\fld $-schemes with   fiber $Z$, then $\class Y=\class X\cdot\class Z$ in $\grot{\funcsub\fld}$.
\end{lemma}
\begin{proof}
By definition and compactness,  there exists a finite open covering $X=X_1\cup \dots\cup X_n$, so that   
$$
\inverse f{X_i}\iso X_i\times_\fld Z,
$$
 for $i=\range 1n$. In fact, for any non-empty subset $I\sub\{1,\dots,n\}$, we have an isomorphism $\inverse f{X_I}\iso X_I\times_\fld Z$, and hence, after taking classes in $\grot{\funcsub\fld}$, we get $\class{\inverse f{ X_I}}=\class{X_I} \cdot\class Z$.  Since the $\inverse f{ X_i}$ form an open affine covering of $Y$ and pre-images commute with intersection, a double application of Lemma~\ref{L:opencov} yields 
$$
 \class{Y}=\sum_{\emptyset\neq I\sub\{1,\dots,n\}} (-1)^{\norm I}\class{\inverse f{X_I}}=\sum_{\emptyset\neq I\sub\{1,\dots,n\}} (-1)^{\norm I}\class{X_I}\cdot\class Z=\class X\cdot\class Z
 $$
 in $\grot{\funcsub\bas}$.
\end{proof}

\begin{theorem}\label{T:fibmot}
If $X$ is a $d$-dimensional smooth $\fld$-scheme and $\bar \fat\sub \fat$ a closed immersion of  fat points, 
then the canonical map $\arc{\fat}X\to \arc{\bar\fat}X$ is a locally trivial 
fibration with fiber $\affine\fld {d(l-\bar l)}$, where $l$ and $\bar l$ are the respective  lengths of $\fat$ and $\bar\fat$. In particular, 
$$
\class{\arc\fat X}=\class X\cdot\lef^{d(l-1)}
$$
 in $\grot{\funcsub\fld}$.
\end{theorem}
\begin{proof}
Let $R$ and $\bar R$ be the Artinian local coordinate rings of $\fat$ and $\bar \fat$ respectively. Since arcs can be calculated locally (see the discussion following Lemma~\ref{L:opencomp}), we may assume $X$ is the (affine) closed subscheme of $\affine\fld m$ with ideal of definition  $\rij fs\pol\fld\var$. 
Since the composition of locally trivial fibrations is again a locally trivial fibration, with general fiber the product of the fibers, we may reduce to the case that $\bar R=R/\alpha R$ with $\alpha $ an element in the socle of $R$, that is to say, such that $\alpha \maxim=0$, where $\maxim$ is the maximal ideal of $R$.  Let $\Delta$ be a basis of $R$ as   in Remark~\ref{R:goodbasismot}, with $\alpha_{l-1}=\alpha $ (since $\alpha $ is a socle element, such a basis always exists). In particular, $\Delta-\{\alpha \}$ is a basis of $\bar R$. We will use these bases to calculate both arc maps.

To  calculate a general fiber of the map $s\colon\arc {\fat} X\to \arc{\bar\fat} X$, fix a fat point $\mathfrak w$ with coordinate ring $S$, and  a $\mathfrak w$-rational point $\bar b\colon\mathfrak w\to \arc{\bar\fat}X$, given by a tuple $\tilde{\tuple u}$ over $S$. The fiber $\inverse{s(\mathfrak w)}{\bar b}$, is equal to the fiber of $X(\fat\mathfrak w)\to X(\bar\fat\mathfrak w)$ above $\bar a$, where  $\bar a\colon\bar\fat\mathfrak w\to X$ is the $\bar\fat\mathfrak w$-rational point corresponding to $\bar b$, that is to say, the composition $\bar\fat\mathfrak w\to \bar\fat\times\arc{\bar\fat}X\to X$ 
given by 
 Theorem~\ref{T:flatrestr} (see Corollary~\ref{C:flatarc}). Being a rational point, $\bar a$ corresponds therefore to a solution ${\tuple u}$  in  ${\bar R\tensor_\fld S}$ of the equations $f_1=\dots=f_s=0$, where the relation with the tuple $\tilde{\tuple u}$ is given by equation~\eqref{eq:expbas}. Let $x$ be the center of $\bar a$, that is to say, the closed point given as the image of $\bar a$ under the canonical map $X(\bar\fat\mathfrak w)\to X(\fld)$. Since $X$ is smooth at $x$, the Jacobian $(s\times n)$-matrix $\jac X:=(\partial f_i/\partial \var_j)$ has rank $m-d$ at $x$. Replacing $X$ by an affine local neighborhood of $x$ and rearranging the variables if necessary, we may assume that the first $(m-d)\times(m-d)$-minor in $\jac X$ is invertible on $X$.

The surjection $R\to \bar R$ induces a surjection $R\tensor_\fld S\to \bar R\tensor_\fld S$.  The fiber above  $\bar a$ is  therefore   defined by the equations $f_j( {\tuple u}+
 \tilde\var_{l-1}\alpha )=0$, for $j=\range 1s$. By Taylor expansion, this becomes
\begin{equation}\label{eq:fib}
 0=f_j( {\tuple u}+\tilde\var_{l-1}\alpha )=\Big(\sum_{i=1}^m\frac{\partial f_j}{\partial \var_i}( {\tuple u})\tilde\var_{l-1,i}\Big)\alpha 
\end{equation}
 since    $f_j( {\tuple u})=0$ and $\alpha^2=0$ in $R\tensor_\fld S$. In fact, since $ {\tuple u}\equiv \tilde{\tuple u}_0\mod\maxim(R\tensor_\fld S)$ and $\alpha \maxim=0$,  we may replace each $\partial f_j/\partial \var_i( {\tuple u})$  in \eqref{eq:fib}  by $\partial f_j/\partial \var_i(\tilde{\tuple u}_0)$. Hence, the fiber above $\bar a$ is is the linear subspace of $(R\tensor S)^m$ defined as the kernel of the Jacobian   $\jac X(\tilde{\tuple u}_0)$. In view of the shape of the Jacobian of $X$,  we can find $g_{ij}\in\pol\fld\var$ such that 
 $$
 \tilde\var_{l-1,i} = \sum_{j>m-d} g_{ij}(\tilde{\tuple u}_0) \tilde\var_{l-1,j}
 $$
 for all  $i\leq m-d$, by 
  Kramer's rule. Therefore, viewing the parameter $\tilde{\tuple u}_0$ as varying over $X(\mathfrak w)$,  the fiber of $s(\mathfrak w)$ is  the constant space $\affine \fld d$, as we needed to show.

Applying this to $\arc{\fat} X\to X$, (note that $X= \arc \fld X$) we get a locally trivial fibration with   fiber equal to $\affine \fld {d(l-1)}$, so that the last assertion   follows from Lemma~\ref{L:fibmot}. 
\end{proof}

Calculations, like for instance Example~\ref{E:lnlm}, suggest that even for certain non-reduced schemes, there might be an underlying locally trivial fibration. Based on these examples, I would venture the following conjecture (here we write $X^{\text{red}}$ for the underlying reduced variety of a scheme $X$):

\begin{question}\label{C:nonredfibmot}
Let $\fat$ be a fat point of length $l$ and $X$ a $d$-dimensional $\fld$-scheme. If the reduction of   $X$ is smooth, when is the induced  reduction map $ (\arc\fat X)^{\text{red}}\to X^{\text{red}}$   a locally trivial fibration with fiber $\affine\fld m$, for some $m$. 
\end{question}

\begin{remark}\label{R:redarc}
As we shall see in Example~\ref{E:deltafat} below, $m$ can be different from $d(l-1)$, the value  that we get in the reduced case. In many cases, the answer seems to be affirmative, but there are exceptions, see Example~\ref{E:nonredfib} below.

Moreover, as can be seen from Table~\eqref{tab:1} below, taking arcs does not commute with reduction, that is to say, $(\arc\fat X)^{\text{red}}$ is in general not equal to the arc space $\arc\fat{(X^{\text{red}})}$ of the reduction of $X$, nor  even   to the reduction of the latter arc space.
\end{remark}

\begin{example}\label{E:nonredfib}
The simplest instance to which Question~\ref{C:nonredfibmot} applies is when $X$ itself is a fat point $\mathfrak x$. The expectation then is that 
\begin{equation}\label{eq:nonredfibfat}
(\arc\fat{\mathfrak x})^{\text{red}}\iso \affine\fld m
\end{equation} 
for some $m$ (for expected values, see Example~\ref{E:deltafat} below). Example~\ref{E:lnlm} provides instances in which  \eqref{eq:nonredfibfat} holds. However, the following is a counterexample: let $\fat:=\jet O4C$, where $C$ is the cuspidal curve with equation $\xi^2-\zeta^3=0$ and $O$ the origin, its unique singularity. Let us calculate its \emph{auto-arcs} $\arc\fat\fat$. As the monomials in $\xi$ and $\zeta$ of degree at most two together with $\xi\zeta^2$ form a basis of the coordinate ring $R$ of $\fat$, its length is $7$ and the generic arcs are
$$
\genarc\var= \tilde\var_0+\tilde\var_1\xi+\dots+\tilde\var_6\xi\zeta^2\quad\text{and}\quad\genarc\vary= \tilde\vary_0+\tilde\vary_1\xi+\dots+\tilde\vary_6\xi\zeta^2.
$$
Since the arc scheme $\arc\fat\fat$ lies above the origin, its reduction lies in the subvariety of $\affine\fld{14}$ defined by $\tilde\var_0=\tilde\vary_0=0$, and hence, we may put these two to zero in the generic arcs and work inside the affine space $\affine\fld{12}$ given by the remaining arc variables. From the fact that $\xi^3=0$ in $R$, the arc scheme is contained in the closed subscheme of $\affine\fld{12}$ by the coefficients of the expansion of 
$$
\genarc\var^3=3\tilde\var_1\tilde\var_2^2\xi\zeta^2+\tilde\var_2^3\xi^2+\dots
$$ 
In particular, since $\tilde\var_2^2$ vanishes, the reduction lies in the subvariety given by $\tilde\var_2=0$, and so we may again put this variable equal to zero and work in the corresponding $11$-dimensional affine space. The remaining equations come from the expansion of 
\begin{align*}
\genarc\var^2-\genarc\vary^3&=(\tilde\var_1\xi+\tilde\var_3\xi^2+\tilde\var_4\xi\zeta+\tilde\var_5\zeta^2+\dots)^2-(\tilde\vary_1\xi+\tilde\vary_2\zeta+\dots)^3\\
&=(\tilde\var_1^2-\tilde\vary_2^3)\xi^2+(2\tilde\var_1\tilde\var_5-3\tilde\vary_1\tilde\vary_2^2)\xi\zeta^2
\end{align*}
showing that the reduced arc space is the singular variety with equations $\tilde\var_1^2-\tilde\vary_2^3=2\tilde\var_1\tilde\var_5-3\tilde\vary_1\tilde\vary_2^2=0$. Note that the latter can be viewed as the tangent bundle of the cusp. More precisely, instead of the anticipated \eqref{eq:nonredfibfat}, we obtain the following modified form of the auto-arc variety
$$
(\arc{\jet O4C}{(\jet O4C)})^{\text{red}}\iso \arc{\mathfrak l_2}C\times\affine\fld 7,
$$
 a singular $9$-dimensional variety. However, I could not find such a form for   values higher than $4$.
 \end{example}

\subsection*{Locally constructible sieves}
We say that a sieve $\motif X$ on a $\fld$-scheme $X$ is \emph{locally constructible}, if $\motif X(\fat)$ is constructible in $X(\fat)$, for each fat point $\fat$, by which we mean that  $\arc\fat{\motif X}(\fld)$ is constructible in the Zariski topology  on the variety $\arc\fat X(\fld)$ viewed as the space of   closed points of $\arc\fat X$.

\begin{proposition}\label{P:locdefform}
Any  formal motif   is locally constructible.
\end{proposition}
\begin{proof}
This follows from Chevalley's theorem and Theorem~\ref{T:arcmot} in case $\motif X$ is sub-\zariski, since, for a morphism $\varphi\colon Y\to X$ of $\fld$-schemes, $\fim\varphi(\fat)$, as a subset of $\arc\fat X$, is the image of the map $\arc\fat Y(\fld)\to\arc\fat X(\fld)$. The formal case then follows from this, since there exists a sub-\zariski\ motif $\motif Y\sub\motif X$ such that $\motif Y(\fat)=\motif X(\fat)$.
\end{proof}

\section{Dimension}\label{s:motdim}
In this section, we assume $\fld$ is an \acf. The dimension of an arc 
scheme $\arc{\fat}X$ is a subtle invariant depending on    $\fat$ and $X$, 
and not just on their respective length $l$ and dimension $d$; see   Table~
\eqref{tab:1} below. The underlying cause for this phenomenon is the fact that  taking reduction does not commute with taking arcs.  To exemplify this behavior, we list, for small lengths, some  defining equations of   arcs and their reductions for three different closed subschemes $X$ with the same underlying one-dimensional variety, the union of two lines in the plane. Here $\mathfrak l_l$ denotes the closed point with coordinate ring $\pol\fld\xi/\xi^l\pol\fld\xi$, that is to say, the $l$-th germ of the origin on the affine line.

\newcommand\breathe{{\vphantom{\left(\complet\Sigma_\Sigma\right)}}} 
\begin{table}[h]
\caption{Dimension $\delta$ and equations of arcs   and their reductions.}
\label{tab:1}     
\begin{tabular}{|c|c|c|r|c|r|c|r|}
\hline
$X$&$l \breathe $ &$xy=0$& $\delta$  & \phantom{xxx}$x^2y=0$\phantom{xxx}&$\delta $ & \phantom{xxxxx}$x^2y^3=0\phantom{xxxxxxxx}$&$\delta $   \\
\hline\hline
\multirow{4}{*}{$\arc {\mathfrak l_l}{ }$} &$1 \breathe $ 	&\multicolumn{2}{|c|}{$\tilde x_0\tilde y_0,$} & \multicolumn{2}{|c|}{$\tilde x_0^2\tilde y_0,$} & \multicolumn{2}{|c|}{$\tilde x_0^2\tilde y_0^3,$}\\\cline{2-8}
		&$2 \breathe $	&\multicolumn{2}{|c|}{$\tilde x_0\tilde y_1+\tilde x_1\tilde y_0,$} & \multicolumn{2}{|c|}{$2\tilde x_0\tilde x_1\tilde y_0+\tilde x_0^2\tilde y_1,$} & \multicolumn{2}{|c|}{$2\tilde x_0\tilde x_1\tilde y_0^3+3\tilde x_0^2\tilde y_0^2\tilde y_1,$}\\\cline{2-8}
		&$3 \breathe $ &\multicolumn{2}{|c|}{ $\tilde x_0\tilde y_2+\tilde x_1\tilde y_1+$} & \multicolumn{2}{|c|}{$\tilde x_0^2\tilde y_2+2\tilde x_0\tilde x_1\tilde y_1+$} &
		\multicolumn{2}{|c|}{$3\tilde x_0^2(\tilde y_0\tilde y_1^2+\tilde y_0^2\tilde y_2)+$}\\
		&&\multicolumn{2}{|c|}{$\breathe \tilde x_2\tilde y_0$\phantom{xxxxxxxx}}& \multicolumn{2}{|c|}{$(2\tilde x_0\tilde x_2+\tilde x_1^2)\tilde y_0$\phantom{xx}}& \multicolumn{2}{|c|}{$6\tilde x_0\tilde x_1\tilde y_0^2\tilde y_1+(\tilde x_1^2+2\tilde x_0\tilde x_2)\tilde y_0^3$}\\
			\hline
\multirow{3}{*}{$\nabla_{\mathfrak l_l}^{\op{red}}$} &$1 \breathe $ & $\tilde x_0\tilde y_0,$&1& $\tilde x_0\tilde y_0,$&1& $\tilde x_0\tilde y_0,$&1\\\cline{2-8}
&$2 \breathe $ & $\tilde x_0\tilde y_1,\tilde x_1\tilde y_0,$ &2& $\tilde x_0\tilde y_1,$ &3& [no new equation]&3\\\cline{2-8}
&$3 \breathe $ & $\tilde x_0\tilde y_2,\tilde x_1\tilde y_1,\tilde x_2\tilde y_0$ &3& $\tilde x_0\tilde y_2,\tilde x_1\tilde y_0$ &4& $\tilde x_1\tilde y_0$&5 \\
\hline
\end{tabular}
\end{table}

As substantiated by the data in  this table, we have the following general estimate:

\begin{lemma}\label{L:dimarc}
The dimension of $\arc{\fat}X$ is at least $dl$, where $d$ is the dimension 
of $X$ and $l$ the length of $\fat$. If $X$ is a variety, then this is an equality.
\end{lemma}
\begin{proof}
Assume first that $X$ is an irreducible variety, so that it contains a dense open subset $U$ which is non-singular. By  Theorem~\ref{T:openarc}, the pull-back $\arc{\fat} U=U\times_X\arc{\fat} X$ is a dense open subset of $\arc{\fat} X$. Moreover, by Theorem~\ref{T:fibmot} the dimension of $\arc{\fat} U$ is equal to $dl$,  whence, by density, so is that of  $\arc{\fat} X$. If $X$ is only reduced, then we may repeat this argument on an irreducible component of $X$, and using once more the openness of arcs, conclude that $\arc\fat X$ has dimension $dl$.

For $X$ arbitrary, let $V:=X^{\text{red}}$ be the   variety 
underlying $X$. The closed immersion $V\sub X$ yields a closed immersion $\arc{\fat}V\sub
\arc{\fat}X$ by Corollary~\ref{C:flatarc}. The result now follows from 
the reduced case applied to $V$.
\end{proof}

We will call the difference $\op{dim}(\arc\fat X)-dl$ the \emph{defect of $X$ at $\fat$}. Varieties therefore have no defect. The bound given by Lemma~\ref{L:dimarc} is far from optimal, as can be seen by taking the arc scheme of a fat point (see, for instance, Example~\ref{E:lnlm}). The growth of the dimension of auto-arcs (see Example~\ref{E:nonredfib}), that is to say, the function
$$
\delta(\fat):=\op{dim}(\arc\fat\fat)
$$
for $\fat$ a fat point, is still quite puzzling. By Example~\ref{E:lnlm}, we have $\delta(\mathfrak l_n)=n-1$. However, the next example shows that $\delta(\fat)$ can be bigger than $\ell(\fat)$.

\begin{example}\label{E:deltafat}
Let $\mathfrak o_n:=\jet On{\affine\fld2}$ be the \emph{$n$-fold origin} in the plane with ideal of definition   $(\xi,\zeta)^n$. Its length is equal to $o_n:=\binomial{n+1}2$, with a basis consisting of all monomials in $\xi$ and  $\zeta$ of degree strictly less than $n$. Let 
\begin{equation}\label{eq:genarcnorigin}
\genarc \var:=\sum_{i+j<n} \tilde\var_{ij}\xi^i\zeta^j \quad\text{and}\quad \genarc \vary:=\sum_{i+j<n} \tilde\vary_{ij}\xi^i\zeta^j
\end{equation} 
be  the generic arcs, so that   $\arc{\mathfrak o_n}{\mathfrak o_n}$ is the closed subscheme of  $\affine\fld{o_n}$ given by the coefficients of the monomials $\genarc\var^i\genarc\vary^{n-i}$, for $i=\range 0,n$. Since the arc scheme $\arc{\mathfrak o_n}{\mathfrak o_n}$ 
 lies above the origin, its defining equations contain the ideal $(\tilde\var_{00},\tilde\vary_{00})^n$. To calculate its dimension, we may take its reduction, which means that we may put $\tilde\var_{00}$ and $\tilde\vary_{00}$ equal to zero in \eqref{eq:genarcnorigin}. However, any monomial of degree $n$ in the generic arcs is then identical zero, showing that the reduction of the arc scheme is given by $\tilde\var_{00}=\tilde\vary_{00}=0$, and hence, its dimension is equal to 
 $$
 \delta(\mathfrak o_n)=2\binomial{n+1}2-2=n^2+n-2.
 $$

%
One might be tempted to propose therefore that $\delta(\fat)$ is equal to the embedding dimension of $\fat$ times its length minus one, but the next example disproves this. 
Namely, without proof, we state that $\delta(\fat)=7$ for $\fat$ the fat point in the plane with equations $\xi^2=\xi\zeta^2=\zeta^3=0$ (note that $\fat$ has length $5$ and embedding dimension $2$, so that the expected value would be $2\times 4=8$). Note that  the auto-arc space $\arc\fat\fat$ is often, but not always an affine space (see Question~\ref{C:nonredfibmot} and the example following it).

It seems  plausible that $\delta(\jet PnY)$ grows as a polynomial in $n$ of degree $d$, for any $d$-dimensional closed germ $(Y,P)$. In particular, we expect the limit 
$$
e(Y,P):=\lim_{n\to\infty}\frac{\delta(\jet PnY)}{\ell(\jet PnY)}
$$
 to exist. For instance, an easy extension of the above examples yields $e(\affine\fld m,O)=m$. In view of  Question~\ref{C:nonredfibmot}, we would even expect that the \emph{auto-Igusa-zeta} series
 $$
 \zeta_\fat(t):= \sum_{n=1}^\infty \lef^{-d \ell(\jet PnY)}\class{\arc{\jet PnY}{(\jet PnY)}} t^n
 $$
 is rational over the localization of the classical \gr\ with respect to $\lef$, for any $d$-dimensional closed germ $(Y,P)$, a phenomenon that we will study  in \S\ref{s:motseries} below under the name of \emph{motivic rationality} (and where we also explain the choice of power of $\lef$). What about its motivic rationality over the localization $\grot{\funcinf\fld}_\lef$ of the formal \gr?
\end{example}

\subsection*{Dimension of a motif}
 Given a formal motif $\motif X$ on a $\fld$-scheme $X$, we define its 
\emph{dimension} as the dimension of $\motif X(\fld)$. This is well-defined since $\motif X(\fld)$ is a constructible subset of 
$X(\fld)$ by Proposition~\ref{P:locdefform}. If $\motif X=\func X$ is 
representable, then its dimension is precisely the dimension of the scheme 
$X$. On the other hand, if $\motif X$ is the formal completion of $X$ at a closed point, then $\motif X$ has dimension zero, whereas its global section ring has dimension equal to that of $X$ at $P$ by Corollary~\ref{C:expsectform}.
 
 \begin{proposition}\label{P:dimmot}
If two formal motives have the same class in $\grot{\funcinf \fld}$, then 
they have the same dimension. 
\end{proposition}
\begin{proof}
Since dimension is determined by the $\fld$-rational points, we may take, using Theorem~\ref{T:mapclass},  
the image of this common class in $\grotclass\fld$, where the result is 
known to hold.
\end{proof}

As we will work over  $\grotmot:=\grot{\funcinf \fld}_\lef$ below, we  extend the notion of dimension 
into an integer valued invariant on this localized \gr\ by defining the dimension of $\class{\motif 
X}\cdot\lef^{-i}$ to be $\op{dim}(\motif X)-i$, for any formal motif $\motif 
X $ and any $i\in\zet$.  In particular, if $X$ has dimension $d$ and $\fat$ 
length $l$, then $\class{\arc{\fat}X}\cdot\lef^{-dl}$ has positive dimension, 
which is the reason behind the introduction of this power of the Lefschetz 
class in the formulas below. This also gives us the Kontsevich filtration by 
dimension on  $\grotmot $. Namely, for each $m\in\nat$, let $
\Gamma_m( \grotmot)$ be the subgroup generated by all classes $
\class{\motif X}\cdot\lef^{-i}$ of dimension at most $-m$. This is a 
descending filtration and the completion of $\grotmot $ with respect to this 
filtration will be denoted $\complet{\mathbf G}$. However, since we define 
motivic filtration locally (see \S\ref{s:motint} below), we will not make use of it.

\section{Deformed arcs}\label{s:defarc}
We continue with the setup from \S\ref{s:arc}: let $j_\fat\colon \fat\to \op{Spec}\fld$ be the structure morphism of a fat point $\fat$ over an \acf\ $\fld$. Instead of looking at the double adjunction giving rise to the arc functor $\arc{\fat}{}$, we consider here the diminution part only, that is to say, the right adjoint $\arc{j_\fat^*}{}$ satisfying for each $\fat$-sieve $\motif Y$ on a $\fat$-scheme $Y$ and each $\fld$-fat point $\mathfrak w$, an isomorphism
$$
(\arc{j_\fat^*}{\motif Y})(\mathfrak w)\iso \motif Y(j_\fat^*\mathfrak w) 
$$
where this time, we have to view $j_\fat^*\mathfrak w=\fat\mathfrak w$ as a $\fat$-fat point. By Theorem~\ref{T:flatrestr}, we associate in particular to any $\fat$-scheme $Y$, a $\fld$-scheme $\arc{j_\fat^*}Y$. In particular, if $Y=j_\fat^*X$ is obtained from a $\fld$-scheme $X$ by base change, then
\begin{equation}\label{eq:defarcbc}
\arc{j_\fat^*}Y=\arc\fat X
\end{equation}
by Corollary~\ref{C:flatarc}.

Apart from $j_\fat$, we also have the residue field morphism  $\pi_\fat\colon \op{Spec} \fld\to \fat$. To a  $\fat$-scheme $Y$, we can therefore also associate   the base change    $\bar Y:=\pi_\fat^*Y$, called the \emph{closed fiber}  of $Y$.   We can think of $Y$ as a \emph{fat 
deformation} of $j_\fat^*\bar Y$. Indeed, since $\fld
\times_\fat \fld=\fld$, any $\fld$-rational point of $Y$ is also a $
\fld$-rational point on $\bar Y$, that is to say, $Y(\fld)=\bar Y(\fld)=j^*_\fat\bar Y(\fld)$, showing that $Y$ and $j_\fat^*\bar Y$ have the same underlying variety.  

\begin{example}\label{E:defarc}
For instance, if $C$ is the curve $\var^2-\vary^3$ and   $\mathfrak l_n$ the 
fat point with coordinate ring $R_n:=\pol\fld\xi/\xi^n\pol\fld\xi$, then 
the $\mathfrak l_3$-scheme  $X$ with coordinate ring $\pol {R_3}{\var,
\vary}/(\var^2-\vary^3-\xi^2)\pol {R_3}{\var,\vary}$ has closed fiber $C$, 
and $X(\fld)=C(\fld)$. Note however that $X(\mathfrak l_3)\neq 
C(\mathfrak l_3)$. In fact, truncation yields  a  map $X(\mathfrak l_3)\to 
C(\mathfrak l_2)$. 
\end{example}

Hence, by \eqref{eq:defarcbc}, we may likewise think of $\arc{j_\fat^*}Y$ as a fat deformation of  the arc space $\arc\fat{\bar Y}$ of its closed fiber, justifying the term \emph{deformed arc space} for $\arc{j_\fat^*}Y$. This construction is compatible then with specializations in the following sense. Fix   a $\fld$-scheme $Z$. The base change $j_Z\colon Z\times\fat\to Z$  is again a finite, flat \homo\ satisfying condition~\eqref{eq:loccond}, thus allowing us to consider the diminution $\arc{j_Z^*}{}$, associating to any $Z\times\fat$-scheme $Y$, a $Z$-scheme $\arc{j_Z^*}Y$, called the \emph{relative arc scheme} of $Y$. The deformed arc space is then given by the special case when $Z=\op{Spec}\fld$.

\begin{proposition}\label{P:defarcfam}
Let $\fat$ be a $\fld$-fat point and $Z$   a $\fld$-scheme. For every $Z\times\fat$-scheme $Y$, viewed as a family over $Z$ in the sense of \S\ref{s:fammot}, and for  any $\fld$-rational point $a$ on $Z$, we have an isomorphism
$$
\arc{j_\fat^*}{(Y_{\tilde a})}= (\arc{j_Z^*}Y)_a
$$
of $\fld$-schemes, where $\tilde a \colon\fat\to Z\times\fat$ is the base change of $a$. 
\end{proposition}
\begin{proof}
Immediately from Theorem~\ref{T:projform} applied to the base change diagram
\commdiagram \fat {j_\fat}{\op{Spec}\fld}{\tilde a}a{Z\times\fat}{j_Z}{Z.}
\end{proof}

So, returning to    Example~\ref{E:defarc}, let $X\sub \affine{\mathfrak 
l_3}3$ be the  hypersurface   with equation $\var^2-\vary^3-z\xi^2$. As a 
family over $\affine{\mathfrak l_3}1$ via projection on the last coordinate, its specializations $X_a$ are all isomorphic if $a
\neq 0$, whereas the special fiber $X_0=C\times\mathfrak l_3$. The 
corresponding relative arc scheme $\arc{j_{\affine\fld1,\mathfrak l_3}^*}X$ is given by 
$$
\tilde\var_0^2-\tilde\vary_0^3=2\tilde\var_0\tilde\vary_1-3\tilde
\vary_0^2\tilde\vary_1=2\tilde\var_0\tilde\var_2+\tilde\var_1^2-3\tilde
\vary_0\tilde\vary_1^2-3\tilde\vary_0^2\tilde\vary_2-\tilde z_0=0
$$
Its specializations are again all isomorphic (to the third order Milnor fiber; 
see below) whereas the special fiber is isomorphic to $\arc{\mathfrak l_3}C
$.

\section{Limit points}\label{s:limpt}

The closed subscheme relation defines a partial order relation on $
\fatpoints\fld$, that is to say,  we say that $\bar\fat\leq\fat$ \iff\   $\bar
\fat$ is a closed subscheme of $\fat$ (and not just isomorphic to one).  We already discussed direct limits with respect to the induced ordering on disjoint unions of fat points in Lemma~\ref{L:dirlimJac}.  Here we will investigate more closely the direct limit of fat points themselves. We will assume that such a direct system contains a least element. It follows that all fat points in the system must have the same center (to wit, the center of the least element). In other words, any fat point in the directed system has the same underlying closed point, and so we will call such a system   a   \emph{point system}.

We want to adjoin to the category of fat points its direct limits, but the 
problem is that the category of schemes is not closed under direct limits 
either. However, the category of locally ringed spaces is: if $(X_i,\loc_{X_i})$ form a direct system, then their direct limit is 
the topological space $X:=\varinjlim X_i$ endowed with the structure sheaf 
$\loc_X:=\varprojlim\loc_{X_i}$. Since we will assume that all fat points have 
the same underlying topological space, namely a single point, the construction  simplifies: the direct limit of a point system is simply the one-point space 
  with its unique stalk given as the inverse limit of all the coordinate 
rings of the fat points in the system. As already indicated in Footnote~\eqref{f:locringed},   
a \emph{morphism} in this setup will mean a morphism of locally ringed 
spaces with values in the category of $\fld$-algebras. For example, if $R$ is 
any $\fld$-algebra and $\mathfrak o$ the locally ringed space with 
underlying set the origin and (unique) stalk $R$, and if $X=\op{Spec}A$ is 
any affine scheme, then $\op{Mor}_\fld(\mathfrak o,X)$ is in one-one 
correspondence with the set of $\fld$-algebra \homo{s} $\hom {\fld-
\text{alg}} AR$.

Let    $\categ X\sub\fatpoints\fld$ be a  point system. Its direct limit $
\varinjlim\categ X$, as a one-point locally ringed space,  is called a 
\emph{limit point}.  Some examples of limit points are:
\begin{enumerate}
\item  If $\categ X$ is finite, the direct limit is just its maximum, whence a 
fat point. 
\item Given a closed germ $(Y,P)$, its formal completion $\complet Y_P$  is the direct limit of the jets $
\jet PnY$, whence a  limit point.
\item The direct limit of all fat points with the same center is called the 
\emph{universal point} and is denoted $\univpoint_\fld$, or just $
\univpoint$. Any limit point admits a closed immersion into $\univpoint$. In particular, up to isomorphism, $\univpoint $ does not depend on the underlying point.
\end{enumerate}

\begin{lemma}\label{L:univpoint}
The stalk of the universal point $\univpoint_\fld$ is isomorphic to the 
power series ring over $\fld$ in countably many indeterminates.
\end{lemma}
\begin{proof}
Any fat point is a closed subscheme of some formal scheme  $
\complet{(\affine \fld n)}$. Hence suffices to show that the inverse limit of 
the power series rings $S_n:=\pow\fld{\var_1,\dots,\var_n}$ under the 
canonical projections $S_m\to S_n$ given by modding out the variables $
\var_i$ for $n<i\leq m$ is isomorphic to the power series ring $\pow\fld
\var$ in countably many indeterminates $\var=(\var_1,\var_2,\dots)$. To 
this end, let $f_n\in S_n$ be a compatible sequence in the inverse system. 
For each exponent $\nu=(\nu_1,\nu_2,\dots)$ in the direct sum $
\nat^{(\nat)}$ of countably many copies of $\nat$, and each $n$, let $a_{n,
\nu}\in \fld$ be the coefficient of $\var^\nu:= \var_1^{\nu_1}\cdots 
\var_{i(\nu)}^{\nu_{i(\nu)}}$ in $f_n$, where $i(\nu)$ is the largest index 
for which $\nu_i$ is non-zero. Compatibility means that there exists for each 
$\nu$ an element $a_\nu\in\fld$ such that $a_\nu=a_{n,\nu}$ for all 
$n>i(\nu)$. Hence $f:=\sum_\nu a_\nu\var^\nu\in\pow\fld\var$ is the 
limit of the sequence $f_n$, proving the claim. 
\end{proof}

\begin{remark}\label{R:univscheme}
I would guess that, similarly,   the direct limit of all disjoint unions of fat points is isomorphic to the affine space $\affine\fld\omega$ of countable dimension.
\end{remark}

To make the limit points into a category, denoted $\limfat\fld$, take 
morphisms to be direct limits of morphisms of fat points. More precisely,  
given point systems $\categ X,\categ Y\sub\fatpoints\fld$ with respective 
direct limits $\mathfrak x$ and $\mathfrak y$, then a morphism (of locally 
ringed spaces)  $\varphi\colon\mathfrak x\to\mathfrak y$ is a 
\emph{morphism of limit points} if for each $\fat\in\categ X$ there exists a 
$\mathfrak v\in\categ Y$ such that $\varphi(\fat)\sub\mathfrak v$, or, 
dually, if the induced morphism $\varprojlim \loc_{\mathfrak v}\to 
\varprojlim \loc_\fat$ has the property that for each $\fat\in\categ X$, we 
can find a $\mathfrak v\in\categ Y$ such that this morphism factors through 
$\loc_{\mathfrak v}\to \loc_\fat$. In this way, the category $\limfat\fld$ of limit 
points is an extension of the category $\fatpoints\fld$ of fat points, which 
in a sense acts as its \emph{compactification}. In particular,  any limit point $\mathfrak x$ admits a canonical structure morphism $j_{\mathfrak x}\colon\mathfrak x\to \op{Spec}\fld$.
We also extend the partial 
order relation on $\fatpoints\fld$ to one on $\limfat\fld$ as follows. 
Firstly, we say that $\fat\leq\mathfrak x$ for $\fat$ a fat point and $
\mathfrak x= \varinjlim\categ X$ a limit point, if $\fat\leq\mathfrak v$ 
for some fat point $\mathfrak v\in\categ X$. It follows that there is a 
canonical embedding $\fat\sub\mathfrak x$ which induces a surjection on 
the stalks, and which we therefore call a \emph{closed embedding} in 
analogy with the scheme-theoretic concept. We then say for a limit point $
\mathfrak y= \varinjlim\categ Y$ that $\mathfrak y\leq\mathfrak x$ if for 
every $\fat\in\categ Y$ we have $\fat\leq\mathfrak x$. It follows that we 
have a canonical morphism of limit points $\mathfrak y\to\mathfrak x$ 
which is again surjective on their stalks, and hence can rightly be called 
once more a \emph{closed embedding}. One checks that this defines indeed a 
partial order on limit points extending the one on fat points. 

We call a limit point $\mathfrak x$ \emph{bounded} if it is the direct limit 
of fat points  of embedding dimension at most $n$,  for some $n$. Formal 
completions of closed germs are examples of   bounded limit points, 
whereas $\univpoint$ clearly is not. In fact, any bounded limit point arises 
in a similar, analytical way:

\begin{proposition}\label{P:boundlim}
The bounded limit points are in one-one correspondence with analytic 
germs. More precisely, the stalks of bounded limit points are precisely the 
complete Noetherian local rings with residue field $\fld$. 
\end{proposition}
\begin{proof}
Let $\mathfrak x$ be a bounded limit point, say, realized as the direct limit 
of fat points $\fat\sub \affine \fld n$ centered at the origin, for some fixed 
$n$.  Let $\pol\fld\var/\id_\fat$ be the coordinate ring of $\mathfrak z$, so that 
$\id_\fat\sub\pol\fld\var$ is $\maxim$-primary, where $\maxim$ is the 
maximal ideal generated by the variables. Let $I$ be the intersection of all $
\id_\fat\pow\fld\var$. I claim that $\mathfrak x$ has stalk equal to $S:=\pow
\fld\var/I$. Indeed, by a theorem of Chevalley (\cite[Exercise 8.7]{Mats}), 
there exists for each $i$ some ideal of definition $\id$ of one of the fat 
points in the direct system such that $\id S\sub\maxim^iS$. In particular, 
the inverse limit is simply the $\maxim$-adic completion of $S$, which is of 
course $S$ itself. 

The converse is also obvious: given a complete Noetherian ring $(S,\maxim)
$ with residue field $\fld$, then by Cohen's structure theorem, it is of the 
form $\pow\fld\var/I$ for some ideal $I$. One easily checks that it is the 
coordinate ring of the direct limit of the corresponding jets $\op{Spec}(S/
\maxim^n)$.
\end{proof}

Any limit point $\mathfrak x=\varinjlim\categ X$ defines a  presheaf $
\func{\mathfrak x}$ by assigning to a fat point $\fat$  the set of morphisms 
$\op{Mor}_{\limfat\fld}(\fat,\mathfrak x)$. 

\begin{corollary}\label{C:boundlim}
 The presheaf  $\func{\mathfrak x}$ defined by a limit point $\mathfrak x=
\varinjlim\categ X$ is  the inverse limit of the representable functors $
\func{\mathfrak v}$ for $\mathfrak v\in\categ X$. If $\mathfrak x$ is 
moreover bounded, then $\func{\mathfrak x}$ is a formal motif.
\end{corollary}
\begin{proof}
Given a fat point $\fat=\op{Spec}R$, we have to show that $\mathfrak 
v(\fat)$ for $\mathfrak v\in\categ X$ forms an inverse system with inverse 
limit equal to $\op{Mor}_{\limfat\fld}(\fat,\mathfrak x)$. Let $(S,\maxim)$ 
be the stalk of $\mathfrak x$, that is to say, the inverse limit of the 
coordinate rings  of the fat points   belonging to $\categ X$. The first 
statement is immediate by functoriality, and for the second, note that since 
$R$ has finite length, 
\begin{equation}\label{eq:morringed}
\op{Mor}_{\limfat\fld}(\fat,\mathfrak x)\iso\hom {\fld-\text{alg}} SR.
\end{equation}
 More precisely, any $\fld$-algebra \homo\ $a\colon S\to R$ factors 
through $S/\maxim^l\to R$, for $l=\ell(R)$. Moreover, if $n$ is the 
embedding dimension of $R$, then there exists a complete, Noetherian 
residue ring $\bar S$ of $S$ of embedding dimension at most $n$ such that 
$a$ factor as $S\to \bar S/\maxim^l\bar S\to R$. By the  same argument as 
in the proof of Proposition~\ref{P:boundlim}, there is some $\mathfrak v=\op{Spec}T\in
\categ X$ such that $T\to \bar S/\maxim^l
\bar S$, showing that $a$ is already induced by the morphism $\fat\to
\mathfrak v$. In fact, if $\mathfrak x$ is bounded, then we may choose $
\mathfrak v$ independent from $a$, showing that $\func{\mathfrak x}
(\fat)=\mathfrak v(\fat)$. Since all $\mathfrak v\in\categ X$ embed in the 
same affine space,  $\func{\mathfrak x}$ is a locally \zariski\ sieve on this 
space, whence a formal motif.
\end{proof}

\begin{remark}\label{R:boundlim}
The identification~\eqref{eq:morringed} shows that  $\func{\mathfrak x}$ 
is pro-representable in the sense of Footnote~\eqref{f:locringed}.
\end{remark}

Let $\mathfrak x=\varinjlim\categ X$ be a limit point. Given a presheaf $
\motif X$, the collection of all  $\motif X(\mathfrak v)$ for $\mathfrak v\in
\categ X$ is an inverse system of sets given by the maps $i_{\mathfrak v,
\mathfrak w}\colon \motif X(\mathfrak w)\to \motif X(\mathfrak v)$ if $
\mathfrak v\leq\mathfrak w$ in $\categ X$, where $i_{\mathfrak v,
\mathfrak w}$ is  induced by the  embedding  $\mathfrak v\sub \mathfrak 
w$.   We denote the inverse limit of this system simply by $\motif 
X(\mathfrak x)$. It follows from the definition of morphisms of limit points 
that $\motif X$ becomes a functor on the category of limit points. In other 
words, any presheaf on $\fatpoints\fld$ extends to a presheaf on $\limfat
\fld$; this principle will simply be called \emph{continuity}. Since inverse 
limits commute with presheafs, one easily verifies that if $s\colon \motif X
\to \motif Y$ is a morphism of presheafs (=natural transformation), then for 
any limit point $\mathfrak x$, this induces a map $\motif X(\mathfrak x)
\to \motif Y(\mathfrak x)$, showing that extension by continuity is 
functorial. 
The extension of a (pro-)representable functor on $\fatpoints\fld$ to $
\limfat\fld$ will again be called \emph{(pro-)representable}.

If $\categ X$ is a point system in $\fatpoints\fld$ with limit $\mathfrak 
x$, and if $\fat$ is any fat point with canonical morphism $j_\fat\colon\fat\to \op{Spec}\fld$, then the base change $j_\fat^*\categ X$ 
consisting of all $\fat\mathfrak v$ for $\mathfrak v\in\categ X$ is again a 
point system, whose limit we simply denote by $\fat\mathfrak x$ (the 
reader can check that this defines a product in the category $\limfat\fld$). 
Repeating this argument on the first factor then shows that we may even 
multiply any two limit points. However, this multiplication does no longer 
behave as well as before. For instance, since the base change   $
j_\fat^*(\fatpoints\fld)$ by any fat point $\fat$ is equal to the whole category $\fatpoints\fld$, we 
get $\fat\univpoint=\univpoint$. 

Let $j_{\mathfrak x}\colon\mathfrak x\to \op{Spec}\fld$ be the structure  morphism of the limit point $\mathfrak x$. 
Strictly speaking, as this is only a direct limit of finite, flat morphisms, the theory of diminution does not apply, and neither  that of augmentation. Nonetheless, without going into   details, one could develop the theory under this weaker condition, although we will only give an ad hoc argument  in the case  we need it.   So, given a  sieve $\motif X
$ on $\limfat\fld$, we   define  $\arc{\mathfrak x}{\motif X}:=\arc{j_\mathfrak x^*}{\arc{(j_{\mathfrak x})_*}{ \motif X}}$ 
  at a limit point $\mathfrak y$ as the set   $\motif X(\mathfrak x \mathfrak y)$, where we view $\mathfrak x\mathfrak y$ again as a limit point (over $\fld$).

\begin{lemma}\label{L:arcrep}
For any limit point $\mathfrak x$ and any $\fld$-scheme $X$, the base 
change $\mathfrak x^*\func X$ is pro-representable, by the so-called 
\emph{arc scheme $\arc{\mathfrak x}X$ along $\mathfrak x$}. 
\end{lemma}
\begin{proof}
Let $\mathfrak x$ be the direct limit of the directed subset $\categ X\sub
\fatpoints\fld$. Suppose first that $X$ is affine. Since the $\arc{\mathfrak 
w} X$ for ${\mathfrak w}\in\categ X$ form an inverse system of affine 
schemes, their inverse limit is a well-defined affine scheme $\tilde X$  with coordinate ring the direct limit of the coordinate rings of the arc 
schemes along fat points in $\categ X$. By continuity, it suffices to verify 
that $\arc{\mathfrak x}{\func X}=\func{\tilde X}$ on $\fatpoints
\fld$.  To this end,  fix a fat point $\fat$. From
\begin{align*}
(\arc{\mathfrak x}{\func X})(\fat)=  X(\mathfrak x\fat)&=
\varprojlim_{\mathfrak w\in\categ X} X(\mathfrak w\fat)\\
&=\varprojlim_{\mathfrak w\in\categ X}\op{Mor}_\fld(\fat, 
\arc{\mathfrak w}X)\\
&=\op{Mor}_\fld(\fat,\varprojlim_{\mathfrak w\in\categ X}
\arc{\mathfrak w}X)\\
&=\op{Mor}_\fld(\fat, \tilde X)=\tilde X(\fat),
\end{align*}
where we used the  universal property of inverse limits in the third line, the 
claim now follows. The general case follows from this by the open nature of 
arc schemes (Theorem~\ref{T:openarc}) and the fact that if $X$ admits an open affine covering of 
cardinality $N$, then so does any arc scheme $\arc\fat X$ by base change. 
\end{proof}

\section{Extendable arcs}\label{s:extarc}
Let $\complet Y$ be a formal completion, viewed as the limit point of the 
germs $\jet OnY$, and let $X$ be a $\fld$-scheme. By Lemma~
\ref{L:arcrep}, we have an associated arc scheme $\arc{\complet Y}X$.  For each $n$, we have a canonical map $
\arc{\complet Y}X\to \arc{\jet OnY}X$, which in general is not surjective 
(it is so, by Theorem~\ref{T:fibmot}, when $X$ is smooth). To study this image, 
we make the following definitions.

Given a closed embedding $\mathfrak v\sub\mathfrak w$, the image sieve 
given by the canonical map $\arc{\mathfrak w}X\to \arc{\mathfrak v}X$  is 
called the \emph{sieve of $\mathfrak w$-extendable arcs on $X$ along $
\mathfrak v$}, and will be denoted $\imarc{\mathfrak w}{\mathfrak v}X$. 
By construction, it is sub-\zariski. Let $\arc nX:=\arc{\jet OnY}X$, and $
\imarc mnX:=\imarc{\jet OmY}{\jet OnY}X$, for $m\geq n$. Since the map 
$\arc{\complet Y}X\to \arc nX$ is not of finite type, the corresponding 
image sieve, denoted $\imarc{\complet Y}nX$ and called the \emph{$n$-th 
order $\complet Y$-extendable arcs on $X$}, may fail to be sub-\zariski. We 
do have:

\begin{theorem}\label{T:extgermformal}
For each $\fld$-scheme $X$, for each formal completion $\complet Y$ of a 
closed subscheme $Y\sub\affine\fld n$ at a point, and for each $n$, the 
$n$-th order extendable arcs on $X$ along this formal completion, $
\imarc{\complet Y}nX$, is a formal motif.  
\end{theorem}
\begin{proof}
Without loss of generality, we may assume that $\complet Y$ is the completion of $Y$ at the origin. 
It is clear that $\imarc{\complet Y}nX$ is the intersection of all $\imarc 
mnX$, for $m\geq n$. To show that it is a formal motif, it suffices to show 
that its complement is locally sub-\zariski. Since each $\imarc mnX$ is sub-\zariski, this will follow if we can show that for each fat point $\fat$, there 
is some $m_{\fat}$ such that $\imarc{\complet Y}nX(\fat)=\imarc 
{m_{\fat}}nX(\fat)$. 

Recall that for $(R,\maxim)$ a quotient of a power series ring $\pow\fld \xi
$ modulo an ideal generated by polynomials, we have uniform strong Artin 
Approximation, in the sense that for any polynomial system of equations 
$f_1=\dots=f_s=0$ and every $n$, there exists some $N$, such any solution 
of $f_1=\dots=f_s=0 $ in $R/\maxim^N$ is congruent modulo $\maxim^n$ 
to a solution in $R$: see for instance \cite[Theorem 7.1.10]{SchUlBook}, 
where the proof is only given for the power series ring itself, but 
immediately generalizes to any quotient by a polynomial ideal, whence in 
particular to the stalk of the formal completion $\complet Y$. This means 
that $\imarc{\complet Y}nX(\fld)=\imarc mnX(\fld)$, for some $m\geq n$. 
To obtain a similar identity over an arbitrary fat point $\fat$, we apply the 
same result but replacing $X$ by the arc scheme $\arc{\fat}X$, yielding the 
existence of a $m_{\fat}\geq n$ such that 
$$
\imarc{\complet Y}nX(\fat)=\imarc{\complet Y}n{(\arc{\fat}X)}(\fld)=
\imarc {m_{\fat}}n{(\arc{\fat}X)}(\fld) =\imarc {m_{\fat}}nX(\fat), 
$$
as required.
\end{proof}

\begin{remark}\label{R:extgermformal}
Since we may no longer have the required strong Artin Approximation   
estimate, I do not know whether this result generalizes to arbitrary limit 
points, that is to say, is $\imarc{\mathfrak y}{\mathfrak x}X$ a formal 
motif, for limit points $\mathfrak x\leq\mathfrak y$. The first case to look 
at is when $\mathfrak x$ is a fat point and $\mathfrak y$ is bounded (but 
not a formal completion).  
\end{remark}

\section{Motivic generating  series}\label{s:motseries}

Although we can work in greater generality, we will assume once more that our base scheme is an \acf\ $\fld$.

\subsection*{Motivic Igusa-zeta series}
For any   $\fld $-scheme $X$ and any closed germ $(Y,P)$, we   define the 
\emph{motivic Igusa-zeta series of  $X$ along the germ $(Y,P)$} as the 
formal power series
$$
\igumot XYP(t):=\sum_{n=1}^\infty\lef^{-d\cdot j^n_P(Y)}    
\class{ \arc{\jet PnY}X} \ t^n
$$
in $\pow{\grot{\funcinf \fld}_\lef}t$, where $d$ is the dimension of $X$ 
and $j^n_P(Y)$ the length of the $n$-th jet $\jet PnY$ (which is also equal to the 
Hilbert-Samuel function of $\loc_{Y,P}$ for large $n$). This definition 
generalizes the one in \cite{DLIgu} or \cite[\S4]{DLDwork}, when we take 
the germ of a point on a line.

\begin{theorem}\label{T:smoothmotint}
If $X$ is a smooth $\fld$-variety and $(Y,P)$ a closed germ, then   
$$
\igumot XYP=\frac{\lef^{-d}\class X t}{1-t},
$$
over $\grot{\funcinf \fld}_\lef$. 
\end{theorem}
\begin{proof}
Since $X$ is smooth, we have 
\begin{equation}\label{eq:smootharc}
\class{\arc{\jet PnY}X}=\class X\lef^{d(j^n_P(Y)-1)}
\end{equation}
by Theorem~\ref{T:fibmot}, from which the assertion follows easily.
\end{proof}

With aid of \eqref{eq:smootharc} applied to affine space, we can write down a more suggestive formula for the motivic Igusa zeta-series
$$
\igumot XYP(t):=\sum_{n=1}^\infty 
\frac{\class{ \arc{\jet PnY}X}}{{ \arc{\jet PnY}\lef^d}} t^n
$$
where $d$ is the dimension of $X$.

\begin{conjecture}\label{C:motigu}
The    motivic  Igusa-zeta series $\igumot XYP$ of a $
\fld$-scheme $X$   along an arbitrary closed germ $(Y,P)$  is rational over $
\grot{\funcinf \fld}_\lef$.
\end{conjecture}

More generally, given any formal motif $\motif X$ on a $\fld$-scheme $X$, we define its Igusa-zeta series along the germ $(Y,P)$ as the formal power series
$$
\igumot {\motif X}YP(t):=\sum_{n=1}^\infty\lef^{-d\cdot j^n_P(Y)}   \cdot
\class{ \arc{\jet PnY}{\motif X}} t^n
$$
and conjecture its rationality. 

\begin{example}\label{E:lnlmIgu}
The present point of view even gives interesting new results over the classical \gr, since we may take the image of the motivic Igusa zeta series in $\grotclass\fld$. Continuing with the calculations made in Example~\ref{E:lnlm}, let $m$ be a  positive integer and consider the image of $\op{Igu}(\mathfrak l_m):=\igumot{\mathfrak l_m}{\affine \fld1}O$ over the classical \gr. This amounts to taking the reduced scheme underlying each arc scheme $\arc{\mathfrak l_n}{\mathfrak l_m}$, and as shown above, this reduction is just $\affine\fld{n-\round nm}$. Write $n=sm-r$ for some unique $s\geq 1$ and $0\leq r<m$, so that $\round nm=s$. Over $\grotclass\fld$, we have
$$
\op{Igu}(\mathfrak l_m)=\sum_{r=0}^{m-1}\sum_{s=1}^\infty \lef^{sm-r-s}t^{sm-r}=\frac{\sum_{r=0}^{m-1}(\lef t)^{-r}}{(1-\lef^{m-1}t^m)}.
$$
In particular, whenever Question~\ref{C:nonredfibmot}   holds affirmatively,   the image of the motivic Igusa zeta series  of the fat point would be rational over the classical \gr. Skipping the easy calculations, we have for instance that
$$
\op{Igu}(\fat)=  \frac{t+\lef t^2}{(1-\lef t^2)} 
$$ 
where $\fat$ is the fat point with coordinate ring $\pol\fld{\var,\vary}/(\var^2,\vary^2)$. Interestingly, this is also the Igusa zeta-series of the fat point with coordinate ring  $\pol\fld{\var,\vary}/(\var^2,\var\vary,\vary^2)$.
\end{example}

\subsection*{Motivic Hilbert series}
Given a motivic site $\categ M$, we let $\categ M_0$ be its restriction to the subcategory of zero-dimensional schemes, that is to say,  the union of all $\restrict{\categ M}Z$, where $Z$ runs over all zero-dimensional $\fld$-schemes. As the product of two 
zero-dimensional schemes is again zero-dimensional,   $\categ M_0$ is   a partial motivic site, and hence has an associated \gr\ $\grotzero {\categ M}:=\grot{\categ M_0}$, called the \emph{\gr\ of $\categ M$ in dimension zero}. There is a natural \homo\ $\grotzero {\categ M}\to \grot{\categ M}$, which in general will fail to be injective, as there are a priori more relations in the latter \gr. In particular, applied to (sub-)\zariski\ or formal motives, we get the corresponding \gr{s} in dimension zero $\grotzero{\funcalg\fld}$,  $\grotzero{\funcsub\fld}$, and  $\grotzero{\funcinf\fld}$.

\begin{proposition}\label{P:zerosch}
The \zariski\ \gr\  in dimension zero, $\grotzero{\funcalg\fld}$, is freely generated, as an additive group, by the isomorphism classes of fat points.
\end{proposition}
\begin{proof}
 A  zero-dimensional scheme $Z$ is a disjoint union $ \fat_1\sqcup\dots\sqcup\fat_s$ of fat points (in a unique way). 
  Moreover, since any fat point is strongly connected, this unique decomposition in fat points is its \zariski\ decomposition. Hence, by (the proof of ) Theorem~\ref{T:classinvmot}, the image of $\class Z$ under  the composition $\grotzero{\funcalg\fld}\to\grot{\funcalg\fld}\map \delta\Gamma$ is $\sym{\fat_1}+\dots+\sym{\fat_s}$, where as before, $\Gamma$ is the free Abelian group on isomorphism classes of strongly indecomposable $\fld$-schemes. Since $\class Z=\class{\fat_1}+\dots+\class{\fat_s}$ in $\grotzero{\funcalg\fld}$, this composition is an isomorphism.
\end{proof}

Let $(X,P)$ be a closed germ over $\fld$.  
For $t$ a single variable, we    define the
\emph{motivic Hilbert series}  as the series
$$
\op{Hilb}^{\text{mot}}_{(X,P)}:=\sum_{n=1}^\infty \class{\jet PnX}\ t^n
$$
in $\pow{\grotzero{\funcalg \fld}}t$. By Proposition~\ref{P:zerosch}, we may 
extend the length function $\ell$  to a \homo\ on the \gr\ in dimension zero. If we extend this further to the power series ring  $\pow{\grotzero{\funcalg \fld}}t$ by letting it act
on the coefficients of a power series,
then $\ell(\op{Hilb}_P(X))$ is a rational function in $\pow\zet t$
by Hilbert-Samuel theory (it is the first difference of the classical Hilbert series
of $X$ at   $P$). This begs the question whether $\op{Hilb}^{\text{mot}}_{(X,P)}$ itself is   rational over $\grotzero{\funcinf \fld}$ or $\grot{\funcinf\fld}$, or possibly their  localizations at $\lef$ (obviously, it will not be rational over the \zariski\ \gr{s} by Proposition~\ref{P:zerosch}). Although no longer specializing to a classical series, we may also consider the more general series
$$
\op{Hilb}^{\text{mot}}_{(X,x)}:=\sum_{n=0}^\infty \class{\jet xnX}\ t^n
$$
where $x$ is any  point on $X$ (not necessarily closed). 

\begin{theorem}\label{T:complinv}
Let $\fld$ be an \acf\ of cardinality $2^\gamma$ for some infinite cardinal $\gamma$ (under the Generalized Continuum hypothesis this means any uncountable \acf). The assignment $\complet X_P\mapsto \op{Hilb}^{\text{mot}}_{(X,P)}$ is a complete invariant in the sense that  for closed germs $(X,P)$ and $(Y,Q)$ over $\fld$, their completions  $\complet X_P$ and $\complet Y_Q$ are abstractly isomorphic (that is to say, over some subfield of $\fld$) \iff\ they have the same motivic Hilbert series in $\grotzero{\funcalg \fld}$.
\end{theorem}
\begin{proof}
Immediate from Proposition~\ref{P:zerosch} and \cite[Theorem~1.1 and \S8.9]{SchClassSing}.
\end{proof}

\subsection*{Motivic Hilbert-Kunz series}
Assume that $\fld $ has \ch\ $p$. Recall that  for a given closed subscheme  $Y\sub X$, we defined in \S\ref{s:adj} its  Frobenius transform  in $X$ as the pull-back  $\frob_X^*Y:=\frob_*X\times_XY$   of $Y$ along $\frob_X$. We may also take the pull-back with respect to the powers $\frob^n_X$ of the Frobenius, yielding the \emph{$n$-th Frobenius transform} $\frob_X^{n*}Y$. 
If $Y$ has dimension zero, then so does any of its Frobenius transforms, and so the following series, called the \emph{motivic Hilbert-Kunz series}, 
$$
\op{HK}^{\text{mot}}_Y(X):=\sum_{n=1}^\infty \class{\frob_X^{n*}Y}\ t^n
$$
is a well-defined series in  $\pow{\grotzero{\funcalg \fld}}t$. Taking the length function $\ell$ yields the classical Hilbert-Kunz series, of which not too much is known (one expects it to be rational). Of course, we could also take $Y$ to be of higher dimension, and get the corresponding motivic Hilbert-Kunz series in  $\pow{\grot{\funcalg \fld}}t$. 

Instead of transforms we could take iterated Frobenius motives  $\motif F^n_Y:=\arc{\frob^n}Y$ given as the image sieve of the $n$-th relative Frobenius $\frob_{\affine\fld s}^n\times\tuple1_Y$, for some closed immersion $Y\sub\affine\fld s$, in case $Y$ is affine, and by glueing for the general case, yielding a series
$$
\op{Fr}^{\text{mot}}(Y):=\sum_{n=1}^\infty \class{\arc{\frob^n}Y}\ t^n
$$
in $\pow{\grotzero{\funcalg \fld}}t$.
Note that $\motif F^n_Y(\fld)=Y(\frob^n \fld)=Y(\fld)$ by Theorem~\ref{T:frobadj}, so that this series becomes the  rational function $\class Y/(1-t)$ in $\grotclass \fld$.

\subsection*{Motivic Milnor series}
Let $(Y,P)$ be a closed germ with formal completion $\complet Y_P$.  
Proposition~\ref{P:formsieve} exhibits $\func{\complet Y_P}$  as a formal 
motif by means of its jets. However, this is not the only way to locally 
approximate it with \zariski\ subsieves. Given a system of parameters $
\xi_1,\dots,\xi_e$ in $\loc_{Y,P}$ (that is to say, a tuple of length $e=
\op{dim}(\loc_{Y,P})$ generating an ideal primary to the maximal ideal), let 
$\mathfrak y_n$ be the fat point with coordinate ring $B_n:=\loc_{Y,P}/
(\xi_1^n,\dots,\xi_e^n)\loc_{Y,P}$, and $j_{\mathfrak y_n}\colon \mathfrak y_n\to \op{Spec}\fld$ the canonical morphism. The reader can check that given a fat 
point $\fat$, there exists some $n$ such that $\mathfrak y_n(\fat)=
\complet Y_P(\fat)$, that is to say, 
$\complet Y_P$ is the limit point corresponding to the direct system $\{\mathfrak y_n\}_n$ (see \S\ref{s:limpt}). Recall that by the Monomial Theorem, the element $
(\xi_1\cdots\xi_e)^{n-1}$ is a non-zero element in the socle of $B_n$ 
(meaning that the ideal it generates has length one). 

Let $X\sub\affine\fld {d+1}$ be the ($d$-dimensional) hypersurface with 
equation $f(\var)=0$ and let $Y_n\sub\affine{\mathfrak y_n}{d+1} $ be the 
deformed hypersurface with equation $f(\var)-(\xi_1\cdots\xi_e)^{n-1}=0$. 
In other words, it is the general fiber in the family $W_n\sub\affine{\mathfrak y_n}{d+2}$ over the last coordinate $z$,  given 
by the equation $f(\var)-z(\xi_1\cdots\xi_e)^{n-1}=0$, whereas the special 
fiber is just  the base change   $j^*_{\mathfrak y_n}X$. We define the \emph{$n$-th order 
Milnor fiber of $X$ along the germ $(Y,P)$} as the deformed arc space
$$
M_n(X):=\arc{j_{\mathfrak y_n}^*}{Y_n}.
$$
Hence, by Proposition~\ref{P:defarcfam}, with $j_{\affine\fld{d+2}}\colon \affine{\mathfrak y_n}{d+2}\to \affine{\fld}{d+2}$ the base change of $j_{\mathfrak y_n}$, the specializations of the relative arc scheme are 
\begin{equation}\label{eq:specrelarc}
\begin{aligned}
(\arc{j_{\affine\fld{d+2}}^*}{W_n})_a &= \arc{\mathfrak y_n}X\ \qquad\text{if $a$ is the zero section;}\\
&=M_n(X)\qquad\text{otherwise.}
\end{aligned}
\end{equation}
We define the associated Milnor series
$$
\mil XYP(t):=\sum_{n=1}^\infty \lef^{-d\ell(\mathfrak y_n)}\class{M_n(X)}t^n
$$
as a power series over   $\grot{\funcinf \fld}_\lef$. When $(Y,P)$ is the germ 
of a point on a line, we get the \zariski\ variant of the series introduced by 
Denef-Loeser et al., and by \eqref{eq:specrelarc}, this series can be viewed as a deformation of the motivic Igusa-zeta series.  Therefore, in view of   Conjecture~\ref{C:motigu}, we expect the 
motivic Milnor series also to be rational, and in fact, as a rational function, to 
have degree zero. Assuming this to be true, we can calculate the limit of this 
series when $t\to\infty$, and this conjectural limit, presumably in $\grot{\funcinf 
\fld}_\lef$, will be called the \emph{motivic Milnor fiber} of $X$ along the 
closed germ $(Y,P)$.

\subsection*{Motivic Hasse-Weil series}
Another important generating series in algebraic geometry whose
rationality---proven by Dwork in \cite{DworkRat}---is postulated to
be motivic, is the Hasse-Weil series of a scheme over a finite
field $\mathbb F_q$: its general coefficient is the number of
rational points over the finite extensions $\mathbb F_{q^n}$. To turn this into an abstract counting principle, we use
the inversion formula relating the number of degree $n$ effective
zero cycles on $X$ to the number of rational points in an extension
of degree $n$, and observe that the former cycles are in one-one
correspondence with  the rational points on the $n$-fold symmetric
product $X^{(n)}$ of $X$ (given as the quotient of $X^n$ modulo the
action of the symmetric group  on $n$-tuples). Therefore, following  Kapranov  \cite{Kap}, we propose  the following motivic variant, the \emph{Motivic Hasse-Weil series}:
  $$
  \op{HW}^{\text{mot}}_X:=\sum_{n=0}^\infty\class{X^{(n)}}\ t^n,
  $$
  as a power series over $\grot{\funcinf \fld}_\lef$. Kapranov himself proved  rationality of the image of this series over $\grotclass\fld_\lef$,  as well as a functional equation, for   certain smooth,
projective irreducible  curves, but the general case is still open.
  We know from work of Larsen and Lunts on smooth
surfaces (\cite{LLMot}), that, in general,  this cannot hold over the \gr\
itself: in \cite{LLRat}, they show that rationality over
the \gr\ is equivalent with the complex surface having negative
Kodaira dimension.  It is therefore natural to conjecture the same properties for our motivic variant $\op{HW}^{\text{mot}}_X$. 

\subsection*{Motivic Poincare series}
Given a closed germ $(Y,P)$ with formal completion $\complet Y$, viewed as a limit point, and a $\fld$-scheme $X$, 
by Theorem~\ref{T:extgermformal}, we can now define the \emph{motivic 
Poincar\'e series of $X$ along $(Y,P)$} as the formal series 
$$
\poinc XYP(t):=\sum_{n=1}^\infty \lef^{-d\ell(\jet PnY)}
\class{\imarc{\complet Y}nX}t^n
$$
over $\grot{\funcinf \fld}_\lef$, where $d$ is the dimension of $X$. Denef and Loeser proved in \cite{DLArcs} that along the 
germ of a point on the line, the image of this series in the localized classical 
\gr\ is rational, provided $\fld$ has \ch\ zero. It is therefore natural to 
postulate:

\begin{conjecture} 
For any closed germ $(Y,P)$ and any $\fld$-scheme $X$, the associated 
motivic Poincar\'e series $\poinc XYP$ is rational over $\grot{\funcinf 
\fld}_\lef$. 
\end{conjecture}

Given $X$ and a formal completion $\complet Y$, we may ask for each $n$, 
which are the fat points $\fat$ containing the $n$-th jet $\mathfrak j_n:=
\jet OnY$ such that $\imarc{\complet Y}{\mathfrak j_n}X\sub \imarc{\fat}
{\mathfrak j_n}X$, that is to say, when are $\complet Y$-extendable arcs 
also $\fat$-extendable? For instance, if $\complet Y$ is the completion of 
the affine line, then by Theorem~\ref{T:fibmot}, we can extend along any jet 
of a non-singular germ $(W,O)$, since there exist closed immersions $
\mathfrak j_n\sub \jet OnW\sub\mathfrak j_n^d$, where $d$ is the 
dimension of $(W,O)$.  
However, I do not know whether we can   extend along the fat point given 
by, say,  $x^4=y^4=x^3-y^2=0$. For which schemes $X$ can every $\complet 
Y$-extendable arc be extended along any fat point? This is true if $X$ is 
smooth, but are there any other cases?

\section{Motivic integration}\label{s:motint}
Unlike the Kontsevitch-Denef-Loeser motivic integration, we will only define integration 
on the (truncated) arc schemes. We will work over the localized \gr\ $\grotmot:=\grot{\funcinf \fld}_\lef$, for $\fld$ an \acf. Before we develop the theory, we discuss a naive approach.

\subsection*{Motivic measure}
We fix a fat point $\fat$. Our goal is to define a motivic measure $\mu_\fat$ on formal motives. To this end, we define
$$
\mu_\fat(\motif X):= \class{\arc\fat{\motif X}}
$$
 in $\grotmot$. In particular, this measure does not depend on the ambient space of $\motif X$, only on its germ. Using Theorem~\ref{T:arcmot}, we can extend  the motivic measure to an endomorphism on $\grotmot$. We would like to normalize this measure, with the ultimate goal---which, however, we do not discuss in this paper---to make the comparison   between different fat points and take limits. One way to normalize   is to make the value weightless (in the sense of dimension), by 
 $$
 \bar\mu_\fat(\motif X):= \frac{\class{\arc\fat{\motif X}}}{\lef^{\op{dim}(\arc\fat{\motif X})}}
 $$
 This, of course, is no longer additive.  Below, however,    we will normalize differently, by fixing an ambient space.
 
 Following integration theory practice, we would like to say that  
 $$
 \int \tuple 1_{\motif X}\ d_\fat x:= \mu_\fat(\motif X):= \class{\arc\fat{\motif X}}
 $$
 and   extend this to arbitrary step functions. Here, a \emph{step function} would be a formal, finite sum $s=\sum g_i\tuple 1_{\motif X_i}$ with $g_i\in\grotmot$ and $\motif X_i$ formal motives. However, how to interpret this as a function? As usual, we should do this at each fat point $\mathfrak w$, and interpret $\tuple 1_{\motif X}(\mathfrak w)$ as the \ch\ function on $X(\mathfrak w)$ of $\motif X(\mathfrak w)$, where $X$ is an ambient space of $\motif X$. Likewise, provided $X$ is an ambient space for all $\motif X_i$, we let  $s(\mathfrak w)$ be   the function   $X(\mathfrak w)\to \grotmot$ associating to a $\mathfrak w$-rational point $a\in X(\mathfrak w)$ the sum of all $g_i$  for which $a\in \motif X_i(\mathfrak w)$. However, the main obstruction is that this point-wise defined function is in general   not functorial. 
 The reason is the non-functorial nature of fibers, which in turn stems from the lack of complements in categories---note that the complement of any fiber is the union of the other fibers. To recover functoriality, we work over a subcategory of fat points:

\subsection*{Flat and split points}\label{s:motintsp}
More precisely, let $\flatpoints\fld$ and $\splitpoints\fld$ be the respective categories of \emph{flat points} and \emph{split points} over $\fld$, whose objects   are fat points over $\fld$ and whose morphisms are respectively flat   and split epimorphisms. Recall that a morphism $\varphi\colon Y\to X$ is called a \emph{split epimorphism} if there exists a morphism, also called a \emph{section},  $\sigma\colon X\to Y$ such that $\varphi\sigma$ is the identity on $X$. Any split epimorphism is (faithfully) flat, so that $\splitpoints\fld$ is a subcategory of $\flatpoints\fld$. Note that each structure morphism $\fat\to\op{Spec}\fld$ is a split epimorphism, and by base change, so is each projection map $\fat\mathfrak w\to\mathfrak w$. 

We will call a contravariant functor $\motif X$ from $\flatpoints\fld$   (respectively,  from $\splitpoints\fld$) to the category of sets a \emph{flat} (respectively, a \emph{split}) presheaf. If, moreover,  we have an inclusion morphism $\motif X\sub \func X$, where $\func X$ is the restriction of  the representable functor of a $\fld$-scheme $X$, we call $\motif X$  a \emph{flat} (respectively, a \emph{split}) sieve.
In particular, ordinary presheafs or sieves (that is to say, defined on $\fatpoints\fld$) when restricted to $\splitpoints\fld$ are  split---and to emphasize this, we may call them \emph{full} sieves---, but as the next result shows, not every split sieve is the restriction of a full sieve:

\begin{proposition}\label{P:complflat}
The complement of a \zariski\ motif $\motif X\sub \func X$  is a flat  sieve. The complement of a formal motif $\motif X\sub \func X$   is a split  sieve. 
\end{proposition}
\begin{proof}
Any (full)  \zariski\ motif is the union of closed subsieves, and hence its complement is the intersection of complements of closed subsieves. Since the intersection of (flat or split) sieves is again a sieve, we only need to verify that the complement of a single closed subsieve $\func Y\sub\func X$ is a flat sieve. The only thing to show is functoriality, so let $\mathfrak w\to \fat$ be a flat morphism of fat points. We have to show that under the induced map $X(\fat)\to X(\mathfrak w)$ any $\fat$-rational point $a$ not in $Y(\fat)$ is mapped to a point not in $Y(\mathfrak w)$. Since fat points are affine, we may replace $X$ by an affine open, and so assume from the start that it is affine with coordinate ring $A$. Let $I$ be the ideal defining $Y$, and let $R$ and $S$ be the coordinate rings of $\fat$ and $\mathfrak w$ respectively. The $\fat$-rational point $a$ corresponds to a morphism $A\to R$; it does not belong  to $Y(\fat)$ \iff\  the image $IR$  of $I$ under $A\to R$ is non-zero. Suppose towards a contradiction that  $a\in Y(\mathfrak w)$, so that $IS=0$. Since $R\to S$ is by assumption flat, whence faithfully flat, we must have $IR=IS\cap R=0$, contradiction.

Assume next that $\motif X$ is a sub-\zariski\ motif, that is to say, $\motif X=\fim\varphi$ for some morphism $\varphi\colon Y\to X$. Let  $\lambda\colon\mathfrak w\to \fat$ be a split epimorphism of fat points  and let  $a\colon\fat\to X$ be a $\fat$-rational point  outside $\fim\varphi(\fat)$.  We have to show that the image   $a\after\lambda$ of $a$ in $X(\mathfrak w)$ does not lie in the  image of $\varphi(\mathfrak w)$. Towards a contradiction, assume the opposite, so that $a\after\lambda$ factors through $Y$, giving rise to a commutative diagram
\commdiagram [YX] {\mathfrak w} b Y\lambda\varphi\fat a {X.}
 By assumption, there exists a section $\sigma\colon \fat\to \mathfrak w$ so that $\lambda\sigma$ is the identity on $\fat$. Let $\tilde b$ be the image of $b$ under $Y(\sigma)\colon Y(\mathfrak w)\to Y(\fat)$, that is to say, $\tilde b=b\after\sigma$. The image of $\tilde b$ under $\varphi(\fat)\colon Y(\fat)\to X(\fat)$ is by \eqref{YX} equal to 
$$
\varphi(\fat)(\tilde b)=\varphi\after\tilde b=\varphi\after b\after\sigma=a\after\lambda\after\sigma=a
$$
showing that $a$ lies in the image of $\varphi(\fat)$, contradiction.

Lastly, assume that $\motif X$ is formal, so that there exists for each fat point $\fat$ a sub-\zariski\ motif $\motif Y_\fat\sub\motif X$ such that $\motif Y_\fat(\fat)=\motif X(\fat)$. Let $\lambda\colon\mathfrak w\to\fat$ be a split epimorphism. Since $\mathfrak Y_{\mathfrak w}\sub\motif X$, we have $-\motif X(\fat)\sub -\motif Y_{\mathfrak w}(\fat)$. By what we just proved, $-\motif Y_{\mathfrak w}(\fat)$, is sent under   $X(\lambda)\colon X(\fat)\to X(\mathfrak w)$ inside $-\motif Y_{\mathfrak w}(\mathfrak w)$, and by construction, the latter is equal to $-\motif X(\mathfrak w)$. A fortiori,   $-\motif X(\fat)$ is then sent inside $-\motif X(\mathfrak w)$,  proving the assertion.
\end{proof}

\begin{remark}\label{R:split}
It is important to note that we may not apply this argument to an arbitrary split sieve, since   a section of a split morphism  is not split and hence does not induce a morphism on the rational points of the split sieve. The point in the above argument is that formal motives are presheafs on the full category of fat points, and hence any section does induce a map between their rational points.
\end{remark}

%

We call any Boolean combination of closed subsieves a \emph{flat-\zariski} motif. Our aim is to define   motivic sites of \zariski, sub-\zariski\ and formal motives   with respect to flat and/or split points,   but without changing the corresponding \gr. A priori, there might be more morphism, that is to say, natural transformations, in this   restricted context, and to circumvent this issue, we only allow morphisms that extend to true motives. More precisely, we  define the \emph{flat-\zariski} 
motivic site   $\funcalgflat\fld$, as the category with objects all flat-\zariski\ motives, and with morphisms all natural transformations $s\colon\motif Y\to \motif X$ of flat \zariski\ motives which extend to a morphism of \zariski\ motives in the sense that there are \zariski\ motives $\motif X\sub\motif X'$ and $\motif Y\sub\motif Y'$ and a morphism of \zariski\ motives $s'\colon\motif Y'\to \motif X'$ whose restriction to $\motif Y$ is $s$.    To not introduce unwanted isomorphisms, we moreover require that if $s$ is injective, then so must its extension $s'$ be.  Likewise, we call any Boolean combination of  sub-\zariski\ (respectively, formal) motives a \emph{split-sub-\zariski} (respectively, a \emph{split-formal}) motif, and we define  the \emph{split-sub-\zariski} motivic site   $\funcsubsplit\fld$ (respectively, the \emph{split-formal} motivic site   $\funcinfsplit\fld$), as the category with objects all split-sub-\zariski\ (respectively,   split-formal) motives, and as  morphisms   all
   natural transformations which extend to a morphism of sub-\zariski\ (respectively, formal) motives, with injective morphisms extending to injective ones.   All these sites satisfy the same properties as ordinary motivic sites, apart from being defined only over a restricted category, but have the additional property that their restriction to any scheme is a Boolean lattice. At any rate, we can   define their corresponding \gr{s}. 
   
\begin{proposition}\label{P:grsplit}
We have equalities of \gr{s}, $\grot{\funcalgflat\fld}=\grot{\funcalg\fld}$, $\grot{\funcsubsplit\fld}=\grot{\funcsub\fld}$, and $\grot{\funcinfsplit\fld}=\grot{\funcinf\fld}$. 

Moreover, for each fat point $\fat$, the arc sheaf of a flat \zariski\ (respectively, split sub-\zariski, or split formal) motif exists, and is again of that form. The induced action on the corresponding \gr\ is equal to that of  $\arc\fat{(\cdot)}$.
\end{proposition}   
  \begin{proof}
  I will only give the argument for the case of most interest to us, the formal motives, and leave the remaining cases, with analogous proof, to the reader. Before we do this, let us first discuss briefly Boolean lattices. Let $\mathcal B$ be a Boolean lattice.  Given a finite collection of subsets $X_1,\dots,X_n\in \mathcal B$, and an $n$-tuple $\varepsilon=\rij\varepsilon n$ with entries $\pm 1$, let $X_\epsilon$ be the subset given by the intersection of all $X_i$ with $\varepsilon_i=1$ and all $-X_i$ with $\varepsilon_i=-1$. Then any element in the Boolean sublattice $\mathcal B\rij Xn$ of $\mathcal B$ generated by $X_1,\dots, X_n$ is a disjoint union of the $X_\varepsilon$. In particular, if all $X_i$ belong to a sublattice $\mathcal L\sub\mathcal B$, then any element in $\mathcal B\rij Xn$ is a disjoint union of sets of the form $C-D$ with $D\sub C$ in $\mathcal L$. 

 We now define a map $\gamma$ from the free Abelian group $\pol\zet{\funcinfsplit\fld}$ to $\grot{\funcinf\fld}$ as follows. By the above argument, a typical element in $\funcinfsplit\fld$ is a disjoint union of split formal motives of the form $\motif X-\motif Y$ with $\motif Y\sub\motif X$ (full) formal motives. We define its $\gamma$-value to be the element $\class{\motif X}-\class{\motif Y}$. This is well-defined, for if it is also equal to a difference of motives $\tilde{\motif X}-\tilde{\motif Y}$ then one easily checks that $\motif X\cup\tilde{\motif Y}=\tilde{\motif X}\cup \motif Y$ and $\motif X\cap\tilde{\motif Y}=\tilde{\motif X}\cap \motif Y$, so that 
 $$
 \class{\motif X}+\class{\tilde{\motif Y}}=\class{\tilde{\motif X}} +\class{\motif Y}.
 $$
  We extend this to disjoint sums by taking the sum of the disjoint components, and then extend by linearity, to the entire free Abelian group. 
It is not hard to verify that $\gamma$ preserves all scissor relations. So we next check that it preserves also isomorphism relations. We may again reduce to an isomorphism of the form $s\colon \motif X-\motif Y\to \tilde{\motif X}-\tilde{\motif Y}$ with $\motif Y\sub\motif X$ and $\tilde{\motif Y}\sub\tilde{\motif X}$ formal motives. By assumption, $s$ extends to an injective morphism $s'\colon \motif X'\to\tilde{\motif X'}$ with $\motif X'$ and $\tilde{\motif X'}$ formal motives. Upon replacing $\motif X$ and $\motif X'$ with their common intersection, we may assume that they are equal. Since $s'$ is injective, it induces an isomorphism between $\motif X$ and its image, as well between  $\motif Y$ and its image. Hence $\motif X-\motif Y\iso s'(\motif X)-s'(\motif Y)$ and, since $s'$ extends $s$,  the latter must be equal to $\tilde{\motif X}-\tilde{\motif Y}$, yielding
  $$
 \gamma(\motif X-\motif Y)= \class{\motif X}-\class{\motif Y}= \class{s'(\motif X)}-\class{s'(\motif Y)}=\gamma(\tilde{\motif X}-\tilde{\motif Y})
 $$
 as we wanted to show. Hence, $\gamma$ induces a map $\grot{\funcinfsplit\fld}\to \grot{\funcinf\fld}$. By construction, it is surjective and  the identity on $\grot{\funcinf\fld}$ (when viewing a full motif as a split motif), showing that it is a bijection. By construction it is also additive, and the reader readily verifies that it preserves products, thus showing that it is an isomorphism.

 For the last assertion, it suffices once more to verify this on a split formal motif of the form $\motif X-\motif Y$ and we set
 $$
 \arc\fat{(\motif X-\motif Y)}:=\arc\fat{\motif X}-\arc\fat{\motif Y}
 $$
 Since the $\mathfrak w$-rational points of both sides are the same, to wit, $\motif X(\fat\mathfrak w)-\motif Y(\fat\mathfrak w)$, for any fat point $\mathfrak w$, this is well-defined, and the assertion then follows from    the universal property of adjunction.
\end{proof}

\begin{remark}\label{R:nonemptysplit}
Consider the flat \zariski\ motif $\func{\mathfrak l_3}-\func{\mathfrak l_2}$. It has no $\fld$-rational points, but is does have a $\mathfrak l_3$-rational point, namely the identity morphism on $\mathfrak l_3$. This example shows that  the analogue of Lemma~\ref{L:empty} does not hold for split formal motives. We do have:
\end{remark}

\begin{lemma}\label{L:arcempt}
If $\motif X$ is a split formal motif and $\fat$ a fat point such that $\arc\fat{\motif X}$ is empty, then   $\motif X$ too  is empty. In particular, all arc maps are injective on each ambient space. 
\end{lemma} 
\begin{proof}
We may again reduce to the case that $\motif X$ is of the form $\motif Y-\motif Z$ with $\motif Z\sub\motif Y$ (full) formal motives. Let $\mathfrak w$ be an arbitrary fat point. The closed immersion $\mathfrak w\sub \fat\mathfrak w$  induces maps $\motif Z(\fat\mathfrak w)\to \motif Z(\mathfrak w)$ and $\motif Y(\fat\mathfrak w)\to \motif Y(\mathfrak w)$. Since composing the closed immersion with the (split) projection $ \fat\mathfrak w\to  \mathfrak w$ is the identity, the two above maps are surjective. Since $\arc\fat{\motif X}$ is the empty motif, it has no $\mathfrak w$-rational points, that is to say, $\motif Z(\fat\mathfrak w)=\motif Y(\fat\mathfrak w)$. Surjectivity then yields that $\motif Z(\mathfrak w)=\motif Y(\mathfrak w)$, whence $\motif X(\mathfrak w)=\empty$. Since this holds for any fat point $\mathfrak w$, we see that $\motif X$ is the empty motif.

To prove the last assertion, assume $\arc\fat{\motif X}=\arc\fat{\motif Y}$ for $\motif X,\motif Y$ split formal motives on a scheme $X$. By what we just proved, $\motif X-(\motif X\cap \motif Y)$ and $\motif Y-(\motif X\cap \motif Y)$ are both empty, from which the claim now follows.
\end{proof}

From now on, we will work in the largest of these motivic sites, the category of split-formal motives  $\funcinfsplit\fld$, and   we view   the class of any such motif as an element in the localized \gr\ $\grotmot:=\grot{\funcinf \fld}_\lef$. Let $\underline\grotmot$ be the \emph{constant presheaf 
with values in $\grotmot$}, that is to say,  the contravariant functor on the category of split points which associates to any fat point the set $
\grotmot$ and to any split epimorphism of fat points the identity on $\grotmot$.  Given 
a morphism, that is to say, a natural transformation, $s\colon \motif X\to \underline \grotmot $, we define, for each $g
\in \grotmot $,  the \emph{fiber} $\inverse sg$ as the subfunctor of $\motif X$ 
given at each fat point $\fat$ by the fiber $\inverse{s(\fat)}g$ of $s(\fat)
\colon \motif X(\fat)\to \grotmot $ at $g$. If both $\motif X$ and all fibers are 
split formal motifs, and $s$ has only finitely many non-empty fibers, then we call $s$ a \emph{formal  invariant}. 

%

\begin{corollary}\label{C:forminvalg}
The formal invariants on a split formal motif $\motif X$ form an algebra over $\grotmot$.
\end{corollary}
\begin{proof}
Clearly, any multiple of a formal invariant by an element in $\grotmot$ is again a formal invariant. Let $s,t\colon\motif X\to \underline\grotmot$ be formal invariants. We have to show that $s+t$ and $st$ are also formal invariants. Functoriality is easily verified, so we only need to show that the fibers are again split formal motives.  Fix a fat point $\fat$, and  an element $g\in \grotmot$. A $\fat$-rational point $a\in\motif X(\fat)$ lies in $\inverse{(s+t)}g(\fat)$ (respectively, in $\inverse{(st)}g(\fat)$), if $s(\fat)(a)+t(\fat)(a)=g$ (respectively, if $s(\fat)(a)\cdot t(\fat)(a)=g$). Since $s(\fat)$ and $t(\fat)$ have finite image, their are only finitely many ways that $g$ can be written as a sum $p+q$ (respectively, a product $pq$), with $p$ in the image of $s(\fat)$ and $q$ in the image of $t(\fat)$. Hence,  the rational point $a$ lies in  the intersection  $\big(\inverse{s(\fat)}p \big)\cap \big(\inverse{t(\fat)}q\big)$, for one of these finitely many choices of $p$ and $q$. Since a finite union of intersections of split formal motives is again split formal, the result follows.
\end{proof}

\subsection*{Motivic integrals}

Let $X$ be a $\fld$-scheme, $\fat$ a fat point, and $s\colon \motif X\to 
\underline\grotmot $   a formal invariant with $\motif X$ a split formal motif on $X$. We define the \emph{(split) motivic 
integral} of $s$ on $X$ along $\fat$ as
\begin{equation}\label{eq:motint}
\motintspl Xs{\fat}:=\lef^{-dl}\sum_{g\in \grotmot}g\cdot 
\class{\arc\fat{(\inverse sg)}},
\end{equation}
where $d$ is the dimension of $X$ and $l$ the length of $\fat$. Note that 
the sum on the right hand side of  \eqref{eq:motint} is finite by definition, so that $\motintspl Xs{\fat}$ is a well-defined element in $
\grotmot $. At the reduced fat point, $\op{Spec}\fld$, we drop the subscript in the measure, and so the this integral becomes
$$
\motintspl Xs{}:=\lef^{-d}\sum_{g\in \grotmot}g\cdot \class{\inverse sg}.
$$

To  a formal motif  $\motif Y$ on $X$,   we can associate two invariants. 
Firstly, the constant map, denoted again $\motif Y$, which at each fat point 
is the constant map  sending every rational point to $\class{\motif Y}$. One 
easily  calculates that
$$
\motintspl X{\motif Y }{\fat}=\class{\motif Y }\motintspl X{}{\fat}=\class{\motif 
Y}\cdot\lef^{-dl}\cdot \class{\arc {\fat}X} .
$$
In particular, $\motintspl X{}{}=\lef^{-d}\class X$ is the normalized class 
map. It follows from Theorem~\ref{T:arcmot} and Proposition~
\ref{P:dimmot} that the integral $\motintspl X{\motif Y }{\fat}$ only depends 
on the classes of $\motif Y $ and $X$. Moreover, by our previous discussion 
$\motintspl X{}{\fat}$ has positive dimension. 

Secondly, we define the \emph{\ch\ function} $\tuple 1_{\motif Y }$ of $
\motif Y$ by the rule that $\tuple 1_{\motif Y}(\fat)$ is the \ch\ function of 
$\motif Y(\fat)$, that is to say, the map sending a rational point $a\in 
X(\fat)$ to $1$, if $a\in \motif Y(\fat)$, and to zero otherwise, for any fat 
point $\fat$. By Proposition~\ref{P:complflat}, this is a formal invariant. Moreover, any formal invariant can be written as a $\grotmot$-linear combination of \ch\ functions, and, in fact, the decomposition
\begin{equation}\label{eq:fibdecom}
s=\sum_{i=1}^ng_i\tuple 1_{\motif Y_i}
\end{equation}
is unique if   the non-empty formal submotives $\motif Y_i$ are mutually disjoint (note that then necessarily $\motif Y_i=\inverse s{g_i}$), and is called the \emph{fiber decomposition} of $s$.  

Since $\inverse{\tuple 1_{\motif Y}}1= \motif Y$, we get
\begin{equation}\label{eq:motintchar}
\motintspl X{\tuple 1_{\motif Y}}{\fat}=\lef^{-dl}\cdot\class{\arc \fat{\motif Y} }
\end{equation}
Using common practice, we will write
$$
\motintsplrel Xs\fat{\motif X}:=\motintspl X{s\cdot\tuple 1_{\motif X}}\fat.
$$
In this notation, we have
$$
\motintspl Xs\fat=\sum_{g\in \grotmot}g\motintsplrel X{}\fat{\inverse sg}
$$

\begin{proposition}\label{P:addmotint}
For each   $\fld$-scheme $X$ and each    fat point  $\fat$, the motivic integral on $X$ along $\fat$ is a $\grotmot$-linear functional on the $\grotmot$-algebra of formal invariants.
\end{proposition}
\begin{proof}
Motivic integration is clearly preserved under multiplication by a constant $g\in\grotmot$. To prove additivity, we may induct on the number of \ch\ functions, and reduce   to the case of a sum $s+h\tuple 1_{\motif Z}$, that is to say, we have to prove  
\begin{equation}\label{eq:addeq}
\motintspl X{s+h\tuple 1_{\motif Z}}\fat=\motintspl Xs\fat+\motintspl X{h\tuple 1_{\motif Z}}\fat.
\end{equation}
Let \eqref{eq:fibdecom} be the fiber decomposition of $s$.  Since the   fiber decomposition of $s+h\tuple 1_{\motif Z}$ is  then 
$$
\sum_{i=1}^ng_i\tuple 1_{\motif Y_i-\motif Z}+ \sum_{i=1}^n(g_i+h)\tuple 1_{\motif Y_i\cap \motif Z}+ h\tuple 1_{\motif Z-\motif Y}
$$
where $\motif Y$ is the union of the $\motif Y_i$, the left hand side of   \eqref{eq:addeq} is  
$$
\lef^{-dl}(\sum_{i=1}^ng_i\class{\arc\fat{(\motif Y_i-\motif Z)}}+ \sum_{i=1}^n(g_i+h)\class{\arc\fat{(\motif Y_i\cap \motif Z)}}+ h\class{\arc\fat{(\motif Z-\motif Y)}}),
$$
where $d$ and $l$ are respectively the dimension of $X$ and the length of $\fat$.
Grouping together the $n+1$ terms with coefficient $h$, and  for each $i$,  the two terms with coefficient $g_i$,  this sum becomes
$$
\lef^{-dl}(\sum_{i=1}^ng_i\class{\arc\fat{\motif Y_i}} +  h\class{\arc\fat{\motif Z}}),
$$
since $\arc\fat{}$ acts on the \gr\ by Theorem~\ref{T:arcmot}, and since both $\motif Y_i-\motif Z$ and $\motif Z-\motif Y$ are disjoint from $\motif Y_i\cap \motif Z$. However, this is just the right hand side of \eqref{eq:addeq}, and so we are done.
\end{proof}


Let $s\colon\motif X\to\underline\grotmot$ be a formal invariant on a $\fld$-scheme $X$. Given an open $U\sub X$,   let $\restrict sU$ denote the restriction of $s$ to $\motif X\cap \func U$. It is easy to see that $\restrict sU$ is a formal invariant on $U$. Let $U_1,\dots,U_n$ be an open covering of $X$. For each non-empty subset $I\sub\{1,\dots,n\}$, let $U_I$ be the intersection of all $U_i$ with $i\in I$. We have the following local formula for the motivic integral (here we call a scheme \emph{equidimensional} if every non-empty open has the same dimension as the scheme):

\begin{theorem}\label{T:locmotint}
Let $s\colon\motif X\to\underline\grotmot$ be a formal invariant on an equidimensional $\fld$-scheme $X$,   let $\fat$ be a flat point, and let $U_1,\dots,U_n$ be an open covering of $X$. Then we have an equality 
\begin{equation}\label{eq:locmotint}
\motintspl Xs\fat=\sum_{\emptyset\neq I\sub \{1,\dots,n\}} (-1)^{\norm I}\motintspl {U_I}{\restrict s{U_I}}\fat.
\end{equation}
\end{theorem}
\begin{proof}
Given $g\in\grotmot$, one easily verifies that we have an equality of motives 
$$
\inverse{(\restrict s{U_I})}g=\inverse sg\cap \func U_I,
$$
for each $I\sub \{1,\dots,n\}$. Applying  the scissor relations to this, we get an identity
$$
\class {\inverse sg}=\class{\bigcup_{i=1}^n \inverse{(\restrict s{U_i})}g}=\sum_{\emptyset\neq I\sub \{1,\dots,n\}} (-1)^{\norm I}\class{\inverse{(\restrict s{U_I})}g}
$$
in $\grotmot$. Applying the arc morphism $\arc\fat\cdot$  as per  Theorem~\ref{T:arcmot}, we get
$$
\class {\arc\fat{(\inverse sg)}}=\sum_{\emptyset\neq I\sub \{1,\dots,n\}} (-1)^{\norm I}\class{\arc\fat{(\inverse{(\restrict s{U_I})}g)}}.
$$
Since by assumption all non-empty $U_I$ have the same dimension as $X$ (and, of course, the empty ones do not contribute), the result follows from   \eqref{eq:motint}.
\end{proof}

\subsection*{Relations among motivic series}
Given an element $\alpha\in \grotzero{\funcalg\fld}$, we   define
$$
\motintspl Xs\alpha:=\sum_{i=1}^s n_i\motintspl Xs{\fat_i}
$$
where $\alpha=n_1\class{\fat_1}+\dots+n_s\class{\fat_s}$ is the unique decomposition in  classes of fat points given by Proposition~\ref{P:zerosch}. We then   formally   extend this over $\pow{\grotzero{\funcalg\fld}}t$, by treating $t$ as a constant. In this sense, we get, for a closed germ $(Y,P)$, and a $\fld$-scheme $X$, the following identity of power series:
$$
\igumot XYP=\motintspl X{}{\op{Hilb}^{\text{mot}}_{(Y,P)}}.
$$

\section{Appendix: lattice rings}
Let $\categ M$ be a motivic site over an \acf\ $\fld$ and let $X$ be an $\fld$-scheme. By assumption, $\restrict{\categ M}X$ is a lattice, and so we can define its \emph{lattice group} $\latt{\categ M}X$ as the free Abelian group on $\categ M$-motives on $X$ modulo the scissor relations 
$$
\sym{\motif X}+\sym{\motif Y}-\sym{\motif X\cup \motif Y}-\sym{\motif 
X\cap \motif Y}
  $$
  for any two  $\categ M$-motives $\motif X$ and $\motif Y$ on $X$. In other words, same definition as for the \gr, but without the isomorphism relations. In particular, there is a natural linear map $\latt{\categ M}X\to \grot{\categ M}$. We will denote the class of a motif $\motif X$ again by $\class{\motif X}$. For each $n$, consider the embedding   $\restrict{\categ M}{X^n}\to \restrict{\categ M}{X^{n+1}}$ via the rule $\motif X\mapsto \motif X\times \func X$. One verifies that this induces a 
well-defined linear map $\Lambda_n:=\latt{\categ M}{X^n}\to \Lambda_{n+1}:=\latt{\categ M}{X^{n+1}}$, where $X^n$ is the $n$-fold Cartesian power of $X$. Moreover, the Cartesian product defines a multiplication $\Lambda_m\times\Lambda_n\to \Lambda_{m+n}$, for all $m,n$. Hence $\oplus_n\Lambda_n$ is a graded ring, called the \emph{graded lattice ring} of $\categ M$ on $X$, and denoted $\grlatt{\categ M}X$. It admits a natural ring \homo\ into $\grot{\categ M}$. 

We can now state a combinatorial property of the split motivic integral:

\begin{proposition}\label{P:lattmotintsp}
Over a $\fld$-scheme $X$ and a fat point $\fat$, we can define for each formal invariant $s\colon\motif X\to\underline\grotmot$ on $X$ and  each $g\in\latt{\funcinf\fld}X$, an integral $\motintsplrel Xs\fat g$, such that if $g$ is the class in $\latt{\funcinf\fld}X$ of a formal motif $\motif Y$ on $X$, then 
$$
\motintsplrel Xs\fat g= \motintsplrel X{s}\fat{\motif Y}.
$$
\end{proposition}
\begin{proof}
By definition, $g$ is a $\zet$-linear combination of classes of formal motives on $X$, say, of the form $g=n_1\class{\motif Y_1}+\dots+n_s\class{\motif Y_s}$. Define
$$
\motintsplrel Xs\fat g:=\sum_{i=1}^s n_i \motintsplrel Xs\fat{\motif Y_i}.
$$
To show that this is well-defined, we have to verify this only for scissor relations, that is to say, we have to show that
$$
\motintsplrel Xs\fat{\motif Y}+ \motintsplrel Xs\fat{\motif Y'}= \motintsplrel Xs\fat{\motif Y\cup\motif Y'}+ \motintsplrel Xs\fat{\motif Y\cap\motif Y'}
$$
for $\motif Y,\motif Y'$ formal motives on $X$. This is immediate from the easily proven identity 
of \ch\ functions 
$$
\tuple 1_{\motif Y}+\tuple 1_{\motif Y'}=\tuple 1_{\motif Y\cup\motif Y'}+\tuple 1_{\motif Y\cap\motif Y'}.
$$
\end{proof}

Using this, we can now show that the lattice rings are not very interesting invariants (and hence only by also taking isomorphism relations, do we get something significant):

\begin{corollary}\label{C:lattclass}
The natural map sending a formal motif on some Cartesian power of $X$ to its class in $\grlatt{\funcinf\fld}X$ is injective. 
\end{corollary}
\begin{proof}
Note that there are no non-trivial relations among classes of motives on different Cartesian powers of $X$, so after replacing $X$ by one of its Cartesian powers, we may reduce to the case that   $\motif X$ and $\motif Y$ are   formal motives on $X$ having the same class in $\latt{\funcinf\fld}X$. By Proposition~\ref{P:lattmotintsp}, we have
\begin{equation}\label{eq:relinteq}
\motintsplrel Xs\fat {\motif X}=\motintsplrel Xs\fat {\motif Y}
\end{equation}
for any formal invariant $s$ on $X$ and any fat point $\fat$. Take $s:=\tuple 1_{\motif X}$. The left hand side of \eqref{eq:relinteq} is equal to $\lef^{-dl}\class{\arc\fat{\motif X}}$ (as an element in $\grotmot $), whereas the right hand side is equal to $\lef^{-dl}\class{\arc\fat{(\motif X\cap\motif Y)}}$, where $d$ and $l$ are respectively  the dimension of $X$ and the length of $\fat$. Using that $\arc\fat{}$ preserves scissor relations, we get $\arc\fat{(\motif X-\motif Y)}=0$. Hence $\motif X-\motif Y=\emptyset$ by Lemma~\ref{L:arcempt}, showing that $\motif Y\sub\motif X$. Replacing the role of $\motif X$ and $\motif Y$ then proves the other inclusion.  
\end{proof}

\providecommand{\bysame}{\leavevmode\hbox to3em{\hrulefill}\thinspace}
\providecommand{\MR}{\relax\ifhmode\unskip\space\fi MR }
\providecommand{\MRhref}[2]{%
  \href{http://www.ams.org/mathscinet-getitem?mr=#1}{#2}
}
\providecommand{\href}[2]{#2}

\end{document}